\documentclass[12pt,a4paper,reqno]{amsart}
\usepackage[utf8]{inputenc}
\usepackage[T1]{fontenc}
\usepackage{indentfirst}
\usepackage[english]{babel}
\usepackage{mathpazo}
\usepackage{enumitem}
\usepackage{csquotes}
\usepackage{amsmath,amsfonts,amssymb,amsopn,amscd,amsthm}
\usepackage[alphabetic]{amsrefs}
\usepackage{mathrsfs,shuffle,bbm}
\usepackage{stmaryrd}
\usepackage{graphicx}
\usepackage[dvipsnames]{xcolor}
\usepackage[colorlinks=true,citecolor=DarkOrchid,linkcolor=NavyBlue]{hyperref}
\usepackage{pgfplots}
\usepackage{tikz}
\usetikzlibrary{patterns,arrows.meta}
\pgfplotsset{compat=1.17}
\usepackage[a4paper,twoside,margin=0.8in]{geometry}
    \newtheorem*{maintheorem}{Main Theorem}
    \newtheorem{definition}{Definition}
    \newtheorem{lemma}[definition]{Lemma}
    
    \newtheorem{theorem}[definition]{Theorem}
    \newtheorem{proposition}[definition]{Proposition}
    \newtheorem{corollary}[definition]{Corollary}
    \newtheorem*{problem}{Problem}
    \newtheorem*{conjecture}{Conjecture}
    \theoremstyle{remark}
    \newtheorem{example}[definition]{Example}
    \newtheorem{remark}[definition]{Remark}

\newcommand{\N}{\mathbb{N}}
\newcommand{\Z}{\mathbb{Z}}

\newcommand{\R}{\mathbb{R}}
\newcommand{\C}{\mathbb{C}}
\newcommand{\Disk}{\mathbb{D}}

\newcommand{\E}{\mathrm{e}}
\newcommand{\I}{\mathrm{i}}
\newcommand{\esper}{\mathbb{E}}
\newcommand{\var}{\mathrm{var}}
\newcommand{\cov}{\mathrm{cov}}
\newcommand{\proba}{\mathbb{P}}

\newcommand{\leb}{\mathscr{L}}
\newcommand{\card}{\mathrm{card}}

\newcommand{\eps}{\varepsilon}
\newcommand{\tr}{\mathrm{tr}}
\newcommand{\id}{\mathrm{id}}
\newcommand{\wt}{\mathrm{wt}}
\newcommand{\mwt}{\overline{\mathrm{wt}}}
\newcommand{\ST}{\mathrm{ST}}
\newcommand{\diag}{\mathrm{diag}}
\newcommand{\Gspa}{\mathscr{G}}
\newcommand{\Gset}{\mathbf{G}}

\newcommand{\Hset}{\mathbf{L}\mathbf{H}}
\newcommand{\obs}{\mathcal{O}_\Gset}
\newcommand{\lle}{\left[\!\left[} 
\newcommand{\rre}{\right]\!\right]}
\newcommand{\sym}{\mathfrak{S}}
\newcommand{\pym}{\mathfrak{P}}
\newcommand{\dir}{_{\mathrm{d}}}
\newcommand{\Spec}{\mathrm{Spec}}
\newcommand{\fredholm}{\mathrm{det}_2}
\newcommand{\ffredholm}{\mathrm{det}_3}
\newcommand{\edgegraph}{\begin{tikzpicture}[baseline=-1mm,scale=0.5]
\foreach \x in {(0,0),(1,0)}
\fill \x circle (4pt);
\draw (0,0) -- (1,0);
\end{tikzpicture}}
\newcommand{\DD}[1]{\,d\hspace{-0.3mm}{#1}}
\newcommand{\scal}[2]{\left\langle #1\vphantom{#2}\,\right |\left.#2 \vphantom{#1}\right\rangle}
\newcommand{\comment}[1]{}
\newcommand{\join}{\bowtie}
\renewcommand{\Re}{\mathrm{Re}}

\setlist[enumerate]{itemsep=10pt,topsep=10pt}
\setlist[itemize]{itemsep=5pt,topsep=5pt}

\title{A central limit theorem for singular graphons}
\author{Pierre-Loïc Méliot}
\date{\today}

\begin{document}

\begin{abstract}
We associate to a graphon $\gamma$ the sequence of $W$-random graphs $(G_n(\gamma))_{n \geq 1}$. We say that the graphon is \emph{singular} if, for any finite graph $F$, the homomorphism density $t(F,G_n(\gamma))$ has a variance of order $O(n^{-2})$. This behavior is singular because generically, the density of a fixed finite graph $F$ in a $W$-random graph has a variance of order $O(n^{-1})$. We conjecture that the only singular graphons are the constant graphons $\gamma_p$ with $p \in [0,1]$, corresponding to the Erd\H{o}s--Rényi random graphs $G(n,p)$. In this paper, we investigate the general properties of the singular graphons, and we show that they share many properties with the Erd\H{o}s--Rényi random graphs. In particular, if $\gamma$ is a singular graphon, then the scaled densities $n(t(F,G_n(\gamma))-\esper[t(F,G_n(\gamma))])$ converge in joint distribution. This generalises the central limit theorem satisfied by the Erd\H{o}s--Rényi random graphs $G(n,p)$; however, the limiting distribution might be non-Gaussian if the conjecture does not hold. We also establish an equation satisfied by the characteristic polynomial of the Laplacian of the graph $G_n(\gamma)$ associated to a singular graphon; this opens the way to a spectral approach of the conjecture.
\end{abstract}

\maketitle

\hrule
\tableofcontents
\hrule

\clearpage

\section{The Gaussian moduli space of graphons}
Throughout the paper, \emph{unless stated explicitly}, by \emph{graph} we mean a finite, simple and unoriented graph $G=(V_G,E_G)$, with $V_G$ finite set, and $E_G$ subset of the set of pairs $\{x,y\}$ with $x,y \in V$ and $x \neq y$.
The \emph{size} of a graph is its number of vertices $k=|V_G|$.

\subsection{Graph functions and graphons}
The space of graphons is a compact metric space which enables the parametrisation of sequences of graphs which are convergent with respect to the notion of left-convergence. We follow here \cites{LS06,BCLSV08} for the presentation of the main properties of this space, which can be defined as a quotient of the set of graph functions; see also the monography \cite{Lov12}. A \emph{graph function} is a Borel measurable function $g : [0,1]^2 \to [0,1]$, which we consider as an element of the Lebesgue space $\leb^\infty([0,1]^2,\DD{x}\DD{y})$ (hence, it is defined up to a set with zero Lebesgue measure), and which is symmetric:
$$g(x,y) = g(y,x) \quad \text{almost everywhere}.$$
We endow $\leb^\infty([0,1]^2)$ with the norm
$$\|w\|_\oblong = \sup_{S,T \subset [0,1]} \left|\int_{S \times T} w(x,y)\DD{x}\DD{y}\right|,$$
where the supremum runs over pairs of measurable subsets of $[0,1]$. On the other hand, we call \emph{Lebesgue isomorphism} a bijective measurable map $\sigma : [0,1] \to [0,1]$ which preserves the Lebesgue measure. The Lebesgue isomorphisms act on the right of $\leb^\infty([0,1]^2)$ by:
$$g^\sigma(x,y) = g(\sigma(x),\sigma(y)).$$
The \emph{cut-metric} between two graph functions $g_1$ and $g_2$ is
$$\delta_\oblong(g_1,g_2) = \inf_{\sigma} \|g_1^\sigma-g_2\|_\oblong,$$
where the infimum runs over Lebesgue isomorphisms $\sigma : [0,1] \to [0,1]$.
Two graph functions $g_1$ and $g_2$ are considered equivalent if $\delta_\oblong(g_1,g_2)=0$, and we call \emph{graphon} an equivalence class $\gamma = [g]$ of graph functions for this relation.  We denote $\Gspa$ the set of graphons; the cut-metric induces a distance on this space, which makes it a compact metric space \cite{LS07}*{Theorem 5.1}.\medskip

Given two graphs $F$ and $G$, we call morphism from $F$ to $G$ a map $\phi : V_F \to V_G$ such that if $\{x,y\} \in E_F$, then $\{\phi(x),\phi(y)\} \in E_G$. The \emph{homomorphism density} of $F$ in $G$ is defined by
$$t(F,G) = \frac{|\hom(F,G)|}{|V_G|^{|V_F|}},$$
where $\hom(F,G)$ denotes the set of morphisms from $F$ to $G$. On the other hand, consider a finite graph $F$ on $k$ vertices and a graphon $\gamma$ represented by a graph function $g$. We label the vertices of $F$ by the integers in $\lle1,k\rre$, and we define the \emph{density} of $F$ in $\gamma$ by
$$t(F,\gamma) = \int_{[0,1]^k} \left(\prod_{\{i,j\} \in E_F} g(x_i,x_j)\right)\DD{x_1}\cdots \DD{x_k};$$
this quantity does not depend on the choice of a representative $g$ of the graphon $\gamma$. Given a graph $G$ on $n \geq 1$ vertices, by drawing its adjacency matrix as a step function $g_G : [0,1]^2 \to \{0,1\}$ which is constant on each small square $[\frac{i-1}{n},\frac{i}{n}] \times [\frac{j-1}{n},\frac{j}{n}]$, one obtains a graph function $g_G$ and a graphon $\gamma_G = [g_G]$ (see Figure \ref{fig:graph_function}). Then, the two definitions above correspond: $t(F,\gamma_G) = t(F,G)$ for any graphs $F,G$.

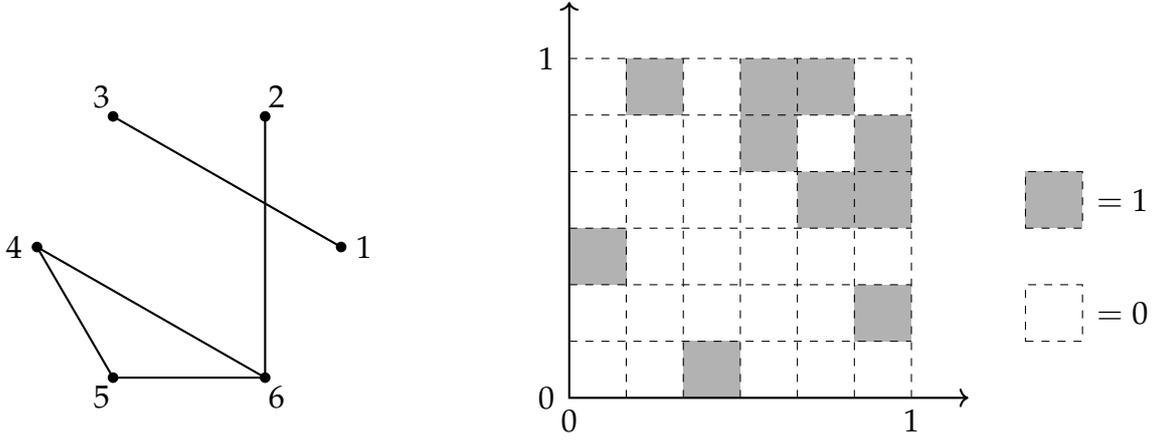
\begin{figure}[ht]
\begin{center}        
\begin{tikzpicture}[scale=1]
\foreach \x in {0,60,120,180,240,300}
\fill (\x:2) circle (2pt);
\draw [thick] (0:2) -- (120:2);
\draw (0:2.3) node {$1$};
\draw (60:2.3) node {$2$};
\draw (120:2.3) node {$3$};
\draw (180:2.3) node {$4$};
\draw (240:2.3) node {$5$};
\draw (300:2.3) node {$6$};
\draw [thick] (60:2) -- (300:2) -- (240:2) -- (180:2) -- (300:2);
\begin{scope}[shift={(5,-2)},scale=0.75]
\fill [white!70!black] (0,2) rectangle (1,3);
\fill [white!70!black] (2,0) rectangle (3,1);
\fill [white!70!black] (1,5) rectangle (2,6);
\fill [white!70!black] (5,1) rectangle (6,2);
\fill [white!70!black] (4,3) -- (6,3) -- (6,5) -- (5,5) -- (5,4) -- (4,4) -- (4,3);
\fill [white!70!black] (3,4) -- (4,4) -- (4,5) -- (5,5) -- (5,6) -- (3,6) -- (3,4); 
\foreach \x in {1,2,3,4,5,6}
{\draw [dashed] (\x,0) -- (\x,6); \draw [dashed] (0,\x) -- (6,\x);}
\draw (0,-0.4) node {$0$};
\draw (6,-0.4) node {$1$};
\draw (-0.4,0) node {$0$};
\draw (-0.4,6) node {$1$};
\fill [white!70!black] (8,3) rectangle (9,4);
\draw [dashed] (8,3) rectangle (9,4);
\draw (9.7,3.5) node {$=1$};
\draw [dashed] (8,1) rectangle (9,2);
\draw (9.7,1.5) node {$=0$};
\draw [thick,<->] (7,0) -- (0,0) -- (0,7);
\end{scope}
\end{tikzpicture}
\caption{The graph function $g_G$ associated to a graph $G$.}\label{fig:graph_function}
\end{center}
\end{figure}

\begin{example}
Suppose that $$F = \begin{tikzpicture}[scale=0.7,baseline=3mm]
\foreach \x in {(0,0),(0,1),(1,0),(1,1)}
\fill \x circle (3pt);
\draw (0,0) rectangle (1,1);
\draw (-0.3,-0.3) node {$1$};
\draw (-0.3,1.3) node {$4$};
\draw (1.3,1.3) node {$3$};
\draw (1.3,-0.3) node {$2$};
\end{tikzpicture}\,.$$
The corresponding density in graphons is $$\int_{[0,1]^4} g(x_1,x_2)g(x_2,x_3)g(x_3,x_4)g(x_4,x_1)\DD{x_1}\DD{x_2}\DD{x_3}\DD{x_4}.$$ Note that this quantity does not depend on the labelling of the vertices of $F$; more generally, the function $t(F,\cdot)$ only depends on the isomorphism type of the graph $F$, so for instance it makes sense to simply write $t(\begin{tikzpicture}[baseline=0.5mm,scale=0.3]
\foreach \x in {(0,0),(0,1),(1,0),(1,1)}
\fill \x circle (5pt);
\draw (0,0) rectangle (1,1);
\end{tikzpicture},\cdot)$. If $G$ is another finite graph on $n$ vertices, then $t(\begin{tikzpicture}[baseline=0.5mm,scale=0.3]
\foreach \x in {(0,0),(0,1),(1,0),(1,1)}
\fill \x circle (5pt);
\draw (0,0) rectangle (1,1);
\end{tikzpicture},G)$ is $\frac{1}{n^4}$ times the number of labelled squares that are subgraphs of $G$.
\end{example}

It can be showed that the topology induced by $\delta_\oblong$ on $\Gspa$ is equivalent to the topology of convergence of all the observables $t(F,\cdot)$: a sequence of graphons $(\gamma_n)_{n \in \N}$ converges with respect to $\delta_\oblong$ to a graphon $\gamma$ if and only if $t(F,\gamma_n)$ converges to $t(F,\gamma)$ for any finite graph $F$ \cite{BCLSV08}*{Theorem 3.8}. In terms of graphs, a sequence of graphs $(G_n)_{n\in \N}$ is said \emph{left-convergent} if $t(F,G_n)$ admits a limit for any finite graph $F$. This is equivalent to the existence of a graphon $\gamma$ such that $\gamma_{G_n} \to \gamma$ in $\Gspa$ \cite{LS06}*{Theorem 2.2}. Therefore, the space of graphons $\Gspa$ is the adequate space in order to understand the notion of left-convergence for sequences $(G_n)_{n \in \N}$ of dense graphs.
\medskip

\subsection{Random graphs associated to graphons}
We have explained above how to associate to any finite simple graph $G$ a graphon $\gamma_G$. An important tool in the theory of graphons is the following reverse construction: one can associate to any graphon $\gamma$ a sequence $(G_n(\gamma))_{n \geq 1}$ of random graphs with $|V_{G_n(\gamma)}|=n$ for any $n \geq 1$. Consider two families $(X_i)_{i \geq 1}$ and $(U_{ij})_{j>i\geq 1}$ of independent random variables uniformly distributed in $[0,1]$. We also fix a representative $g$ of the graphon $\gamma$. The \emph{$W$-random graph} $G_n(\gamma)$ is the graph with vertex set $\lle 1,n\rre$, and with an edge between $i$ and $j$ if and only if
$$U_{ij} \leq g(X_i,X_j).$$
Hence, conditionally to the random vector $(X_1,\ldots,X_n)$, the entries $1_{(i\sim j)}$ of the adjacency matrix of $G_n(\gamma)$ are independent Bernoulli variables with parameters $g(X_i,X_j)$. It is easily seen that the law of the random graph $G_n(\gamma)$ on the set of vertices $V=\lle 1,n\rre$ does not depend on the choice of a representative $g$ of the graphon $\gamma$. Besides, an immediate computation of the two first moments shows that for any graphon $\gamma$ and any finite graph $F$, $t(F,G_n(\gamma))$ converges in probability towards $t(F,\gamma)$ (see Section \ref{sec:moments}). Therefore, $(G_n(\gamma))_{n\geq 1}$ converges in probability to $\gamma$ in the space of graphons \cite{LS06}*{Corollary 2.6}; this shows in particular that the finite graphs form a dense subset of $\Gspa$.\medskip

In \cite{FMN20}, the fluctuations of the random variables $t(F,G_n(\gamma))$ have been studied for any graphon $\gamma$ and any fixed observable $t(F,\cdot)$. In order to state these results, it is convenient to introduce the combinatorial algebra $\obs$ of finite (unlabelled) graphs: it is the $\R$-algebra with a countable basis indexed by the set $\Gset$ of isomorphism types of finite graphs, and where the product $F_1F_2$ of two graphs $F_1$ and $F_2$ is their disjoint union $F_1 \sqcup F_2$. One evaluates an element of $\obs$ on a graphon or on a graph by means of the formula:
$$\left(\sum_{F} c_F\,F\right)(\gamma) = \sum_{F} c_F\,t(F,\gamma).$$
This rule yields a morphism of real algebras from $\obs$ to the space of continuous functions $\mathscr{C}(\Gspa,\R)$. The image of this morphism is dense by the Stone--Weierstrass theorem. Given two non-empty finite graphs $F_1$ and $F_2$, let us fix an arbitrary labelling of their vertices by the sets of integers $\lle 1,k_1\rre$ and $\lle 1',k_2'\rre$. For $i_1 \in \lle 1,k_1\rre$ and $i_2 \in \lle 1,k_2\rre$, we denote $(F_1\join F_2)(i_1,i_2)$ the (isomorphism type of) graph obtained by identifying the vertex $i_1$ in $F_1$ with the vertex $i_2'$ in $F_2$. Thus, the graph $(F_1\join F_2)(i_1,i_2)$ has $k_1+k_2-1$ vertices, and its number of edges is $|E_{F_1}|+|E_{F_2}|$. We then set:
$$\kappa_2(F_1,F_2) = \sum_{\substack{1\leq i_1\leq k_1\\1\leq i_2 \leq k_2}} \big((F_1\join F_2)(i_1,i_2) - F_1\,F_2\big),$$
which is an element of $\obs$. 

\begin{example}
 If $F = F_1 = F_2 = \begin{tikzpicture}[baseline=-1mm,scale=0.5]
\foreach \x in {(0,0),(1,0),(2,0)}
\fill \x circle (4pt);
\draw (0,0) -- (2,0);
\end{tikzpicture}$\,, then
$$\kappa_2(F,F) = 4 \,\,\begin{tikzpicture}[baseline=-1mm,scale=0.5]
\foreach \x in {(0,0),(1,0),(2,0),(3,0),(4,0)}
\fill \x circle (4pt);
\draw (0,0) -- (4,0);
\end{tikzpicture} \,+\, 4\,\,\begin{tikzpicture}[baseline=-1mm,scale=0.5]
\foreach \x in {(0,0),(1,0),(2,0),(2,1),(2,-1)}
\fill \x circle (4pt);
\draw (0,0) -- (2,0) -- (2,1);
\draw (2,0) -- (2,-1);
\end{tikzpicture} \,+\, \begin{tikzpicture}[baseline=-1mm,scale=0.5]
\foreach \x in {(0,0),(1,0),(2,0),(1,1),(1,-1)}
\fill \x circle (4pt);
\draw (0,0) -- (2,0);
\draw (1,1) -- (1,-1);
\end{tikzpicture}\,- \,9\,\, \begin{tikzpicture}[baseline=1.5mm,scale=0.5]
\foreach \x in {(0,0),(1,0),(2,0),(0,1),(1,1),(2,1)}
\fill \x circle (4pt);
\draw (0,0) -- (2,0);
\draw (0,1) -- (2,1);
\end{tikzpicture}\,.$$
\end{example} 
\medskip

\noindent One of the main results from \cite{FMN20}*{Theorem 8} on the fluctuations of graphon models is:
\begin{theorem}\label{thm:fmn}
For any finite graphs $F_1$ and $F_2$ with sizes $k_1,k_2\geq 1$, and any graphon $\gamma \in \Gspa$,
$$\lim_{n \to \infty} n\,\cov(t(F_1,G_n(\gamma)),t(F_2,G_n(\gamma))) = \kappa_2(F_1,F_2)(\gamma).$$
The random variable
$$Y_n(F,\gamma) = \sqrt{n}\big(t(F,G_n(\gamma))-\esper[t(F,G_n(\gamma))]\big)$$
converges in distribution to the normal law $\mathcal{N}(0,\kappa_2(F,F)(\gamma))$.
\end{theorem}

The results from \cite{FMN20} are in fact much more precise: they belong to the framework of \emph{mod-Gaussian convergence} developed in \cites{FMN16,FMN19}, and therefore, the central limit theorem stated above is refined by a Berry--Esseen upper bound on the Kolmogorov distance between $Y_n(F,\gamma)$ and $\mathcal{N}(0,1)$; a moderate deviation estimate; a local limit theorem; and a concentration inequality. We refer in particular to \cite{FMN20}*{Theorems 8, 9 and 21}. In \cite{FMN20}*{Section 6}, the abstract notion of \emph{mod-Gaussian moduli space} was introduced in order to describe these results. We have a pair $(\Gspa,\obs)$, where:
\begin{itemize}
    \item $\Gspa$ is a (infinite-dimensional) compact metric space, and for every element $\gamma \in \Gspa$, we have a sequence of random elements $G_n(\gamma)$ with $G_n(\gamma) \to_{\proba} \gamma$.
    \item $\obs$ is a real algebra of observables of the elements of $\Gspa$, such that the topology of $\Gspa$ is the weak topology induced by $\obs$.
    \item there exists a combinatorial (countable) basis $\Gset$ of the algebra $\obs$, and a map $\kappa_2 : \Gset \times \Gset \to \obs$, such that for any elements $f_1,f_2 \in \Gset$, and any parameter $\gamma \in \Gspa$,
    $$\cov(f_1(G_n(\gamma)),f_2(G_n(\gamma))) = \frac{\kappa_2(f_1,f_2)(\gamma)}{n} + o\!\left(\frac{1}{n}\right).$$
    Moreover, if $\kappa_2(f,f)(\gamma) \neq 0$, then the random variables
    $$Y_n(f,\gamma)=\sqrt{n}\big(f(G_n(\gamma))-\esper[f(G_n(\gamma))]\big)$$ 
    are asymptotically normal, and we even have the estimates from the theory of mod-Gaussian sequences.
\end{itemize}
A similar framework can be constructed for models of random permutations or random integer partitions, see again \cite{FMN20}. More recently, an analogous construction has been performed for random measure metric spaces; see \cite{DCM21}.\bigskip

In this framework of mod-Gaussian moduli spaces, a natural question is: what are the fluctuations of the random observables $f(G_n(\gamma))$ in the singular case where $\kappa_2(f,f)(\gamma)=0$? For the $W$-random graphs attached to graphons, this means that we consider a graphon $\gamma$ and a finite graph $F$ such that
\begin{equation}
\kappa_2(F,F)(\gamma)=0,\tag{S1}\label{eq:singular_1}
\end{equation} 
or, equivalently, such that
\begin{equation}
\var(t(F,G_n(\gamma))) = O(n^{-2})\quad \text{(instead of } O(n^{-1})). \tag{S2}\label{eq:singular_2}
\end{equation}
Indeed, we shall recall in Section \ref{sec:moments} that the moments of the random densities $t(F,G_n(\gamma))$ can always be expanded in negative powers of $n$; therefore, if $\kappa_2(F,F)(\gamma)$ vanishes, then the leading term in the expansion of the variance is a $O(n^{-2})$. 
\medskip

\subsection{Singular graphons and Erd\H{o}s--Rényi random graphs}
Given a graphon $\gamma$, the vanishing of the asymptotic covariance coefficient $\kappa_2(F_1,F_2)(\gamma)$ is equivalent to the following equation:
\begin{equation}
t(F_1\,F_2,\gamma) = \frac{1}{k_1k_2}\,\sum_{\substack{1\leq i_1 \leq k_1\\ 1\leq i_2 \leq k_2}} t((F_1\join F_2)(i_1,i_2),\gamma)\tag{S3}\label{eq:singular_3}
\end{equation}
if $F_1$ and $F_2$ have respectively $k_1$ and $k_2$ vertices. As Equation \eqref{eq:singular_3} imposes a condition of finite-dimensional nature on elements of the infinite-dimensional space $\Gspa$, it is reasonable to believe that when $F_1$ and $F_2$ are fixed, there are many graphons $\gamma$ for which the identity holds. A more interesting question is whether Equation \eqref{eq:singular_3} can hold \emph{simultaneously for every finite graphs $F_1$ and $F_2$}. This question leads to the following definition:

\begin{definition}\label{def:singular}
We call a graphon $\gamma$ \emph{globally singular}, or in short \emph{singular} if, for any finite graphs $F_1$ and $F_2$ on $k_1$ and $k_2$ vertices,
Equation \eqref{eq:singular_3} holds.
\end{definition}
\noindent Equivalently, a graphon $\gamma$ is singular when Equation \eqref{eq:singular_1} or \eqref{eq:singular_2} holds simultaneously for every finite graph $F$. This means that the fluctuations of the random densities $t(F,G_n(\gamma))$ are all of order $O(n^{-1})$, instead of $O(n^{-1/2})$ (which is the generic case). 
\medskip

\begin{example}\label{ex:erdos_renyi}
For $p \in [0,1]$, denote $\gamma_p$ the graphon of the graph function $[0,1]^2 \to [0,1]$ which is constant equal to $p$. The $W$-random graph $G_n(\gamma_p)$ is then the Erd\H{o}s--Rényi random graph $G(n,p)$ \cite{ER60}: all the edges $\{i,j\}$ of $G(n,p)$ are independent, and each edge $\{i,j\}$ with $1 \leq i < j \leq n$ appears with probability $p$. The densities of the graphon $\gamma_p$ are given by:
$$t(F,\gamma_p) = p^{|E_F|}$$
for any finite graph $F$. Then, Equation \eqref{eq:singular_3} trivially holds, since $|E_{(F_1\join F_2)(i_1,i_2)}| = |E_{F_1}| + |E_{F_2}|$ for any join of the two graphs $F_1$ and $F_2$. Therefore, the constant graphons $\gamma_p$ are singular. In this case, one can show that 
$$\var(t(F,G(n,p))) = \frac{2\,|E_F|^2\,p^{2|E_F|-1}\,(1-p)}{n^2} + o\!\left(\frac{1}{n^2}\right),$$
and that the random vector 
\begin{equation}
\big(n(t(F,G(n,p)) - \esper[t(F,G(n,p))])\big)_{F \in \Gset} \label{eq:fluctuations}\tag{F}
\end{equation} 
converges towards a Gaussian distribution, although the fluctuations of the random variables $t(F,G(n,p))$ are not of the same size as in the generic case (we have rescaled the fluctuations by a factor $n$ instead of $\sqrt{n}$). We refer to \cites{Jan88,Now89} for the first proofs of these results; see also \cite{FMN16}*{Section 10}, where the mod-Gaussian convergence of the subgraph counts of the Erd\H{o}s--Rényi random graphs has been established. Let us notice however that in the framework of mod-Gaussian moduli spaces, it is not always the case that after a different renormalisation, a singular parameter provides observables whose fluctuations are asymptotically normal. In the setting of random metric spaces, an explicit counter-example has been described in \cite{DCM21}*{Sections 5 and 6}.
\end{example}
\medskip

After the writing of \cite{FMN20}, its authors tried without success to find other examples of singular graphons. So, they conjectured the following:

\begin{conjecture}
If $\gamma$ is a singular graphon, then there exists a parameter $p \in [0,1]$ such that $\gamma=\gamma_p$. This parameter $p$ is the edge density
$$p = t(\edgegraph,\,\gamma).$$
\end{conjecture}

\noindent Let us explain briefly why this conjecture is really difficult to prove. The origin of Formula \eqref{eq:singular_3} will be given in Section \ref{sec:moments}. Let us consider a singular graphon $\gamma$ with edge density $p$, and see what is implied readily by this identity. Since
$$\kappa_2(\edgegraph\,,\,\edgegraph) = 4\left(\begin{tikzpicture}[baseline=-1mm,scale=0.5]
\foreach \x in {(0,0),(1,0),(2,0)}
\fill \x circle (4pt);
\draw (0,0) -- (1,0) -- (2,0);
\end{tikzpicture} \,-\,\begin{tikzpicture}[baseline=2mm,scale=0.5]
\foreach \x in {(0,0),(1,0),(0,1),(1,1)}
\fill \x circle (4pt);
\draw (0,0) -- (1,0);
\draw (0,1) -- (1,1);
\end{tikzpicture}\,\right),$$
we obtain immediately $t(\begin{tikzpicture}[baseline=-1mm,scale=0.5]
\foreach \x in {(0,0),(1,0),(2,0)}
\fill \x circle (4pt);
\draw (0,0) -- (1,0) -- (2,0);
\end{tikzpicture},\,\gamma)=(t(\edgegraph),\gamma)^2 = p^2$. We can then try to prove that $t(F,\gamma)=p^{|E_F|}$ for any (connected) finite graph $F$; indeed, these identities characterise the constant graphon $\gamma_p$. For the trees with $4$ vertices, the only equation that seems to come from \eqref{eq:singular_3} is 
$$ 0 = \kappa_2(\edgegraph\,,\,\begin{tikzpicture}[baseline=-1mm,scale=0.5]
\foreach \x in {(0,0),(1,0),(2,0)}
\fill \x circle (4pt);
\draw (0,0) -- (2,0);
\end{tikzpicture})(\gamma) = 4\,t(\begin{tikzpicture}[baseline=-1mm,scale=0.5]
\foreach \x in {(0,0),(1,0),(2,0),(3,0)}
\fill \x circle (4pt);
\draw (0,0) -- (3,0);
\end{tikzpicture},\,\gamma) + 2\, t\!\left(\begin{tikzpicture}[baseline=2mm,scale=0.5]
\foreach \x in {(0,0),(1,0),(2,0),(1,1)}
\fill \x circle (4pt);
\draw (0,0) -- (2,0);
\draw (1,0) -- (1,1);
\end{tikzpicture},\,\gamma\right)- 6p^3.$$
This does not seem to determine the value of the density in $\gamma$ of the two trees on $4$ vertices drawn above; however, we shall provide later a direct argument which proves that $t(T,\gamma)=p^{|E_T|}$ for any tree $T$ and any singular graphon $\gamma$ with edge density $p$ (see Corollary \ref{cor:tree_densities}). Unfortunately, we then have no information on the density of the triangle or of the square in $\gamma$, because these graphs with cycles cannot be obtained from simpler graphs by means of the join operation $\join$\,.
A famous result due to Chung, Graham and Wilson shows that a graphon $\gamma$ is a constant graphon $\gamma_p$ if and only if 
$$t(\begin{tikzpicture}[baseline=0.5mm,scale=0.3]
\foreach \x in {(0,0),(0,1),(1,0),(1,1)}
\fill \x circle (5pt);
\draw (0,0) rectangle (1,1);
\end{tikzpicture},\,\gamma) = (t(\edgegraph,\,\gamma))^4;$$ 
see \cite{CGW89} and \cite{Lov12}*{Section 11.8.1}. Unfortunately, our Conditions \eqref{eq:singular_1}-\eqref{eq:singular_3} only relate the values of observables $t(F,\gamma)$ for finite graphs $F$ which can be obtained from one another by using the join operation $\join$; and the square $\begin{tikzpicture}[baseline=0.5mm,scale=0.3]
\foreach \x in {(0,0),(0,1),(1,0),(1,1)}
\fill \x circle (5pt);
\draw (0,0) rectangle (1,1);
\end{tikzpicture}$ cannot be expressed as a join of simpler graphs.\medskip

A question which is related to our Conjecture is whether the vector of correctly rescaled fluctuations \eqref{eq:fluctuations} with $G_n(\gamma)$ instead of $G(n,p)$ has a limiting distribution for any singular graphon $\gamma$. The answer to this question is positive, and this is the main result of this article:

\begin{maintheorem}
Let $\gamma$ be a singular graphon. Then, the random vector
$$\big(n(t(F,G_n(\gamma)) - \esper[t(F,G_n(\gamma))])\big)_{F \in \Gset}$$
converges in joint finite-dimensional distributions. The limiting distribution is determined by its joint moments.
\end{maintheorem} 

\noindent Obviously, this is a very partial result: we do not know precisely what is the limiting distribution if $\gamma$ is not a constant graphon, and we do not know if $\gamma$ can be something else than a constant graphon. Actually, we shall see that given a singular graphon $\gamma$, if $n(t(\edgegraph,G_n(\gamma))-\esper[t(\edgegraph,G_n(\gamma))])$ (scaled random edge density) is asymptotically normal, then $\gamma$ is a constant graphon $\gamma_p$.  With the techniques presented in this article, we are unfortunately not able to prove that the scaled random edge density is always asymptotically normal for a singular graphon.

\begin{remark}
Throughout the article, many Propositions and Theorems will start with an assumption on the edge density $p \in [0,1]$ of a (singular) graphon. We shall always assume (and omit to recall) that $p>0$: indeed, the case $p=0$ is trivial, because the only graphon with vanishing edge density is the constant graphon $\gamma_0$.
\end{remark}
\medskip

\subsection{Techniques of proof and outline of the paper}\label{sub:outline} 
Before describing the proof of our Main Theorem, we should explain that the models of $W$-random graphs belong to the class of random models \emph{with a fast decay of cumulants}. Let us fix an arbitrary graphon $\gamma \in \Gspa$, and consider the sequence of centered random densities $(Z_n)_{n \in \N}$ with $Z_n = t(F,G_n(\gamma))-\esper[t(F,G_n(\gamma))]$, $F$ being a fixed motive. Theorem \ref{thm:fmn} implies in particular that the second cumulant of $Z_n$ satisfies:
$$\kappa^{(2)}(Z_n) = \var(Z_n) = O(n^{-1}).$$
Consider more generally the $r$-th cumulant of $Z_n$, which is up to a factor $\frac{1}{r!}$ the coefficient of $t^r$ in the log-Laplace transform $\log (\esper[\E^{tZ_n}])$. Then, as we shall recall in Section \ref{sec:moments}, the higher order cumulants with $r \geq 3$ are all $O(n^{-2})$ when $n$ goes to infinity (this is why we have a central limit theorem for $\sqrt{n}\,Z_n$ for any graphon $\gamma \in \Gspa$). Moreover, their decay to $0$ is faster when $r$ is large: thus,
\begin{equation}
|\kappa^{(r)}(Z_n)| = O\!\left(n^{1-r}\right), \tag{FD}\label{eq:FD}
\end{equation}
with a constant in the $O(\cdot)$ which depends on $r$ and on the fixed graph $F$. This faster decay of the higher order cumulants is not at all required in order to ensure the asymptotic normality, but it occurs quite frequently. For instance, in \cite{Jan88}, \cite{Col03} and \cite{Sni06}, the fast decay of cumulants has been established for observables of models respectively of random graphs, of random matrices and of random characters of symmetric groups. Regarding the random characters of symmetric groups, the observables are related to the Plancherel and Schur--Weyl measures on integer partitions, and the fast decay of cumulants has been generalised to other similar models in \cites{FM12,Mel12,DS19,MS20}. The probabilistic consequences of these estimates on cumulants have been explored in \cites{FMN16,FMN19}, in the framework of mod-$\phi$ and mod-Gaussian convergence. Now, a general question which can be asked in this setting of fast decay of cumulants is the following. Given random models which yield centered random variables $(Z_n)_{n \in \N}$ which satisfy the inequality \eqref{eq:FD}, let us consider one specific singular model for which the variance is smaller than usual:
$$\kappa^{(2)}(Z_n) = O\!\left(n^{-2}\right)\quad \text{(instead of }O(n^{-1})).$$
\begin{problem}
Given a singular model in a class of model with fast decay of cumulants, do we also have smaller higher order cumulants, with $$|\kappa^{(r)}(Z_n)| = O(n^{-r})\quad \text{(instead of } O(n^{1-r}))$$ for any $r \geq 3$?
\end{problem}
\noindent This is an important question, because if the answer is positive, then one might be able to prove a convergence in distribution
$$n\,Z_n \quad \big(\text{instead of }\sqrt{n}\,Z_n\big)\rightharpoonup_{n \to \infty} \text{some limiting law}.$$ 
This technique has already been used in the setting of random metric spaces, see \cite{DCM21}*{Theorem 5.7}.
\medskip

Our main result is obtained by proving that the singular graphons indeed provide random densities with smaller higher order cumulants; and by estimating precisely these cumulants of higher order. In Section \ref{sec:moments}, we start by developing a general framework for the computation of the moments and cumulants of the homomorphism densities $t(F,G_n(\gamma))$ of the $W$-random graphs. These computations lead in particular to Equation \eqref{eq:singular_3}, and to a crude upper bound $|\kappa^{(r)}(Z_n)| = O(n^{-r})$ for the singular graphons. This crude upper bound (without an explicit constant in the $O(\cdot)$) is not sufficient in order to prove the central limit theorem, but it entails \emph{hidden equations} satisfied by the observables of singular graphons (Theorem \ref{thm:additional_equations}). Consider for instance the following graphs:
\begin{align*}
F_1 &=  \begin{tikzpicture}[baseline=-1mm,scale=0.7]
\foreach \x in {(0,0),(1,0),(0.5,0.5),(0.5,-0.5)}
\fill \x circle (3pt);
\draw (0,0) -- (0.5,0.5) -- (1,0) -- (0.5,-0.5) -- (0,0);
\draw (0,0) -- (1,0);
\end{tikzpicture}\quad;\quad 
F_{2a} = \begin{tikzpicture}[baseline=-1mm,scale=0.7]
\foreach \x in {(0,0),(1,0),(0.5,0.5),(0.5,-0.5),(1.5,0.5),(1.5,-0.5),(2,0)}
\fill \x circle (3pt);
\draw (0,0) -- (0.5,0.5) -- (1,0) -- (0.5,-0.5) -- (0,0);
\draw (1,0) -- (1.5,0.5) -- (2,0) -- (1.5,-0.5) -- (1,0);
\draw (0,0) -- (2,0);
\end{tikzpicture}\quad;\quad
F_{2b} = \begin{tikzpicture}[baseline=-1mm,scale=0.7]
\foreach \x in {(0,0),(1,0),(0.5,0.5),(0.5,-0.5),(1.5,0.5),(1.5,-0.5),(2,0)}
\fill \x circle (3pt);
\draw (0,0) -- (0.5,0.5) -- (1,0) -- (0.5,-0.5) -- (0,0);
\draw (1,0) -- (1.5,0.5) -- (2,0) -- (1.5,-0.5) -- (1,0);
\draw (0,0) -- (1,0);
\draw (1.5,0.5) -- (1.5,-0.5);
\end{tikzpicture}\quad;\quad
F_{2c} = \begin{tikzpicture}[baseline=-1mm,scale=0.7]
\foreach \x in {(0,0),(1,0),(0.5,0.5),(0.5,-0.5),(1.5,0.5),(1.5,-0.5),(2,0)}
\fill \x circle (3pt);
\draw (0,0) -- (0.5,0.5) -- (1,0) -- (0.5,-0.5) -- (0,0);
\draw (1,0) -- (1.5,0.5) -- (2,0) -- (1.5,-0.5) -- (1,0);
\draw (0.5,0.5) -- (0.5,-0.5);
\draw (1.5,0.5) -- (1.5,-0.5);
\end{tikzpicture}\,.\\
\end{align*}
The graphs $F_{2a}$, $F_{2b}$ and $F_{2c}$ are involved in the computation of $\kappa_2(F_1,F_1)$:
$$\frac{1}{4}\,\kappa_2(F_1,F_1) = F_{2a}+2\,F_{2b}+F_{2c}-4\,(F_1)^2.$$
If $\gamma$ is a singular graphon, then this observable in $\obs$ vanishes on $\gamma$. Let us now see which equations one can write for the joins of three copies of $F_1$. We set:
\begin{align*}
F_{3a} &= \begin{tikzpicture}[baseline=-1mm,scale=0.7]
\foreach \x in {(0,0),(1,0),(0.5,0.5),(0.5,-0.5),(1.5,0.5),(1.5,-0.5),(2,0),(2.5,0.5),(2.5,-0.5),(3,0)}
\fill \x circle (3pt);
\draw (0,0) -- (0.5,0.5) -- (1,0) -- (0.5,-0.5) -- (0,0);
\draw (1,0) -- (1.5,0.5) -- (2,0) -- (1.5,-0.5) -- (1,0);
\draw (2,0) -- (2.5,0.5) -- (3,0) -- (2.5,-0.5) -- (2,0);
\draw (0,0) -- (3,0);
\end{tikzpicture}\quad;\quad 
F_{3b} = \begin{tikzpicture}[baseline=-1mm,scale=0.7]
\foreach \x in {(0,0),(1,0),(0.5,0.5),(0.5,-0.5),(1.5,0.5),(1.5,-0.5),(2,0),(2.5,0.5),(2.5,-0.5),(3,0)}
\fill \x circle (3pt);
\draw (0,0) -- (0.5,0.5) -- (1,0) -- (0.5,-0.5) -- (0,0);
\draw (1,0) -- (1.5,0.5) -- (2,0) -- (1.5,-0.5) -- (1,0);
\draw (2,0) -- (2.5,0.5) -- (3,0) -- (2.5,-0.5) -- (2,0);
\draw (0,0) -- (2,0);
\draw (2.5,0.5) -- (2.5,-0.5);
\end{tikzpicture}\quad;\quad 
F_{3c} = \begin{tikzpicture}[baseline=-1mm,scale=0.7]
\foreach \x in {(0,0),(1,0),(0.5,0.5),(0.5,-0.5),(1.5,0.5),(1.5,-0.5),(2,0),(2.5,0.5),(2.5,-0.5),(3,0)}
\fill \x circle (3pt);
\draw (0,0) -- (0.5,0.5) -- (1,0) -- (0.5,-0.5) -- (0,0);
\draw (1,0) -- (1.5,0.5) -- (2,0) -- (1.5,-0.5) -- (1,0);
\draw (2,0) -- (2.5,0.5) -- (3,0) -- (2.5,-0.5) -- (2,0);
\draw (0,0) -- (1,0);
\draw (1.5,0.5) -- (1.5,-0.5);
\draw (3,0) -- (2,0);
\end{tikzpicture}\\
F_{3d} &= \begin{tikzpicture}[baseline=-1mm,scale=0.7]
\foreach \x in {(0,0),(1,0),(0.5,0.5),(0.5,-0.5),(1.5,0.5),(1.5,-0.5),(2,0),(2.5,0.5),(2.5,-0.5),(3,0)}
\fill \x circle (3pt);
\draw (0,0) -- (0.5,0.5) -- (1,0) -- (0.5,-0.5) -- (0,0);
\draw (1,0) -- (1.5,0.5) -- (2,0) -- (1.5,-0.5) -- (1,0);
\draw (2,0) -- (2.5,0.5) -- (3,0) -- (2.5,-0.5) -- (2,0);
\draw (2,0) -- (1,0);
\draw (0.5,0.5) -- (0.5,-0.5);
\draw (2.5,0.5) -- (2.5,-0.5);
\end{tikzpicture} \quad;\quad 
F_{3e} = \begin{tikzpicture}[baseline=-1mm,scale=0.7]
\foreach \x in {(0,0),(1,0),(0.5,0.5),(0.5,-0.5),(1.5,0.5),(1.5,-0.5),(2,0),(2.5,0.5),(2.5,-0.5),(3,0)}
\fill \x circle (3pt);
\draw (0,0) -- (0.5,0.5) -- (1,0) -- (0.5,-0.5) -- (0,0);
\draw (1,0) -- (1.5,0.5) -- (2,0) -- (1.5,-0.5) -- (1,0);
\draw (2,0) -- (2.5,0.5) -- (3,0) -- (2.5,-0.5) -- (2,0);
\draw (2,0) -- (3,0);
\draw (1.5,0.5) -- (1.5,-0.5);
\draw (0.5,0.5) -- (0.5,-0.5);
\end{tikzpicture}\quad;\quad 
F_{3f} = \begin{tikzpicture}[baseline=-1mm,scale=0.7]
\foreach \x in {(0,0),(1,0),(0.5,0.5),(0.5,-0.5),(1.5,0.5),(1.5,-0.5),(2,0),(2.5,0.5),(2.5,-0.5),(3,0)}
\fill \x circle (3pt);
\draw (0,0) -- (0.5,0.5) -- (1,0) -- (0.5,-0.5) -- (0,0);
\draw (1,0) -- (1.5,0.5) -- (2,0) -- (1.5,-0.5) -- (1,0);
\draw (2,0) -- (2.5,0.5) -- (3,0) -- (2.5,-0.5) -- (2,0);
\draw (1.5,0.5) -- (1.5,-0.5);
\draw (0.5,0.5) -- (0.5,-0.5);
\draw (2.5,0.5) -- (2.5,-0.5);
\end{tikzpicture}\\
F_{3g} &=\,\, \begin{tikzpicture}[baseline=-1mm,scale=0.7]
\foreach \x in {(0,0),(0.7,0),(0.7,0.7),(0,0.7),(0.7,-0.7),(1.4,-0.7),(1.4,0),(2.1,0),(1.4,0.7),(2.1,0.7)}
\fill \x circle (3pt);
\draw (0,0) rectangle (0.7,0.7);
\draw (0.7,-0.7) rectangle (1.4,0);
\draw (1.4,0) rectangle (2.1,0.7);
\draw (0,0.7) -- (1.4,-0.7);
\draw (1.4,0.7) -- (2.1,0); 
\end{tikzpicture}\quad;\quad
F_{3h} =\,\, \begin{tikzpicture}[baseline=-1mm,scale=0.7]
\foreach \x in {(0,0),(0.7,0),(0.7,0.7),(0,0.7),(0.7,-0.7),(1.4,-0.7),(1.4,0),(2.1,0),(1.4,0.7),(2.1,0.7)}
\fill \x circle (3pt);
\draw (0,0) rectangle (0.7,0.7);
\draw (0.7,-0.7) rectangle (1.4,0);
\draw (1.4,0) rectangle (2.1,0.7);
\draw (0,0.7) -- (1.4,-0.7);
\draw (1.4,0) -- (2.1,0.7); 
\end{tikzpicture}\quad;\quad
F_{3i} =\,\, \begin{tikzpicture}[baseline=-1mm,scale=0.7]
\foreach \x in {(0,0),(0.7,0),(0.7,0.7),(0,0.7),(0.7,-0.7),(1.4,-0.7),(1.4,0),(2.1,0),(1.4,0.7),(2.1,0.7)}
\fill \x circle (3pt);
\draw (0,0) rectangle (0.7,0.7);
\draw (0.7,-0.7) rectangle (1.4,0);
\draw (1.4,0) rectangle (2.1,0.7);
\draw (0,0) -- (0.7,0.7);
\draw (0.7,0) -- (1.4,-0.7);
\draw (1.4,0) -- (2.1,0.7); 
\end{tikzpicture}\quad;\quad
F_{3j} =\,\, \begin{tikzpicture}[baseline=-1mm,scale=0.7]
\foreach \x in {(0,0),(0.7,0),(0.7,0.7),(0,0.7),(0.7,-0.7),(1.4,-0.7),(1.4,0),(2.1,0),(1.4,0.7),(2.1,0.7)}
\fill \x circle (3pt);
\draw (0,0) rectangle (0.7,0.7);
\draw (0.7,-0.7) rectangle (1.4,0);
\draw (1.4,0) rectangle (2.1,0.7);
\draw (0,0) -- (0.7,0.7);
\draw (0.7,0) -- (1.4,-0.7);
\draw (1.4,0.7) -- (2.1,0); 
\end{tikzpicture}\\
F_{3k} &=\!\!\begin{tikzpicture}[baseline=1mm,scale=0.7]
\foreach \x in {(0,0),(-1,0),(1,0),(-0.5,-0.7),(0.5,-0.7),(-1.5,-0.7),(1.5,-0.7),(0.5,0.7),(-0.5,0.7),(0,1.4)}
\fill \x circle (3pt);
\draw (0,1.4) -- (0.5,0.7) -- (-0.5,-0.7) -- (-1.5,-0.7) -- (-1,0) -- (1,0) -- (1.5,-0.7) -- (0.5,-0.7) -- (-0.5,0.7) -- (0,1.4);
\draw (-1.5,-0.7) -- (0,0) -- (0,1.4);
\draw (1.5,-0.7) -- (0,0);
\end{tikzpicture}\quad;\quad
F_{3l} =\!\!\begin{tikzpicture}[baseline=1mm,scale=0.7]
\foreach \x in {(0,0),(-1,0),(1,0),(-0.5,-0.7),(0.5,-0.7),(-1.5,-0.7),(1.5,-0.7),(0.5,0.7),(-0.5,0.7),(0,1.4)}
\fill \x circle (3pt);
\draw (0,1.4) -- (0.5,0.7) -- (-0.5,-0.7) -- (-1.5,-0.7) -- (-1,0) -- (1,0) -- (1.5,-0.7) -- (0.5,-0.7) -- (-0.5,0.7) -- (0,1.4);
\draw (-1.5,-0.7) -- (0,0) -- (0,1.4);
\draw (1,0) -- (0.5,-0.7);
\end{tikzpicture}\quad;\quad
F_{3m} =\!\!\begin{tikzpicture}[baseline=1mm,scale=0.7]
\foreach \x in {(0,0),(-1,0),(1,0),(-0.5,-0.7),(0.5,-0.7),(-1.5,-0.7),(1.5,-0.7),(0.5,0.7),(-0.5,0.7),(0,1.4)}
\fill \x circle (3pt);
\draw (0,1.4) -- (0.5,0.7) -- (-0.5,-0.7) -- (-1.5,-0.7) -- (-1,0) -- (1,0) -- (1.5,-0.7) -- (0.5,-0.7) -- (-0.5,0.7) -- (0,1.4);
\draw (0,0) -- (0,1.4);
\draw (-1,0) -- (-0.5,-0.7);
\draw (1,0) -- (0.5,-0.7);
\end{tikzpicture}\quad;\quad
F_{3n} =\!\!\begin{tikzpicture}[baseline=1mm,scale=0.7]
\foreach \x in {(0,0),(-1,0),(1,0),(-0.5,-0.7),(0.5,-0.7),(-1.5,-0.7),(1.5,-0.7),(0.5,0.7),(-0.5,0.7),(0,1.4)}
\fill \x circle (3pt);
\draw (0,1.4) -- (0.5,0.7) -- (-0.5,-0.7) -- (-1.5,-0.7) -- (-1,0) -- (1,0) -- (1.5,-0.7) -- (0.5,-0.7) -- (-0.5,0.7) -- (0,1.4);
\draw (-0.5,0.7) -- (0.5,0.7);
\draw (-1,0) -- (-0.5,-0.7);
\draw (1,0) -- (0.5,-0.7);
\end{tikzpicture}
\end{align*}\medskip

\noindent Equation \eqref{eq:singular_1} shows that on any singular graphon $\gamma$, the three following observables vanish:
\begin{align*}
\kappa_2(F_1,F_{2a}) &= 4\,F_{3a} + 4\,F_{3b} + 8\,F_{3g} + 8\,F_{3h} + 2\,F_{3k} + 2\,F_{3l}- 28\,F_1\,F_{2a};\\
\kappa_2(F_1,F_{2b}) &= 2\,F_{3b} + 2\,F_{3c} + 2\,F_{3d}+2\,F_{3e} + 4\,F_{3h}+ 8\,F_{3i}+ 4\,F_{3j}+2\,F_{3l}+2\,F_{3m} - 28\,F_1\,F_{2b};\\
\kappa_2(F_1,F_{2c}) &= 4\,F_{3e}+4\,F_{3f} + 8\,F_{3g} +8\,F_{3j} + 2\,F_{3m}+2\,F_{3n}- 28\,F_1\,F_{2c}.
\end{align*}
It turns out that there is an additional equation satisfied by the observables of those graphs which can be obtained by joining three diamond-shaped graphs $F_1$. Thus, if $\gamma$ is a singular graphon, then 
\begin{align*}
\frac{\kappa_3(F_1,F_1,F_1)}{8}&=3\,(F_{3a} + 2\,F_{3b}+ F_{3c}+F_{3d} + 2\,F_{3e}+F_{3f}) +12\,(F_{3g}+F_{3h}+F_{3i}+F_{3j})  \\
&\quad+ (F_{3k} + 3\,F_{3l}+3\,F_{3m}+F_{3n}) -20\,F_1\,(F_{2a}+2 \,F_{2b}+ F_{2c})\end{align*}
vanishes on $\gamma$. This equation is independent from the the vanishing of the three covariance observables $\kappa_2(F_1,F_{2a})$, $\kappa_2(F_1,F_{2b})$ and $\kappa_2(F_1,F_{2c})$. Indeed, up to a scalar multiple, the only linear combination of these three observables which yields a term proportional to $F_1\,(F_{2a}+2 \,F_{2b}+ F_{2c})$ is
\begin{align*}
&\frac{\kappa_2(F_1,F_{2a}+2\,F_{2b}+F_{2c})}{2} \\
&= 2\,(F_{3a} + 2\,F_{3b}+ F_{3c}+F_{3d} + 2\,F_{3e}+F_{3f}) +8\,(F_{3g}+F_{3h}+F_{3i}+F_{3j}) \\
&\quad + (F_{3k} + 3\,F_{3l}+3\,F_{3m}+F_{3n}) -14\,F_1\,(F_{2a}+2 \,F_{2b}+ F_{2c}),
\end{align*}
which is not proportional to the observable written above. These hidden equations for the graph densities of singular graphons will help us to prove our central limit theorem. We believe that they might be of independent interest in order to solve the Conjecture, and to study the singular points of other classes of random models; see the next paragraph.\medskip

The reader might wonder why we considered above all the possible joins of two or three diamond-shaped graphs $F_1$; obviously, there are simpler graphs which could have been used as examples (\emph{e.g.}, the square or the triangle). The reason is the following: independently from the equations coming from the vanishing of the leading terms of the variance and of the higher cumulants of the graph densities of the singular $W$-random graph models, one can write much simpler equations for the graph densities $t(F,\gamma)$ of a singular graphon $\gamma$ when $F$ is a \emph{join-transitive motive} (and this is not the case of $F_1$). These simpler equations are detailed in Section \ref{sec:spectral} (see Theorem \ref{thm:simple_equations}), and they imply in particular the following: if $\gamma$ is a singular graphon with edge density $p$, then for any tree $T$, we have $t(T,\gamma)=p^{|E_T|}$ (Corollary \ref{cor:tree_densities}). It should be noticed that the tree densities of the graphons are not known to determine the graphons themselves. However, they are related to the spectral properties of the integral operator 
\begin{align*}
T_g : \leb^2([0,1]) &\to \leb^2([0,1])\\
f &\mapsto \left(T_g f(\cdot) = \int_0^1 g(x,y)\,f(y)\DD{y} \right)
\end{align*}
associated to a graph function $g$ representative of $\gamma$. The constant graphon $\gamma_p$ yields the operator $T_g = p\,\pi_{\mathrm{constant}}$, where $\pi_\mathrm{constant}$ is the orthogonal projection on the one-dimensional space of constant functions. As a consequence, in order to prove that a graphon is constant, it suffices to estimate the (expectation of the) characteristic polynomial 
$$\det\!\left(I_n + z\,\frac{A_n(\gamma)}{n}\right),$$
where $A_n(\gamma)$ is the \emph{adjacency} matrix of the random graph $G_n(\gamma)$: see Theorem \ref{thm:spectral_constant_graphon}. On the other hand, by combining a form of the matrix-tree theorem with the result on tree densities, we can compute for any singular graphon $\gamma$ with edge density $p$ the expectation of the generalised characteristic polynomial
$$\det\!\left((1+pz)I_n + zE_n - z\,\frac{L_n(\gamma)}{n}\right),$$ 
where $L_n(\gamma)$ is the \emph{Laplacian} matrix of $G_n(\gamma)$, and $E_n = \diag(\eps_1,\ldots,\eps_n)$ is an arbitrary diagonal matrix (Proposition \ref{prop:expectation_laplacian}). Notice that if $E_n = \diag(\eps_1(\gamma),\ldots,\eps_n(\gamma))$ with
$$\eps_i(\gamma) = \big(\text{rescaled and recentered degree of $i$ in $G_n(\gamma)$}\big)=\frac{\deg(i,G_n(\gamma))}{n}-p, $$
then $(1+pz)I_n + zE_n - z\,\frac{L_n(\gamma)}{n} = I_n + z\,\frac{A_n(\gamma)}{n}$. Denote $D\!L_n(z,\gamma;\eps_1,\ldots,\eps_n)$ the determinant of the random matrix with diagonal component $E_n=\diag(\eps_1,\ldots,\eps_n)$, and $F_n(z,\gamma;\eps_1,\ldots,\eps_n) = \esper[D\!L_n(z,\gamma;\eps_1,\ldots,\eps_n)]$.  If $\gamma$ is a singular graphon with edge density $p$, then:
\begin{itemize}
    \item One has a simple formula for $F_n(z,\gamma;\eps_1,\ldots,\eps_n)$.
    \item In order to prove that $\gamma=\gamma_p$, it suffices to estimate $\esper[D\!L_n(z,\gamma;\eps_1(\gamma),\ldots,\eps_n(\gamma))]$.
\end{itemize}
The problem which arises is that the random function $D\!L_n(\gamma)$ is correlated with the random variables $\eps_1(\gamma),\ldots,\eps_n(\gamma)$, so in particular we do not expect that
$$\esper[D\!L_n(z,\gamma;\eps_1(\gamma),\ldots,\eps_n(\gamma))] $$
and
$$\esper[F_n(z,\gamma;\eps_1(\gamma),\ldots,\eps_n(\gamma))]$$
are the same. Nonetheless, our computations show that there is a small hope of proving the Conjecture with a spectral approach if one is able to understand the correlations between the spectrum of the matrices $A_n(\gamma)$ and $L_n(\gamma)$ on the one side, and the (rescaled) random degrees $\eps_1(\gamma),\ldots,\eps_n(\gamma)$ on the other side.\medskip

In Section \ref{sec:convergence}, we establish our Main Theorem by exhibiting an explicit constant $C_r$ in the estimate $|\kappa^{(r)}(Z_n)| = O(n^{-r})$ satisfied by the singular graphons. The proof of this upper bound (Theorem \ref{thm:upper_bound}) relies on arguments similar to those of \cite{DCM21}*{Theorem 5.6}, and it is an adaptation of combinatorial techniques developed in \cite{FMN16}*{Section 9}. It also involves another generalisation of the matrix-tree theorem; this generalisation and the one used in order to compute the expected determinants related to the Laplacian matrices $L_n(\gamma)$  are recalled in an Appendix at the end of the article. The constants $C_r$ which we obtain are good enough so that we can resum the cumulant generating series, thereby proving the convergence on a disk of the Laplace transforms of the variables $nZ_n$. We also compute the limit of $n^r\,\kappa^{(r)}(Z_n)$ when the motive chosen is the graph $\edgegraph$, and when $r \in \{2,3,4\}$. These computations rely in particular on the equations from Section \ref{sec:spectral} satisfied by the graph densities of the join-transitive motives in a singular graphon. They enable us to prove that the only singular graphons with asymptotically normal random edge densities are the constant graphons (Proposition \ref{prop:gaussian_edge_density}).
\medskip

\subsection{Identification of the singular points of the mod-Gaussian moduli spaces} The theoretical results and methods of this article might prove useful in order to study other similar classes of random models. Indeed, the articles \cite{FMN16}*{Section 9} and \cite{FMN20} provide us with general techniques that enable one to prove that a class of random models has observables which are asymptotically normal with a fast decay of cumulants, and with an explicit constant in the $O(\cdot)$ in Equation \eqref{eq:FD}. 
Therefore, Theorems \ref{thm:additional_equations} and \ref{thm:upper_bound} adapt \emph{mutatis mutandi} to all the models considered in \cite{FMN20}, and also to the models of random metric spaces considered in \cite{DCM21}. In particular, this might allow us to identify the singular points of the mod-Gaussian moduli spaces considered in these two papers. More precisely:
\begin{itemize}
    \item The singular points of the space of graphons are studied in the present paper.
    \item The singular points of the space of permutons are unknown, and it might be that there are no singular points. The analogue for permutons of Theorem \ref{thm:additional_equations} might be the key in order to prove this result.
    \item The Thoma simplex is a bi-infinite dimensional simplex which parametrises the central measures on integer partitions and the extremal characters of the infinite symmetric group; the observables of these random models are random character values of the symmetric groups $\sym(n)$, $n \geq 1$. It is not very difficult to prove that the singular points of the Thoma simplex, which correspond to renormalised random character values with variance $O(n^{-2})$, are parametrised by $\Z$: the integer $0$ corresponds to the Plancherel measures, the positive integer $n \geq 1$ corresponds to the Schur--Weyl measures associated to the Thoma parameter $((\frac{1}{n},\ldots,\frac{1}{n},0,0\ldots),(0,0,\ldots))$, and the negative integer $-n \leq -1$ corresponds to the dual central measure associated to the Thoma parameter $((0,0,\ldots),(\frac{1}{n},\ldots,\frac{1}{n},0,0\ldots))$. This is a relatively easy computation by using the Kerov--Olshanski algebra of polynomial observables, see \cite{KO94} and \cite{Mel17}*{Chapter 7}.
    \item The Gromov--Hausdorff--Prohorov space of metric measure spaces parametrises the complete metric spaces endowed with a probability measure $(X,d,\mu)$, and these spaces can be approximated by random discrete spaces. We conjecture that the only singular points of the GHP space are the compact homogeneous spaces $X=G/K$ with $G$ compact group and $K$ closed subgroup. An important step towards this conjecture has been proved in \cite{DCM21}: if one replaces the equivalent of Equation \eqref{eq:singular_3} in this setting by a slightly stronger condition, then the solutions are indeed the compact homogeneous spaces. The replacement of the equation of singularity by a stronger condition could be justified by using the hidden equations for the observables, and an analogue in the setting of random metric spaces of our equations for join-transitive motives. In the opposite direction, \cite{DCM21} contains the equivalent in the setting of random metric spaces of what is really missing in order to solve our Conjecture: namely, the implication that if $t(T,\gamma)=p^{|E_T|}$ for any tree, then $t(G,\gamma)=p^{|E_G|}$ for any graph.
\end{itemize}
We plan to address the conjectures evoked in the second and fourth item above in forthcoming works.

\begin{remark}
The project planned above enables us to clarify a bit the status of our Main Theorem with respect to the Conjecture:
\begin{itemize}
    \item Our main result and the computations which lead to it show that the singular gra\-phons share many properties with the constant (Erd\H{o}s--Rényi) graphons. Thus, it stands clearly in favor of the Conjecture.
    \item However, if the Conjecture holds, then our Main Theorem is contained in the previously known results from \cites{Jan88,Now89,FMN16}. So in a sense, our Main Theorem is stronger \emph{if the Conjecture does not hold}: we would have then showed that the singular graphons which are not constant still satisfy many results known for the constant graphons.
    \item Regardless of the Conjecture, the techniques developed throughout the article hold in the more general setting of mod-Gaussian moduli spaces, and this might be the main interest of our paper.
\end{itemize}
\end{remark}
\medskip

\noindent \textbf{Acknowledgements.} The author would like to address his thanks to his colleagues of the probability team of the Laboratoire de Mathématiques d'Orsay for several enlightening discussions around the graphon models and the matrix-tree theorem. He is also thankful to the departments of mathematics of the University College Dublin and of the École Polytechnique, where he was invited to present preliminary versions of this work; and to his colleagues Valentin Féray and Ashkan Nikeghbali, with whom he discussed about several aspects of the conjecture.
\bigskip

\section{Moments and cumulants of the subgraph counts}\label{sec:moments}
In this section, we fix a graphon $\gamma \in \Gspa$, and given a finite simple graph $F$ on $k$ vertices, we denote $S_n(F) = n^k\,t(F,G_n(\gamma)) = |\hom(F,G_n(\gamma))|$. We are going to explain how to compute the moments and the cumulants of these random variables. These computations are similar to those performed in \cite{FMN20}*{Section 5.1}, but with a different focus: we want to explain why the condition of being singular (Equations \eqref{eq:singular_1} or \eqref{eq:singular_2} or \eqref{eq:singular_3}) implies additional equations for the graph observables $t(F,\gamma)$ (see Theorem \ref{thm:additional_equations}). These hidden equations will be used in Section \ref{sec:convergence} in order to prove the central limit theorem for singular graphons. It turns out that they also imply that $t(T,\gamma)=p^{|E_T|}$ for any singular graphon $\gamma$ with edge density $p$ and for any tree $T$ with size less than $11$; see the discussion at the end of Example \ref{ex:cutting_trees}. As explained in the introduction, similar hidden equations hold for any singular point of a mod-Gaussian moduli space in the sense of \cite{FMN20}*{Section 6}.
\medskip

\subsection{Moments of the graph densities}\label{sub:moments_graph_densities}
The computation of the moments of the random variables $S_n(F)$ is related to the operation of contraction of a graph along a set partition of its vertices. Given a set $V$, we denote $\pym(V)$ the set of set partitions of $V$, and we set $\pym(k) = \pym(\lle 1,k\rre)$. For instance, the $5$ elements of $\pym(3)$ are $\{1,2,3\}$, $\{1,2\}\sqcup \{3\}$, $\{1,3\}\sqcup \{2\}$, $\{2,3\}\sqcup \{1\}$ and $\{1\} \sqcup \{2\}\sqcup \{3\}$. Consider a graph $F$ on the set of integers $\lle 1,k\rre$, and a set partition $\pi=\pi_1 \sqcup \pi_2 \sqcup \cdots \sqcup \pi_\ell$ in $\pym(k)$. The \emph{contracted graph} $F \downarrow \pi$ is the simple graph on the set of vertices $\{1,2,\ldots,\ell\}$ with one edge between $i$ and $j$ if there exists $a \in \pi_i$ and $b \in \pi_j$ such that $\{a,b\} \in E_F$. Here, we also allow \emph{loops}, so we create a loop $(i,i)$ in $F \downarrow \pi$ if there are two elements $a \neq b$ in $\pi_i$ such that $\{a,b\} \in E_F$. The calculation of the first moment of $S_n(F)$ involves the set partitions $\pi$ which do not create loops:

\begin{proposition}\label{prop:moment}
For any graph $F$ on $k$ vertices, we have 
$$\esper[S_n(F)] = \sum_{\substack{\pi\in \pym(k) \\ F \downarrow \pi \text{ is loopless}}} n^{\downarrow \ell(\pi)}\,t(F\downarrow \pi,\gamma),$$
where $\ell(\pi)$ denotes the number of parts of a set partition $\pi$, and $n^{\downarrow l}=n(n-1)(n-2)\cdots (n-l+1)$.
\end{proposition}

\begin{proof}
It is convenient to consider $S_n(F)$ as a sum of $n^k$ dependent Bernoulli random variables $A_{\phi}(F)$ labelled by the maps $\phi : \lle 1,k \rre \to \lle 1,n\rre$:
$$A_{\phi}(F) = \begin{cases}
    1 &\text{if }\phi \text{ is a morphism from $F$ to }G_n(\gamma),\\
    0 & \text{otherwise}.
\end{cases}$$
By definition of a graph morphism, we have 
$$A_\phi(F) = \prod_{\{a,b\} \in E_F} 1_{\big(\{\phi(a),\phi(b)\} \in E_{G_n(\gamma)}\big)}.$$
We denote $\pi=\pi_1\sqcup \cdots \sqcup \pi_l$ the set partition of size $k$ associated to the equivalence relation $a \sim_\pi b \iff \phi(a)=\phi(b)$, and $\psi : \lle 1,l\rre \to \lle 1,n\rre$ the injective map defined by 
$\psi(i) = \phi(a)$ for any $a \in \pi_i$. If $a \in \pi_i$ and $b \in \pi_j$, then 
\begin{align*}
\big(\{\phi(a),\phi(b)\} \in E_{G_n(\gamma)} \big) &\iff \big(\{\psi(i),\psi(j)\} \in E_{G_n(\gamma)}).
\end{align*}
Therefore, $\phi$ is a morphism from $F$ to $G_n(\gamma)$ if and only if $\psi$ is a morphism from $F \downarrow \pi$ to $G_n(\gamma)$. Moreover, given a set partition $\pi \in \pym(k)$ of length $l$ and an injective map $\psi : \lle 1,l\rre \to \lle 1,n\rre$, we can reconstruct the map $\phi$. So,
$$\sum_{\phi : \lle 1,k\rre \to \lle 1,n\rre} A_\phi(F) = \sum_{\substack{\pi \in \pym(k) \\ \psi : \lle 1,\ell(\pi)\rre \to \lle 1,n\rre \\ \psi \text{ injective map}}} A_{\psi}(F \downarrow \pi).$$
Now, the $W$-random graphs have a property of coherence: the law of the restriction of the graph $G_n(\gamma)$ to any subset $\{\psi(1),\psi(2),\ldots,\psi(l)\}$ of $\lle 1,n\rre$ is the law of $G_l(\gamma)$. This is trivial from the definition, and it is a part of a characterisation of the graphon models $(G_n(\gamma))_{n \geq 1}$, see \cite{LS06}*{Theorem 2.7}. Therefore, for any injective map $\psi : \lle 1,l\rre \to \lle 1,n\rre$, we have
\begin{align*}
\esper[A_{\psi}(F \downarrow \pi)] &= \proba[F \downarrow \pi \text{ is a subgraph of }G_l(\gamma)]\\
&=\begin{cases}
    \esper\!\left[\prod_{\{i,j\}\in E_{F\downarrow \pi}}1_{\big(U_{ij} \leq g(X_i,X_j)\big)}\right] &\text{if $F \downarrow \pi$ is loopless},\\
    0 &\text{if $F\downarrow \pi$ contains a loop}. 
\end{cases}
\end{align*}
The vanishing of the expectation in the case of loops comes from the fact that by definition, $G_n(\gamma)$ is loopless. On the other hand, the expectation in the loopless case is $$\int_{[0,1]^l} \left(\prod_{\{i,j\}\in E_{F\downarrow \pi}} g(x_i,x_j)\right)\DD{x_1}\cdots \DD{x_l} = t(F\downarrow \pi,\gamma).$$ 
Finally, there are $n^{\downarrow l}$ injective maps from $\lle 1,l\rre$ to $\lle 1,n\rre$, whence the identity.
\end{proof}

\begin{example}\label{ex:expectations_bi_edge}
We have \begin{align*}
\esper[S_n(\edgegraph)] &= n(n-1)\,t(\edgegraph,\,\gamma);\\
\esper[S_n(\begin{tikzpicture}[baseline=-1mm,scale=0.5]
\foreach \x in {(0,0),(1,0),(2,0)}
\fill \x circle (4pt);
\draw (0,0) -- (1,0) -- (2,0);
\end{tikzpicture})] &= n(n-1)(n-2)\,t(\begin{tikzpicture}[baseline=-1mm,scale=0.5]
\foreach \x in {(0,0),(1,0),(2,0)}
\fill \x circle (4pt);
\draw (0,0) -- (1,0) -- (2,0);
\end{tikzpicture},\,\gamma) + n(n-1)\,t(\edgegraph,\,\gamma).
\end{align*}
 Indeed, in the first case, the only loopless contraction of the graph $F = \edgegraph$
is $F$ itself, and in the second case where $F= \begin{tikzpicture}[baseline=-1mm,scale=0.5]
\foreach \x in {(0,0),(1,0),(2,0)}
\fill \x circle (4pt);
\draw (0,0) -- (1,0) -- (2,0);
\end{tikzpicture}$, in addition to $F$ itself, we can contract $F$ without creating loops by identifying the two ends of the line.
\end{example}
\medskip
 
Since $\prod_{r=1}^s\hom(F_r,G) = \hom\left(\bigsqcup_{r=1}^s F_r,G\right)$, Proposition \ref{prop:moment} also yields the joint moment of any set of random variables $S_n(F_1),S_n(F_2),\ldots,S_n(F_s)$:
$$\esper\!\left[\prod_{r=1}^s S_n(F_r)\right] = \sum_{\pi} n^{\downarrow \ell(\pi)}\,1_{\left(\bigsqcup_{r=1}^s F_r\right)\downarrow \pi \text{ is loopless}}\,\,t\!\left(\left(\bigsqcup_{r=1}^s F_r\right)\downarrow \pi,\gamma\right)$$
where the sum runs over set partitions $\pi$ of the set of vertices of $\bigsqcup_{r=1}^s F_r$. As an application of these formulas, let us see how to compute the leading terms in the two first moments of a random density $t(F,G_n(\gamma))$. For the first moment, notice that in the expansion of $\esper[S_n(F)]$, the only term of degree $k=|V_F|$ in $n$ is the one where $\pi=\{1\}\sqcup \{2\} \sqcup \cdots \sqcup \{k\}$ and $F \downarrow \pi = F$; all the other terms yield polynomials of degree smaller than $k-1$. So,
$$\esper[t(F,G_n(\gamma))] = \frac{\esper[S_n(F)]}{n^k}  = \frac{n^{\downarrow k}}{n^k}\,t(F,\gamma) + O\!\left(\frac{1}{n}\right) = t(F,\gamma) + O\!\left(\frac{1}{n}\right).$$
For the second moment, the interesting quantity to look at is $\cov(t(F_1,\gamma),t(F_2,\gamma))$, where $F_1$ and $F_2$ are two finite graphs on $k_1$ and $k_2$ vertices. We shall denote $V_{F_1}=\lle 1,k_1\rre$ and $V_{F_2}=\lle 1',k_2'\rre$. Notice that $n^{\downarrow l} = n^l - \frac{l(l-1)}{2}\, n^{l-1} + O(n^{l-2})$. Therefore, the leading terms of the joint moment $\esper[S_n(F_1)\,S_n(F_2)]$ are:
\begin{align*}
& n^{\downarrow k_1+k_2}\,t(F_1\,F_2,\gamma) + \sum_{\pi \in \pym_{\mathrm{pair}}(V_{F_1}\sqcup V_{F_2})} n^{\downarrow k_1+k_2-1}\,t((F_1\,F_2) \downarrow \pi,\gamma) + O(n^{k_1+k_2-2}) \\
&= n^{k_1+k_2}\,t(F_1\,F_2,\gamma) + \sum_{\pi \in \pym_{\mathrm{pair}}(V_{F_1}\sqcup V_{F_2})} n^{k_1+k_2-1}\,(t((F_1\,F_2) \downarrow \pi,\gamma)-t(F_1\,F_2,\gamma)) + O(n^{k_1+k_2-2}),
\end{align*}
where $\pym_{\text{pair}}(V)$ denotes the subset of $\pym(V)$ whose elements are the set partitions which consists in one pair $\{a,b\}$ and singletons. Similarly, the leading terms in $\esper[S_n(F_1)]\,\esper[S_n(F_2)]$ of expectations are:
\begin{align*}
&n^{\downarrow k_1}n^{\downarrow k_2}\,t(F_1\,F_2,\gamma) + \sum_{\pi \in \pym_{\mathrm{pair}}(V_{F_1})} n^{\downarrow k_1-1}n^{\downarrow k_2}\,t((F_1 \downarrow \pi)\,F_2,\gamma) \\
&+ \sum_{\pi \in \pym_{\mathrm{pair}}(V_{F_2})} n^{\downarrow k_1}n^{\downarrow k_2-1}\,t(F_1\,(F_2 \downarrow \pi),\gamma) + O(n^{k_1+k_2-2})\\
&=n^{k_1+k_2}\,t(F_1\,F_2,\gamma) + \sum_{\pi \in \pym_{\mathrm{pair}}(V_{F_1})} n^{ k_1+k_2-1}\,(t((F_1 \downarrow \pi)\,F_2,\gamma) - t(F_1\,F_2,\gamma)) \\
&\quad + \sum_{\pi \in \pym_{\mathrm{pair}}(V_{F_2})} n^{k_1+k_2-1}\,(t(F_1\,(F_2 \downarrow \pi),\gamma) -t(F_1\,F_2,\gamma))+ O(n^{k_1+k_2-2}).
\end{align*}
When subtracting these two expressions, the terms of degree $k_1+k_2$ disappear, and in degree $k_1+k_2-1$, the set partitions that remain are those that consist in singletons and one pair $\{a,b\}$ with $a \in \lle 1,k_1\rre$ and $b \in \lle 1',k_2'\rre$. Then, $(F_1\,F_2)\downarrow \pi = (F_1 \join F_2)(a,b)$, so we recover the expression of $\kappa_2(F_1,F_2)$:
\begin{align*}
\cov(t(F_1,G_n(\gamma)),t(F_2,G_n(\gamma))) & = \frac{\cov(S_n(F_1),S_n(F_2))}{n^{k_1+k_2}} \\
&= \frac{1}{n} \sum_{\substack{1\leq i_1\leq k_1 \\ 1 \leq i_2 \leq k_2}} \big(t((F_1\join F_2)(i_1,i_2),\gamma) - t(F_1\,F_2,\gamma)\big) + O\!\left(\frac{1}{n^2}\right) \\
&= \frac{\kappa_2(F_1,F_2)(\gamma)}{n} + O\!\left(\frac{1}{n^2}\right).
\end{align*}
This explains the discussion of the introduction and our Equation \eqref{eq:singular_3} characterising the singular graphons with small fluctuations of observables.
\medskip

\subsection{Joint cumulants of the graph densities}\label{sub:joint_cumulants}
An important idea introduced in \cite{Jan88} for the study of the fluctuations of the models of random graphs is to look more generally at the asymptotic properties of the \emph{joint cumulants} of the random densities $t(F_1,G_n(\gamma)),\ldots$, $t(F_s,G_n(\gamma))$. Recall that if $X_1,\ldots,X_s$ are bounded random variables, then their joint cumulant is given by the formula:
$$\kappa(X_1,X_2,\ldots,X_s) = \sum_{\pi \in \pym(s)} \mu(\pi)\,\prod_{i=1}^{\ell(\pi)} \esper\!\left[\prod_{j \in \pi_i} X_j\right],$$
where $\mu(\pi) = (-1)^{\ell(\pi)-1}\,(\ell(\pi)-1)!$ is the Möbius function of the lattice of set partitions \cite{LS59}. For instance, the joint cumulant of two random variables is their covariance. As a consequence of the fact that the random variables $S_n(F_i)$ are sums of dependent random variables with a sparse dependency graph, \cite{FMN16}*{Theorem 9.2.2} and \cite{FMN20}*{Lemma 19} proved the following upper bound on cumulants. Given graphs $F_1,\ldots,F_s$ on $k_1,\ldots,k_s$ vertices, assuming without loss of generality that $k_1\leq k_2 \leq \cdots \leq k_s$, we have for any graphon $\gamma$:
\begin{equation}
\left|\kappa(S_n(F_1),\ldots,S_n(F_s))\right|\leq 2^{s-1}\,s^{s-2}\,(k_1k_2\cdots k_{s-1})^2\,n^{k_1+\cdots+k_s-s+1}.\tag{BC}\label{eq:bound_cumulant}
\end{equation}
In particular, the joint cumulant of the random variables $S_n(F_1),\ldots,S_n(F_s)$, which is a combination of the joint moments of these variables and therefore should be a polynomial in $n$ of degree smaller than $k_1+k_2+\cdots+k_s$, is in fact a polynomial in $n$ of degree at most $k_1+k_2+\cdots+k_s-(s-1)$. The goal of this paragraph is to give a general expression of the leading term of this polynomial; see Proposition \ref{prop:expansion_cumulant}. This is important because, as we shall explain at the end of this section, if a graphon $\gamma$ is singular, then this leading term vanishes. An improvement of the upper bound \eqref{eq:bound_cumulant} for singular graphons will be the main argument of the proof of our Main Theorem; see Theorem \ref{thm:upper_bound}.\medskip

The computation of the leading term in $\kappa(S_n(F_1),\ldots,S_n(F_s))$ involves hypergraphs and hypertrees, and these combinatorial objects shall play an essential role until the end of this article. Given a finite set $V$, a \emph{hypergraph} $H$ with vertex set $V$ is a finite collection $E$ of hyperedges which are multisets $e=\{v_1,v_2,\ldots,v_d\}$ of elements of $V$, with $d \geq 2$. Beware that we allow multiple occurrences of a vertex $v$ in a hyperedge $e$ of a hypergraph (we then speak of a \emph{hyperloop}), and also that the collection $E$ of hyperedges can contain multiple occurrences of a hyperedge. The degree of a hyperedge $e$ is defined by $\deg e = d-1$, where $d$ is the number of elements in $e$ (counting multiplicities). Thus, a usual edge of a graph is a hyperedge of degree $1$. A hypergraph $H=(V,E)$ is said \emph{connected} if, given two vertices $v\neq w$ in $V$, there is a sequence of hyperedges $e_1,e_2,\ldots,e_r$ in $E$ such that $v \in e_1$, $w \in e_r$ and the intersections $e_i \cap e_{i+1}$ are not empty. By induction on $|V|$, it is easily seen that one has the following generalisation of the inequality $|E|\geq |V|-1$ for connected graphs:
$$\big(H=(V,E) \text{ is a connected hypergraph}\big) \Rightarrow \left(\sum_{e \in E} \deg e \geq |V|-1\right).$$
A \emph{hypertree} is a connected hypergraph $H=(V,E)$ such that $\sum_{e \in E} \deg e = |V|-1$. This implies in particular that $G$ does not contain a hyperloop: every hyperedge $e \in E$ is a set of elements without multiplicities. We denote $\Hset(s)$ the set of labelled hypertrees on $s$ vertices in $\lle 1,s\rre$; the cardinality of $\Hset(s)$ is computed in \cite{Kal99} by using generating series.
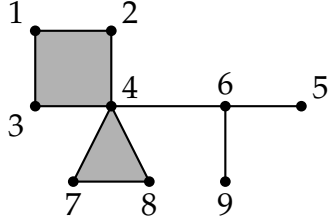
\begin{figure}[ht]
\begin{center}        
\begin{tikzpicture}[scale=1]
\fill [white!70!black] (-1,0) rectangle (0,1);
\fill [white!70!black] (-0.5,-1) -- (0.5,-1) -- (0,0);
\foreach \x in {(0,0),(0,1),(-1,0),(-1,1),(1.5,0),(2.5,0),(1.5,-1),(-0.5,-1),(0.5,-1)}
\fill \x circle (2pt);
\draw [thick] (0,0) -- (-0.5,-1) -- (0.5,-1) -- (0,0) -- (-1,0) -- (-1,1) -- (0,1) -- (0,0) -- (1.5,0) -- (2.5,0);
\draw [thick] (1.5,0) -- (1.5,-1);
\draw (-1.25,1.25) node {$1$};
\draw (0.25,1.25) node {$2$};
\draw (0.25,0.25) node {$4$};
\draw (-1.25,-0.25) node {$3$};
\draw (-0.5,-1.3) node {$7$};
\draw (0.5,-1.3) node {$8$};
\draw (1.5,-1.3) node {$9$};
\draw (1.5,0.3) node {$6$};
\draw (2.75,0.25) node {$5$};
\end{tikzpicture}
\caption{A hypertree on $9$ vertices, with $3$ hyperedges of degree $1$, one hyperedge of degree $2$ and one hyperedge of degree $3$.\label{fig:hypertree}}
\end{center}
\end{figure}

In the following, given graphs $F_1,\ldots,F_s$ on $k_1,\ldots,k_s$ vertices, we shall always label the union of their vertex sets as follows:
$$V = \bigsqcup_{r=1}^s V_{F_r} = \bigsqcup_{r=1}^s \{(r,1),(r,2),\ldots,(r,k_r)\},$$
and we shall also always use the letter $V$ to denote this disjoint union.
A set partition $\pi \in \pym(V)$ gives rise to a hypergraph $H_\pi$ on the set of vertices $\lle 1,s\rre$: to any part $\pi_i=\{(r_{i,1},a_{i,1}),\ldots,(r_{i,m_i},a_{i,m_i})\}$ of the set partition $\pi$ with size $m_i \geq 2$, we associate the hyperedge $e_i = \{r_{i,1},r_{i,2},\ldots,r_{i,m_i}\}$, and we set $E_{H_\pi} = \{e_1,e_2,\ldots,e_{\ell_{\geq 2}(\pi)}\}$, where $\ell_{\geq 2}(\pi)$ is the number of parts of $\pi$ with size larger than $2$. If we draw each part of the set partition $\pi$ as a hyperedge connecting the elements of this part and denote $G_\pi$ the resulting hypergraph, then $H_\pi$ is obtained from $G_\pi$ by taking the image hypergraph by the map $(r,a)\mapsto r$. 
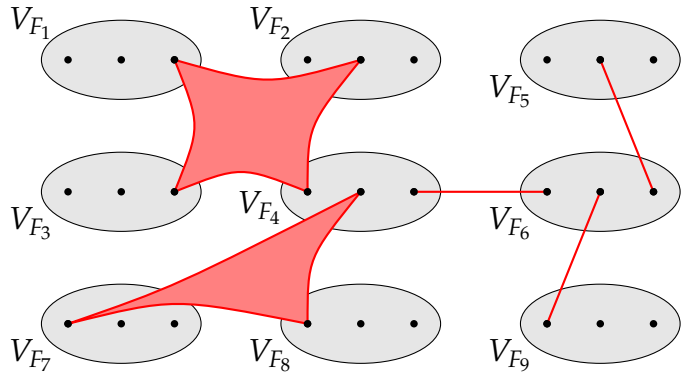
\begin{figure}[ht]
\begin{center}
\begin{tikzpicture}[scale=0.7]

\filldraw [fill=white!90!black] (2,0) ellipse (1.5cm and 0.75cm);
\foreach \x in {1,2,3}
\fill (\x,0) circle (2pt);
\draw (0.3, 0.7) node {$V_{F_1}$};

\begin{scope}[shift={(4.5,0)}]
\filldraw [fill=white!90!black] (2,0) ellipse (1.5cm and 0.75cm);
\foreach \x in {1,2,3}
\fill (\x,0) circle (2pt);
\draw (0.3, 0.7) node {$V_{F_2}$};
\end{scope}

\begin{scope}[shift={(0,-2.5)}]
\filldraw [fill=white!90!black] (2,0) ellipse (1.5cm and 0.75cm);
\foreach \x in {1,2,3}
\fill (\x,0) circle (2pt);
\draw (0.3, -0.6) node {$V_{F_3}$};
\end{scope}

\begin{scope}[shift={(4.5,-2.5)}]
\filldraw [fill=white!90!black] (2,0) ellipse (1.5cm and 0.75cm);
\foreach \x in {1,2,3}
\fill (\x,0) circle (2pt);
\draw (0.1, -0.3) node {$V_{F_4}$};
\end{scope}

\begin{scope}[shift={(9,0)}]
\filldraw [fill=white!90!black] (2,0) ellipse (1.5cm and 0.75cm);
\foreach \x in {1,2,3}
\fill (\x,0) circle (2pt);
\draw (0.3, -0.6) node {$V_{F_5}$};
\end{scope}

\begin{scope}[shift={(9,-2.5)}]
\filldraw [fill=white!90!black] (2,0) ellipse (1.5cm and 0.75cm);
\foreach \x in {1,2,3}
\fill (\x,0) circle (2pt);
\draw (0.3, -0.6) node {$V_{F_6}$};
\end{scope}

\begin{scope}[shift={(0,-5)}]
\filldraw [fill=white!90!black] (2,0) ellipse (1.5cm and 0.75cm);
\foreach \x in {1,2,3}
\fill (\x,0) circle (2pt);
\draw (0.3, -0.6) node {$V_{F_7}$};
\end{scope}

\begin{scope}[shift={(4.5,-5)}]
\filldraw [fill=white!90!black] (2,0) ellipse (1.5cm and 0.75cm);
\foreach \x in {1,2,3}
\fill (\x,0) circle (2pt);
\draw (0.3, -0.6) node {$V_{F_8}$};
\end{scope}

\begin{scope}[shift={(9,-5)}]
\filldraw [fill=white!90!black] (2,0) ellipse (1.5cm and 0.75cm);
\foreach \x in {1,2,3}
\fill (\x,0) circle (2pt);
\draw (0.3, -0.6) node {$V_{F_9}$};
\end{scope}

\filldraw [red,fill=red!50!white,thick] (3,0) .. controls (4.75,-0.5) .. (6.5,0) .. controls (5.5,-1.25) .. (5.5,-2.5) .. controls (4.25,-2) .. (3,-2.5) .. controls (3.5,-1.25) .. (3,0);

\filldraw [red,fill=red!50!white,thick] (6.5,-2.5) .. controls (5.5,-3.75) .. (5.5,-5) .. controls (3,-4.5) .. (1,-5) .. controls (3,-4.3) .. (6.5,-2.5);

\draw [red,thick] (7.5,-2.5) -- (10,-2.5);
\draw [red,thick] (11,0) -- (12,-2.5);
\draw [red,thick] (10,-5) -- (11,-2.5);

\foreach \z in {(3,0), (6.5,0), (5.5,-2.5) , (3,-2.5), (6.5,-2.5), (5.5,-5) , (1,-5), (7.5,-2.5) , (10,-2.5), (11,0) , (12,-2.5), (10,-5), (11,-2.5)}
\fill \z circle (2pt);
\end{tikzpicture}
\caption{A set partition $\pi$ of a set of $27=9*3$ vertices whose corresponding hypergraph $H_\pi$ is the hypertree from Figure \ref{fig:hypertree}; the parts of $\pi$ which are not singletons are drawn in red.\label{fig:set_partition_hypertree}}
\end{center}
\end{figure}
\medskip

\noindent For instance, with $s=9$ and $k_1=k_2=\cdots=k_s=3$, a set partition $\pi$ which gives the hypertree $H_\pi$ from Figure \ref{fig:hypertree} is drawn in Figure \ref{fig:set_partition_hypertree}. The hypergraphs $H_\pi$ are allowed to have multiple hyperedges and to have hyperedges with multiple occurrences of a vertex; however, this does not happen if $H_\pi$ is a hypertree. \medskip

We now consider the joint cumulant 
$\kappa(S_n(F_1),\ldots,S_n(F_s))$, which we expand by multilinearity:
\begin{equation}
\kappa(S_n(F_1),\ldots,S_n(F_s)) = \sum_{\substack{\phi_1 : \lle 1,k_1 \rre \to \lle 1,n\rre \\ \vdots \qquad\vdots \vspace*{1mm}\\ \phi_s : \lle 1,k_s\rre \to \lle 1,n\rre}} \kappa(A_{\phi_1}(F_1),\ldots,A_{\phi_s}(F_s)).\tag{M}\label{eq:multilinearity}
\end{equation}
A collection of maps $(\phi_r)_{1 \leq r \leq s}$ can be seen as a single map $\phi : V \to \lle 1,n\rre$, by setting $\phi(r,a) = \phi_r(a)$. Then, as in the proof of Proposition \ref{prop:moment}, we can associate to the map $\phi$ a set partition $\pi=\pi(\phi) \in \pym(V)$ whose parts correspond to elements with the same image by $\phi$.

\begin{lemma}\label{lem:connected_hypergraph}
With the notations introduced above, if $H_{\pi(\phi)}$ is not a connected hypergraph, then the elementary joint cumulant $\kappa(A_{\phi_1}(F_1),\ldots,A_{\phi_s}(F_s))$ vanishes.
\end{lemma}

\begin{proof}
Suppose that we can split $H_{\pi(\phi)}$ in two components which are not connected; up to permutation of the graphs $F_r$ (this does not change the value of the joint cumulant), we can assume that these components are $\lle 1,t\rre$ and $\lle t+1,s\rre$. The unions of images
$$P = \bigcup_{r=1}^t \phi_r(\lle 1,k_r\rre) \quad \text{and}\quad Q = \bigcup_{r=t+1}^s \phi_r(\lle 1,k_r\rre)$$
are then disjoint, because otherwise we would have a (hyper)edge between some $r \in \lle 1,t\rre$ and some $r' \in \lle t+1,s\rre$. However, the variables $A_{\phi_1}(F_1),\ldots,A_{\phi_t}(F_t)$ are measurable with respect to the $\sigma$-field spanned by the random variables $X_i$ and $U_{ij}$ with indices $i$ and $j$ in $P$; whereas the variables $A_{\phi_{t+1}}(F_{t+1}),\ldots,A_{\phi_s}(F_s)$ are measurable with respect to the $\sigma$-field spanned by the random variables $X_i$ and $U_{ij}$ with indices $i$ and $j$ in $Q$. Therefore, these two sets of variables are independent:
$$(A_{\phi_1}(F_1),\ldots,A_{\phi_t}(F_t)) \text{ and } (A_{\phi_{t+1}}(F_{t+1}),\ldots,A_{\phi_s}(F_s)) \text{ are independent}.$$
It is then a well-known property of the joint cumulants that they vanish when evaluated on random variables that can be split in independent blocks; see \cite{LS59}.
\end{proof}
\medskip

In the expansion by multilinearity \eqref{eq:multilinearity}, we can now gather the maps $\phi : V \to \lle 1,n\rre$ according to their hypergraphs $H_{\pi(\phi)}$. By the previous lemma, if a hypergraph $H$ yields a non-zero contribution, then it is connected. Moreover, if $H$ is a fixed connected hypergraph, then the number of choices for a map $\phi$ with $H_{\pi(\phi)}=H=(V,E)$ is a $O(n^{k_1+\cdots+k_s - \sum_{e \in E} \deg e})$. Let us be a bit more precise on this enumeration. In order to reconstruct a map $\phi$, we first have to choose for each hyperedge $e_i = \{r_{i,1},r_{i,2},\ldots,r_{i,m_i}\}$ a collection of elements $\{a_{i,1},\ldots,a_{i,m_i}\}$ with each $a_{i,j} \in [\![]] 1,k_{r_{i,j}}]\!]$. We ask that $a_{i,j} \neq a_{i',j'}$ if $r_{i,j}=r_{i',j'}$ but $(i,j) \neq (i',j')$. We have then reconstructed a partial set partition $\pi$ of the set $V$, whose parts are the $\{(r_{i,1},a_{i,1}),\ldots,(r_{i,m_i},a_{i,m_i})\}$. We complete $\pi$ in a full set partition of $V$ by adding singletons. Then, to reconstruct $\phi$ such that $\pi=\pi(\phi)$ and $H= H_{\pi(\phi)}$, we have $n^{\downarrow \ell(\pi)}$ possibilities. The length of the set partition $\pi$ is $k_1+k_2+\cdots+k_s-\sum_{e \in E}\deg e$, and on the other hand, the number of choices for $\pi$ is finite and independent from $n$: it only depends on the hypergraph $H$. This explains the estimate claimed above. Now, if we want to maximise the power of $n$, then we need $H$ to be connected but with $\sum_{e \in E} \deg e$ as small as possible; so, $H$ has to be a hypertree. We have therefore proved:

\begin{lemma}
For any finite graphs $F_1,\ldots,F_s$ on $k_1,\ldots,k_s \geq 1$ vertices,
$$\kappa(S_n(F_1),\ldots,S_n(F_s)) =  \sum_{\substack{ \phi : V \to \lle 1,n\rre \text{ such that} \\ H_{\pi(\phi)} \text{ is a hypertree}}} \kappa(A_{\phi_1}(F_1),\ldots,A_{\phi_s}(F_s)) + O(n^{k_1+\cdots+k_s-s}).$$
\end{lemma}

Now that we have isolated the leading term, starting from a hypertree $H$ on $\lle 1,s\rre$, the reconstruction of the corresponding set partitions $\pi$ and maps $\phi$ is easier to describe. The set partitions $\pi$ such that $H_\pi=H$ are provided by the following procedure:
\begin{enumerate}[label=(SP\arabic*)]
    \item\label{enum:reconstruction_1} For any vertex $r \in \lle 1,s\rre$ connected to hyperedges $e_{i_1},e_{i_2},\ldots,e_{i_c}$ of $H$, choose \emph{distinct} elements $a_{r,i_1}\neq a_{r,i_2} \neq \cdots \neq a_{r,i_c} \in \lle 1,k_r\rre$. 
    \item\label{enum:reconstruction_2} The set partition $\pi$ of $V$ consists then in the parts $\{(r,a_{r,i}),\,\,r \in e_i\}$ for $e_i \in E_{H}$, plus singletons.
\end{enumerate}
On the other hand, given $\pi \in \pym(V)$ a set partition such that $H_\pi$ is a hypertree, we claim that the $n^{\downarrow k_1+\cdots + k_s-(s-1)}$ maps $\phi=(\phi_1,\ldots,\phi_s)$ such that $\pi=\pi(\phi)$ all yield the same elementary joint cumulant
$$\kappa_\pi(F_1,\ldots,F_s)(\gamma)=\kappa(A_{\phi_1}(F_1),\ldots,A_{\phi_s}(F_s)),$$
and that this value is the evaluation of an observable in $\obs$ on the graphon $\gamma$. To prove this claim, note first that if $H_{\pi(\phi)}$ is a hypertree, then since it does not have loops, the maps $\phi_1,\ldots,\phi_r$ are all injective. Then,
$$
 \kappa(A_{\phi_1}(F_1),\ldots,A_{\phi_s}(F_s)) = \sum_{\theta \in \pym(s)} \mu(\theta)\,\prod_{j=1}^{\ell(\theta)} \esper\!\left[\prod_{r \in \theta_j}A_{\phi_r}(F_r)\right].
$$
Given a set partition $\theta \in \pym(s)$, we lift it to a set partition $\Theta \in \pym(V)$ with the same length, by taking the inverse images of the parts of $\theta$ by the surjective map $S : (r,a) \in V \mapsto r \in \lle 1,s\rre$. We claim that each joint moment $\esper[\prod_{r \in \theta_j}A_{\phi_r}(F_r)]$ is an observable $t(F,\gamma)$ for some graph $F$ which depends only on $\pi$ and on the graphs $F_r$, $r \in \theta_j$. In order to construct this graph $F$, we start from the disjoint union $\bigsqcup_{r \in \theta_j} F_r$, and for each part $\pi_i$ of the set partition $\pi$, we look at the intersection $\Theta_j \cap \pi_i$. If this intersection contains more than one element, say $(r_1,a_1),\ldots,(r_c,a_c)$, then we replace the disjoint graphs $F_{r_1},\ldots,F_{r_c}$ by their join
$$(F_{r_1}\join F_{r_2}\join \cdots \join F_{r_c})(a_1,a_2,\ldots,a_c),$$
in which the vertices $(r_1,a_1),\ldots,(r_c,a_c)$ have been identified. Notice that we may have to do several joins involving the same graph $F_r$, for instance if two parts $\pi_i$ and $\pi_{i'}$ contain elements $a_{r,i}$ and $a_{r,i'}$ of $V_{F_r}$ (by \ref{enum:reconstruction_1}, these elements are then distinct). The graph that we obtain at the end is 
$$F = \left(\bigsqcup_{r \in \theta_j} F_r\right) \downarrow (\Theta_j \wedge \pi),$$
where $\Theta_j \wedge \pi$ is the set partition of $\Theta_j$ induced by the parts of $\pi$. The same arguments as in the proof of Proposition \ref{prop:moment} show then that:
$$\kappa(A_{\phi_1}(F_1),\ldots,A_{\phi_s}(F_s)) = \sum_{\theta \in \pym(s)} \mu(\theta)\,\prod_{j=1}^{\ell(\theta)} \,t\!\left(\left(\bigsqcup_{r \in \theta_j} F_r\right) \downarrow (\Theta_j \wedge \pi),\gamma\right), $$
and the right-hand side is indeed the evaluation on $\gamma$ of an observable $\kappa_\pi(F_1,\ldots,F_s)$ which depends only on $\pi$ and the graphs $F_1,\ldots,F_s$. So, to summarise:

\begin{proposition}\label{prop:expansion_cumulant}
Let $F_1,\ldots,F_s$ be finite graphs with sizes $k_1,\ldots,k_{s\geq 2}$. We have
$$\frac{\kappa(S_n(F_1),S_n(F_2),\ldots,S_n(F_s))}{n^{k_1+\cdots+k_s-s+1}} = \kappa_s(F_1,F_2,\ldots,F_s)(\gamma) + O\!\left(\frac{1}{n}\right)$$
where
\begin{align*}
\kappa_s(F_1,F_2,\ldots,F_s) &=  \sum_{H \in \Hset(s)} \left(\sum_{\pi \in \pym(V)\,|\,H_\pi=H} \kappa_\pi(F_1,F_2,\ldots,F_s)\right);\\
\kappa_\pi(F_1,F_2,\ldots,F_s) &= \sum_{\theta \in \pym(s)} \mu(\theta)\,\prod_{j=1}^{\ell(\theta)} \,t\!\left(\left(\bigsqcup_{r \in \theta_j} F_r\right) \downarrow (\Theta_j \wedge \pi),\gamma\right),
\end{align*}
In the formula for $\kappa_\pi(F_1,F_2,\ldots,F_s)$, given a set partition $\theta \in \pym(s)$, $\Theta$ denotes the induced set partition of the set of all vertices $V=\bigsqcup_{r=1}^s V_{F_r}$.
The set partitions $\pi$ such that $H_\pi=H$ is a fixed hypertree can all be obtained thanks to the reconstruction procedure \ref{enum:reconstruction_1} + \ref{enum:reconstruction_2}.
\end{proposition}

\begin{example}
Suppose $s=2$. There is a unique hypertree on two vertices:
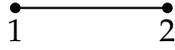
\vspace{3mm}

\begin{center}
\begin{tikzpicture}[scale=1]
\foreach \x in {(0,0),(2,0)}
\fill \x circle (2pt);
\draw [thick] (0,0) -- (2,0);
\draw (0,-0.3) node {$1$};
\draw (2,-0.3) node {$2$};
\end{tikzpicture}
\end{center}    
A set partition $\pi$ corresponding to this hypertree $H$ is obtained by choosing two indices $i_1 \in \lle 1,k_1\rre$ and $i_2 \in \lle 1,k_2\rre$. Then,
$$\kappa_\pi(F_1,F_2) = (F_1\join F_2)(i_1,i_2) - F_1 \, F_2.$$
We thus recover the formula for the asymptotics of the covariance.
\end{example}

\begin{example} Suppose $s=3$. There are four hypertrees on three vertices:
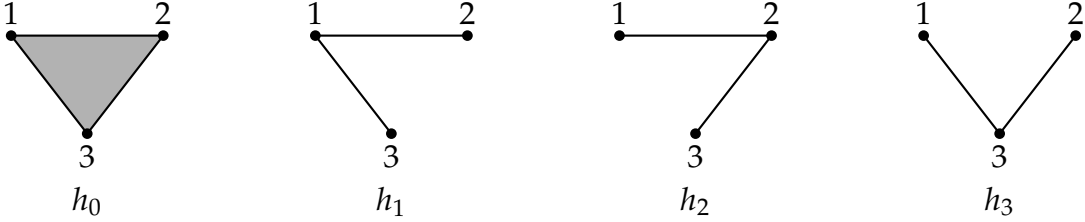
\vspace{3mm}

\begin{center}
\begin{tikzpicture}[scale=1]
\fill [white!70!black] (0,0) -- (2,0) -- (1,-1.3);
\foreach \x in {(0,0),(2,0),(1,-1.3)}
\fill \x circle (2pt);
\draw [thick] (0,0) -- (2,0) -- (1,-1.3) -- (0,0);
\draw (0,0.3) node {$1$};
\draw (1,-1.6) node {$3$};
\draw (2,0.3) node {$2$};
\draw (1,-2.2) node {$h_0$};
\begin{scope}[shift={(8,0)}]
\foreach \x in {(0,0),(2,0),(1,-1.3)}
\fill \x circle (2pt);
\draw [thick] (0,0) -- (2,0) -- (1,-1.3);
\draw (0,0.3) node {$1$};
\draw (1,-1.6) node {$3$};
\draw (2,0.3) node {$2$};
\draw (1,-2.2) node {$h_2$};
\end{scope}
\begin{scope}[shift={(12,0)}]
\foreach \x in {(0,0),(2,0),(1,-1.3)}
\fill \x circle (2pt);
\draw [thick] (2,0) -- (1,-1.3) -- (0,0);
\draw (0,0.3) node {$1$};
\draw (1,-1.6) node {$3$};
\draw (2,0.3) node {$2$};
\draw (1,-2.2) node {$h_3$};
\end{scope}
\begin{scope}[shift={(4,0)}]
\foreach \x in {(0,0),(2,0),(1,-1.3)}
\fill \x circle (2pt);
\draw [thick] (1,-1.3) -- (0,0) -- (2,0);
\draw (0,0.3) node {$1$};
\draw (1,-1.6) node {$3$};
\draw (2,0.3) node {$2$};
\draw (1,-2.2) node {$h_1$};
\end{scope}
\end{tikzpicture}
\end{center}    

\noindent For the first hypertree $h_0$, a corresponding set partition is obtained by choosing three indices $i_1 \in \lle 1,k_1\rre$, $i_2 \in \lle 1,k_2\rre$, $i_3 \in \lle 1,k_3\rre$. The corresponding term in $\kappa_3(F_1,F_2,F_3)$ is 
\begin{align*}
\kappa_\pi(F_1,F_2,F_3) &= (F_1\join F_2 \join F_3)(i_1,i_2,i_3) - (F_1\join F_2)(i_1,i_2) \, F_3 \\
&\quad- (F_2\join F_3)(i_2,i_3) \, F_1- (F_3\join F_1)(i_3,i_1) \, F_2 + 2\,F_1\, F_2 \, F_3.
\end{align*}
For the second hypertree $h_1$, 
a corresponding set partition is obtained by choosing four indices $i_1\neq j_1 \in \lle 1,k_1\rre$, $i_2 \in \lle 1,k_2\rre$, $i_3 \in \lle 1,k_3\rre$. The corresponding term is 
\begin{align*}
\kappa_\pi(F_1,F_2,F_3) &= (F_1\join F_2 \join F_3)(i_1,i_2;j_1,i_3)- (F_1\join F_2)(i_1,i_2) \, F_3  \\
&\quad - (F_1\join F_3)(j_1,i_3) \, F_1 + F_1\, F_2 \, F_3,
\end{align*}
where in the first term we identify the vertex $i_1$ of $F_1$ with the vertex $i_2$ of $F_2$, and the vertex $j_1$ of $F_1$ with the vertex $i_3$ of $F_3$.
The two remaining hypertrees give similar expressions, with the roles of $F_1$, $F_2$ and $F_3$ permuted cyclically; see the formula at the end of the proof of \cite{FMN20}*{Proposition 20}. These formulas were used in the introduction in order to compute $\kappa_3(F_1,F_1,F_1)$ when $F_1=\begin{tikzpicture}[baseline=-1mm,scale=0.5]
\foreach \x in {(0,0),(1,0),(0.5,0.5),(0.5,-0.5)}
\fill \x circle (3pt);
\draw (0,0) -- (0.5,0.5) -- (1,0) -- (0.5,-0.5) -- (0,0);
\draw (0,0) -- (1,0);
\end{tikzpicture}$.
\end{example}
\medskip

\subsection{Hidden equations satisfied by the singular graphons}
The maps $\kappa_s $ and $\kappa_\pi: \Gset^s \to \obs$ have a complicated combinatorial description, but they enable us to get more information on the singular graphons:

\begin{theorem}\label{thm:additional_equations}
Let $\gamma$ be a singular graphon. For any $s \geq 2$ and any finite graphs $F_1,\ldots,F_s$,
$$\kappa_s(F_1,\ldots,F_s)(\gamma)=0,$$
where $\kappa_s(F_1,\ldots,F_s)$ is the observable in $\obs$ defined by Proposition \ref{prop:expansion_cumulant}.
\end{theorem}

\begin{proof}
The case $s=2$ is the definition of singular graphons. This implies that the variance of a random variable $S_n(F)$ is a $O(n^{2k-2})$ if $F$ has $k$ vertices. In particular, by the Bienaymé--Chebyshev inequality, for any $\eps>0$, 
$$X_n(F)=\frac{S_n(F)-\esper[S_n(F)]}{n^{k-1+\eps}}$$
converges in probability to $0$. Let us then prove the result of the proposition by induction on $s\geq 2$.
Suppose the result true up to order $s-1\geq 2$, and let us fix some finite graphs $F_1,\ldots,F_s$. We consider the joint generating series
\begin{align*}
E_n(z_1,\ldots,z_s)&=\esper\!\left[\exp\!\left(\frac{z_1(S_n(F_1)-\esper[S_n(F_1)])+\cdots+z_s(S_n(F_s)-\esper[S_n(F_s)])}{n^{k_1+\cdots+k_s-(s-1)}}\right)\right]\\
&=\esper[\exp(z_1X_n(F_1)+\cdots + z_sX_n(F_s))],
\end{align*}
the parameter $\eps$ being here chosen equal to $\frac{1}{s}$. The function above is an entire function on $\C^s$, since the $S_n(F)$'s are bounded random variables. The logarithm of this generating series involves the joint cumulants:
\begin{align*}
L_n(z_1,\ldots,z_s) &= \log E_n(z_1,\ldots,z_s) \\
&=\sum_{r=2}^\infty \sum_{\substack{r_1,r_2,\ldots,r_s \geq 0 \\ r_1+r_2+\cdots+r_s=r}} \frac{(z_1)^{r_1}(z_2)^{r_2}\cdots (z_s)^{r_s}}{r_1!\,r_2!\cdots r_s!}\,\,\kappa\!\left((X_n(F_1))^{\otimes r_1},\ldots,(X_n(F_s))^{\otimes r_s}\right).
\end{align*}
Here by $(X_n(F_i))^{\otimes r_i}$ we mean that the variable $X_n(F_i)$ appears $r_i$ times in the joint cumulant. The sum starts at $r=2$ because the random variables $X_n(F_i)$ are centered, and it is well-defined on a polydisk of $\C^s$ which we are now going to describe. For $2\leq r\leq s-1$, by the induction hypothesis,
\begin{align*}
\kappa\!\left((X_n(F_1))^{\otimes r_1},\ldots,(X_n(F_s))^{\otimes r_s}\right) &= \frac{\kappa((S_n(F_1))^{\otimes r_1},\ldots,(S_n(F_s))^{\otimes r_s})}{n^{\frac{r}{s} + r_1(k_1-1)+\cdots + r_s(k_s-1)}} = O\!\left(n^{-\frac{r}{s}}\right)=O\!\left(n^{-\frac{2}{s}}\right),
\end{align*}
with a constant in the final $O(\cdot)$ that depends only on the graphs $F_1,\ldots,F_s$; we can take the same constant for any sequence $(r_1,\ldots,r_s)$ with sum smaller than $s-1$, since there is a finite set of such sequences. On the other hand, if $r \geq s$, then the upper bound \eqref{eq:bound_cumulant} ensures that
 $$\left|\kappa\!\left((X_n(F_1))^{\otimes r_1},\ldots,(X_n(F_s))^{\otimes r_s}\right)\right|\leq C^r\,r!\,n^{1 - \frac{r}{s}}$$
 for some constant $C$ with depends only on $k_1,\ldots,k_s$. Therefore, the logarithm $L_n(z_1,\ldots,z_s)$ is convergent on the polydisk $\Delta = \{(z_1,\ldots,z_s) \in \C^s\,|\,\forall i \in \lle 1,s\rre,\,\,|z_i|<\frac{1}{2Cs} \}$.
 Indeed, if we remove from the series the finite number of terms corresponding to sequences $(r_1,\ldots,r_s)$ whose sum $r$ is smaller than $s-1$, then what remains is a series which is absolutely bounded by
 \begin{align*}
 &\sum_{r=s}^\infty \sum_{r_1+\cdots + r_s =r} \frac{r!}{r_1!\cdots r_s!}\,(C|z_1|)^{r_1} \cdots (C|z_s|)^{r_s}\\
  &= \sum_{r=s}^\infty (C(|z_1|+\cdots+|z_s|))^r = \frac{(C(|z_1|+\cdots+|z_s|))^s}{1-C(|z_1|+\cdots+|z_s|)} \leq \frac{1}{2^{s+1}}.
 \end{align*}
 Proposition \ref{prop:expansion_cumulant} and the upper bounds on cumulants that we have written above also prove that the sequence of series $(L_n(z_1,\ldots,z_s))_{n \in \N}$ converges uniformly on the polydisk $\Delta$ towards the function
 $$L(z_1,\ldots,z_s)=\sum_{\substack{r_1,r_2,\cdots,r_s \geq 0 \\ r_1+r_2+\cdots+r_s=s}} \frac{(z_1)^{r_1}\cdots (z_s)^{r_s}}{r_1!\cdots r_s!}\,\kappa_s(F_1^{\otimes r_1},\ldots,F_s^{\otimes r_s})(\gamma).$$
 The same limiting result holds for $E_n(z_1,\ldots,z_s)$ and the exponential $E(z_1,\ldots,z_s)$ of the series $L(z_1,\ldots,z_s)$ written above. Hence, the random variables $X_n(F_1),\ldots,X_n(F_s)$ converge in joint moments. However, we also know that they converge in probability to $0$, so we conclude that $E(z_1,\ldots,z_s)=1$ and that $L(z_1,\ldots,z_s)=0$. In particular, the term corresponding to the sequence $r_1=r_2=\cdots=r_s=1$ in $L(z_1,\ldots,z_s)$ yields $\kappa_s(F_1,\ldots,F_s)(\gamma)=0$, whence the result for the order $s$.
\end{proof}
\medskip

Let us insist on the fact that without the strong upper bound \eqref{eq:bound_cumulant} stemming from the theory of dependency graphs developed in the framework of mod-Gaussian convergence, we would not have been able to prove the convergence in joint moments used in the proof of Theorem \ref{thm:additional_equations}. Thus, the upper bounds on cumulant established in \cite{FMN16}*{Section 9} are a crucial argument in order to obtain the hidden equations satisfied by the singular models.

\begin{remark}
In Section \ref{sec:convergence}, we shall establish a much better upper bound than \eqref{eq:bound_cumulant} on the joint cumulants of the random variables $S_n(F)$ when $\gamma$ is singular (Theorem \ref{thm:upper_bound}). This result will rely crucially on Theorem \ref{thm:additional_equations}, and therefore on Equation \eqref{eq:bound_cumulant}. Hence, the \emph{generic} upper bound on cumulants, which is valid for any graphon $\gamma$, is needed in order to prove the \emph{singular} upper bound on cumulants, which is only valid for a singular graphon.
\end{remark}

\begin{example}\label{ex:cutting_trees}
Just after the statement of our Conjecture, we explained that given a singular graphon with edge density $p$, the only equation for trees with size $4$ stemming directly from Formula \eqref{eq:singular_3} was
$$p^3 = \frac{2}{3}\,t(\begin{tikzpicture}[baseline=-1mm,scale=0.5]
\foreach \x in {(0,0),(1,0),(2,0),(3,0)}
\fill \x circle (4pt);
\draw (0,0) -- (3,0);
\end{tikzpicture},\,\gamma) + \frac{1}{3}\, t\!\left(\begin{tikzpicture}[baseline=2mm,scale=0.5]
\foreach \x in {(0,0),(1,0),(2,0),(1,1)}
\fill \x circle (4pt);
\draw (0,0) -- (2,0);
\draw (1,0) -- (1,1);
\end{tikzpicture},\,\gamma\right),$$
because of the vanishing of $\kappa_2(\edgegraph\,,\,\begin{tikzpicture}[baseline=-1mm,scale=0.5]
\foreach \x in {(0,0),(1,0),(2,0)}
\fill \x circle (4pt);
\draw (0,0) -- (2,0);
\end{tikzpicture})(\gamma)$. However, by Theorem \ref{thm:additional_equations}, we now also have $\kappa_3(\edgegraph\,,\edgegraph\,,\edgegraph)(\gamma)=0$, and the observable on the left-hand side of this equation is equal to
$$  24\,\,\begin{tikzpicture}[baseline=-1mm,scale=0.5]
\foreach \x in {(0,0),(1,0),(2,0),(3,0)}
\fill \x circle (4pt);
\draw (0,0) -- (3,0);
\end{tikzpicture}\,+\,8\,\begin{tikzpicture}[baseline=2mm,scale=0.5]
\foreach \x in {(0,0),(1,0),(2,0),(1,1)}
\fill \x circle (4pt);
\draw (0,0) -- (2,0);
\draw (1,0) -- (1,1);
\end{tikzpicture} \,- 72\, \begin{tikzpicture}[baseline=2mm,scale=0.5]
\foreach \x in {(0,0),(1,0),(2,0),(0.5,1),(1.5,1)}
\fill \x circle (4pt);
\draw (0,0) -- (2,0);
\draw (0.5,1) -- (1.5,1);
\end{tikzpicture} + 40\, (\edgegraph)^3 .$$
We already know that when evaluated on a singular graphon with edge density $p$, the two last terms yield a contribution $-72p^3+40p^3=-32p^3$, so we obtain
$$p^3 = \frac{3}{4} \,t(\begin{tikzpicture}[scale=1,baseline=-1mm]
\foreach \x in {(0,0),(0.5,0),(1,0),(1.5,0)}
\fill \x circle (2pt);
\draw (0,0) -- (1.5,0);
\end{tikzpicture},\gamma)+\frac{1}{4}\,t\!\left(\begin{tikzpicture}[scale=1,baseline=2mm]
\foreach \x in {(0,0),(0.5,0),(0.5,0.5),(1,0)}
\fill \x circle (2pt);
\draw (0,0) -- (1,0);
\draw (0.5,0) -- (0.5,0.5);
\end{tikzpicture},\,\gamma\right).$$
This equation is independent from the previous one, so the two trees with $3$ edges both have density $p^3$ in the singular graphon $\gamma$. This argument makes one wonder if it is possible to use the equations from Theorem \ref{thm:additional_equations} in order to prove that $t(T,\gamma)=p^{|E_T|}$ for any tree $T$ (we shall see in the next Section another much simpler argument which proves this result). It turns out that these equations determine the tree observables of a singular graphon for all the trees of size $k\leq 11$. For the trees of size $12$, there are $551$ such trees (up to isomorphism), and the system of equations which one can write by using Theorem \ref{thm:additional_equations} and the identities $t(T,\gamma)=p^{|E_T|}$ for $|E_T|\leq 10$ is only of rank $550$. One can overcome this obstruction either by using Sidorenko's inequality (see Remark \ref{rem:sidorenko}), or by using the argument on join-transitive graphs which we are now going to present.
\end{example}
\bigskip

\section{Join-transitive subgraphs and spectral properties of the singular graphons}\label{sec:spectral}

In Subsection \ref{sub:factorisation}, we establish a set of equations satisfied in the singular case by the observables corresponding to the join-transitive motives (see Definition \ref{def:join_transitive} and Theorem \ref{thm:simple_equations}). This class of motives includes the trees; therefore, one is able to prove that a singular graphon with edge density $p$ has the same tree observables as the Erd\H{o}s--Rényi model $\gamma_p$. By combining this observation with a form of the matrix-tree theorem, we obtain in Subsection \ref{sub:determinant_laplacian} a formula for the expectation of the determinant of the diagonally modified Laplacian of the $W$-random graph $G_n(\gamma)$ associated to a singular graphon $\gamma$. Then, in Subsection \ref{sub:hilbert_schmidt}, we examine in full generality the behavior of the characteristic polynomial of the adjacency matrix of a $W$-random graph. We are not able to relate the two expected characteristic polynomials
$$\esper\!\left[\det\!\left(I_n+z\,\frac{A_n(\gamma)}{n}\right)\right]$$
and 
$$\left.\esper\!\left[\det\!\left((1+pz)I_n + z\,\diag(\eps_1,\ldots,\eps_n) - z\,\frac{L_n(\gamma)}{n}\right)\right]\right|_{\eps_1=\eps_1(\gamma),\ldots,\eps_n=\eps_n(\gamma)},$$
but at the end of this section we provide an argument which opens the way to a possible spectral solution of our Conjecture. 
\medskip

\subsection{Factorisation of the join-transitive densities}\label{sub:factorisation}
Recall that a (finite) graph $F$ is said to be \emph{transitive} if, for any pair of vertices $(i,j) \in (V_F)^2$, there exists a graph automorphism $\psi : V_F \to V_F$ such that $\psi(i)=j$.

\begin{definition}\label{def:join_transitive}
We say that a finite graph $F$ is join-transitive if there exists a family of transitive graphs $F_1,\ldots,F_r$ such that $F$ can be obtained recursively by junction of these transitive motives. Hence, there exists $(i_1,j_1) \in V_{F_1} \times V_{F_2}$, $(i_2,j_2) \in V_{(F_1 \join F_2)(i_1,j_1)}\times V_{F_3}$, \emph{etc.} such that
$$F = ((F_1 \join F_2(i_1,j_1))\join F_3(i_2,j_2))\join \cdots  \join F_r(i_{r-1},j_{r-1}).$$
\end{definition}

\begin{example}
The graph 
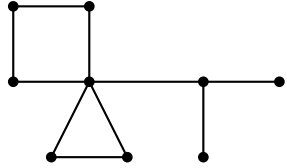
\begin{center}
\begin{tikzpicture}[scale=1]
\foreach \x in {(0,0),(0,1),(-1,0),(-1,1),(1.5,0),(2.5,0),(1.5,-1),(-0.5,-1),(0.5,-1)}
\fill \x circle (2pt);
\draw [thick] (0,0) -- (-0.5,-1) -- (0.5,-1) -- (0,0) -- (-1,0) -- (-1,1) -- (0,1) -- (0,0) -- (1.5,0) -- (2.5,0);
\draw [thick] (1.5,0) -- (1.5,-1);
\end{tikzpicture}
\end{center}

\noindent is join-transitive, as it can be obtained by recursive junction of edges, of the triangle graph and of the square graph. On the other hand, the diamond-shaped graph $F_1=\begin{tikzpicture}[baseline=-1mm,scale=0.5]
\foreach \x in {(0,0),(1,0),(0.5,0.5),(0.5,-0.5)}
\fill \x circle (3pt);
\draw (0,0) -- (0.5,0.5) -- (1,0) -- (0.5,-0.5) -- (0,0);
\draw (0,0) -- (1,0);
\end{tikzpicture}$ is not join-transitive, since it is not transitive (some vertices have degree $2$ and some vertices have degree $3$) and it is not a join of smaller graphs. It is actually the smallest graph which is not join-transitive.
\end{example}

In the sequel, given a join-transitive motive $F$, we write $F=F_1\join F_2 \join \cdots \join F_r$ if $F$ can be obtained by recursive junction of the transitive motives $F_1,\ldots,F_r$ as in Definition \ref{def:join_transitive}; thus, we omit the indices of the vertices of junction.

\begin{theorem}\label{thm:simple_equations}
Let $\gamma$ be a singular graphon. For any join-transitive motive $F=F_1\join F_2 \join \cdots \join F_r$, 
$$t(F,\gamma) = \prod_{i=1}^r t(F_i,\gamma).$$
\end{theorem}

In order to prove this result, let us introduce the notion of density of a marked graph in a graphon $\gamma=[g]$. Let $F$ be a motive with vertex set $V_F=\lle 1,k\rre$, and $i$ be an index in $\lle 1,k\rre$. The \emph{marked density} $t(F^{\bullet i},g)$ is defined by
$$t(F^{\bullet i},g,x_i) = \int_{[0,1]^{k-1}} \left(\prod_{\{j_1,j_2\} \in E_T} g(x_{j_1},x_{j_2}) \right)\DD{x_1}\cdots \DD{x_{i-1}}\DD{x_{i+1}}\cdots \DD{x_k};$$
hence, we consider the same function as in the definition of $t(F,\gamma)$, and we integrate all the variables except $x_i$. The function which one obtains is in $\leb^\infty([0,1],\DD{x})$ and it satisfies
$$\int_{0}^1 t(F^{\bullet i},g,x_i) \DD{x_i} = t(F,\gamma).$$
Notice that $t(F^{\bullet i},g,x_i)$ usually depends on the choice of a graph function $g$ representative of the graphon $\gamma$.
\begin{lemma}
 Let $F^{\bullet i}$ be a marked transitive graph, and $\gamma$ be a singular graphon. For any graph function $g$ in $\gamma$, the marked density
 $t(F^{\bullet i},g,x)$:
 \begin{itemize}
     \item does not depend on the marked index $i \in V_F$;
     \item does not depend on the choice of $g \in \gamma$;
     \item is almost everywhere constant as a function on $[0,1]$.
 \end{itemize}
 \end{lemma} 

\begin{proof}
Since $F$ is transitive, for any indices $i \neq j$, there exists a graph automorphism $\psi : V_F \to V_F$ such that $\psi(i) = j$. Then,
\begin{align*}
&t(F^{\bullet j},g,x) \\
&= \int_{[0,1]^{k-1}} \left(\prod_{\substack{\{l_1,l_2\} \in E_T\\ l_1,l_2 \neq j}} g(x_{l_1},x_{l_2})\prod_{l_3\,|\,\{j,l_3\} \in E_T} g(x,x_{l_3})\right) \DD{x_1}\cdots \DD{x_{j-1}}\DD{x_{j+1}}\cdots \DD{x_k} \\
&= \int_{[0,1]^{k-1}} \left(\prod_{\substack{\{l_1,l_2\} \in E_T\\ l_1,l_2 \neq i}} g(x_{\psi(l_1)},x_{\psi(l_2)}) \prod_{l_3\,|\,\{i,l_3\} \in E_T} g(x,x_{\psi(l_3)})\right) \DD{x_{\psi(1)}}\cdots \DD{x_{\psi(i-1)}}\DD{x_{\psi(i+1)}}\cdots \DD{x_{\psi(k)}}\\
&=\int_{[0,1]^{k-1}} \left(\prod_{\substack{\{l_1,l_2\} \in E_T\\ l_1,l_2 \neq i}} g(y_{l_1},y_{l_2}) \prod_{l_3\,|\,\{i,l_3\} \in E_T} g(x,y_{l_3})\right) \DD{y_{1}}\cdots \DD{y_{i-1}}\DD{y_{i+1}}\cdots \DD{y_{k}}\\
&= t(F^{\bullet i},g,x).
\end{align*}
The second item of the lemma is a consequence of the third item, because if $t(F^{\bullet i},g,x)$ is a constant, then it is equal to $\int_0^1 t(F^{\bullet i},g,x)\DD{x} = t(F,\gamma)$, which does not depend on the choice of a representative $g$. To establish this third item, notice that if $F$ is transitive, then a join $(F \join F)(i,j)$ does not depend on the pair of indices $(i,j) \in (V_F)^2$. Indeed, if $(i',j')$ is another pair, then there exists two graph automorphisms $\psi$ and $\phi$ of $F$ such that $\psi(i)=i'$ and $\phi(j) = j'$, and these automorphisms can be combined in order to build an isomorphism of graphs between $(F\join F)(i,j)$ and $(F \join F)(i',j')$. So,
$$\kappa_2(F,F) = |V_F|^2\,\big((F\join F)(1,1) - F^2).$$
Consider the random variable $t(F^{\bullet 1},g,X)$ with $X$ chosen uniformly in the interval $[0,1]$. We have $\esper[t(F^{\bullet 1},g,X)] = t(F,\gamma)$, and
\begin{align*}
\esper&[(t(F^{\bullet 1},g,X))^2] = \int_0^1 (t(F^{\bullet 1},g,x))^2\DD{x}\\
&=\int_{[0,1]^{2k-1}}  \left(\prod_{\{j_1,j_2\} \in E_T} g(x_{j_1},x_{j_2}) \prod_{\{l_1,l_2\} \in E_T} g(y_{l_1},y_{l_2})\right)\DD{x} \DD{x_2}\cdots \DD{x_k}\DD{y_2} \cdots \DD{y_k}
\end{align*}
with $x_1=y_1=x$ in this integral over $2k-1$ variables. This is the formula for the graph density $t((F\join F)(1,1),\gamma)$. As $\kappa_2(F,F)(\gamma)=0$, we conclude that the random variable $t(F^{\bullet 1},g,X)$ has variance $0$, hence is almost everywhere constant.
\end{proof}

\begin{proof}[Proof of Theorem \ref{thm:simple_equations}]
By induction, it suffices to prove that if $\gamma=[g]$ is a singular graphon, $F$ is an arbitrary motive and $G$ is a transitive motive, then $t((F\join G)(i,j),\gamma)=t(F,\gamma)\,t(G,\gamma)$ for any pair of indices $(i,j) \in V_F \times V_G$. However, this is a trivial consequence of the fact that $t(G^{\bullet j},g,x)$ is constant equal to $t(G,\gamma)$:
\begin{align*}
t((F\join G)(i,j),\gamma) &= \int_{0}^1 t(F^{\bullet i},g,x)\,t(G^{\bullet j},g,x) \DD{x} \\
&= t(G,\gamma)\,\int_{0}^1 t(F^{\bullet i},g,x) \DD{x}  = t(F,\gamma)\,t(G,\gamma).\qedhere
\end{align*}
\end{proof}

\begin{corollary}\label{cor:tree_densities}
Let $\gamma$ be a singular graphon with edge density $p$. For any tree of forest $F$, $t(F,\gamma)=p^{|E_F|}$.
\end{corollary}

\begin{proof}
The trees are obtained recursively by junction of the transitive motive $\edgegraph$, so we can factorise over the edges any tree or forest density of a singular graphon.
\end{proof}

\begin{remark}\label{rem:sidorenko}
Corollary \ref{cor:tree_densities} can be understood as a minimality property of the tree observables of the singular graphons. Indeed, a theorem due to Sidorenko states that for any graphon $\gamma=[g]$ with edge density $t(\edgegraph,\gamma) = \iint_{[0,1]^2} g(x,y) \DD{x}\DD{y} = p$, and for any tree $T$, we have
$$t(T,\gamma)\geq p^{|E_T|}.$$
This result is a particular case of \cite{Sid92}*{Corollary 1}; the inequality holds for many other graphs, including the even cycles and the complete bipartite graphs. It is actually conjectured that any bipartite graph $F$ satisfies the inequality $t(F,\gamma) \geq (t(\edgegraph,\gamma))^{|E_F|}$ for any graphon $\gamma$; see \cites{Hat10,CFS10,KLL16,CKLL18} for recent developments around this conjecture. 
\end{remark}
\medskip

\subsection{Determinants of the diagonally modified Laplacian operators}\label{sub:determinant_laplacian}
Let us now investigate the consequences of Theorem \ref{thm:simple_equations} and its Corollary \ref{cor:tree_densities} in terms of the spectral properties of the singular graphons. If $G$ is a finite graph with vertex set $V_G=\lle 1,n\rre$ and with adjacency matrix $A_G = (1_{(i \sim_G j)})_{1\leq i,j\leq n}$, we recall that its Laplacian matrix is 
$$L_G = D_G-A_G = \diag(d_1,\ldots,d_n) - A_G,$$
where $d_i = \deg(i,G)=\sum_{j \neq i}1_{i\sim j}$ is the degree of the vertex $i$ if $G$. The characteristic polynomial of this matrix is related to the \emph{spanning forests} $F$ of $G$: they are the subgraphs of $G$ with the same vertex set, without cycle and with $E_F \subset E_G$. A \emph{spanning tree} of $G$ is a connected spanning forest; it exists if and only if $G$ is connected. The spanning trees of (hyper)graphs shall play an important role in the next section. Here, we shall use the following identity: for any graph $G$ on $n$ vertices and diagonal matrix $\Delta_n = \diag(\delta_1,\delta_2,\ldots,\delta_n)$, we have:
\begin{equation}
\det(\Delta_n+L_G) = \sum_{F \text{ spanning forest of }G} \left(\prod_{T \subset F} \left(\sum_{x \in T} \delta_x\right)\right). \label{eq:MT1}\tag{MT1}
\end{equation}
This generalisation of Kirchhoff's matrix-tree theorem is proved in the appendix at the end of the paper. It implies the following combinatorial formula for the number of rooted spanning forests of the complete graph $K_n$ with $k \geq 1$ connected components:
\begin{equation}
\text{number of families of $k$ disjoint rooted trees covering a set of $n$ vertices} = n^{n-k}\,\binom{n-1}{k-1}.\label{eq:MT2}\tag{MT2}
\end{equation}
In the following, we fix a singular graphon $\gamma$ with edge density $p$, and we denote $L_n(\gamma)$ the Laplacian matrix of the $W$-random graph $G_n(\gamma)$. Given $z \in \C$, we then set
$$M_n(z,\gamma;\eps_1,\ldots,\eps_n) = (1+pz)I_n +z\,\diag(\eps_1,\ldots,\eps_n)- z\,\frac{L_n(\gamma)}{n},$$
where the $\eps_i$'s are arbitrary real numbers. A key observation is that if 
$$\eps_i = \eps_i(\gamma) = \frac{\deg(i,G_n(\gamma))}{n}-p,$$
then $M_n(z,\gamma;\eps_1(\gamma),\ldots,\eps_n(\gamma)) = I_n + z\,\frac{A_n(\gamma)}{n}$, where $A_n(\gamma)$ is the adjacency matrix of $G_n(\gamma)$. For a singular graphon with edge density $p$, the parameters $\eps_i(\gamma)$ are expected to be of small size; see Proposition \ref{prop:concentration_epsi}. So, $I_n+z\,\frac{A_n(\gamma)}{n}$ is a small random perturbation of the matrix $M_n(z,\gamma;0,\ldots,0)$. On the other hand, as we shall explain in the next paragraph, the evaluation of the characteristic polynomials of the matrices $\frac{A_n(\gamma)}{n}$ provides us with a way to prove that a graphon $\gamma$ is equal to a constant graphon $\gamma_p$. Therefore, it is natural ot try to compute the (expectations of the) determinants of the matrices $M_n(z,\gamma;\eps_1,\ldots,\eps_n)$ for arbitrary parameters $\eps_i \in \R$. The combinatorial formulas \eqref{eq:MT1} and \eqref{eq:MT2} lead readily to the following:

\begin{proposition}\label{prop:expectation_laplacian}
Consider a singular graphon $\gamma$ with edge density $p$. For any real parameters $\eps_1,\ldots,\eps_n$, the determinant
$$D\!L_n(z,\gamma;\eps_1,\ldots,\eps_n) = \det\!\left((1+pz)I_n + z\,\diag(\eps_1,\ldots,\eps_n) - z\,\frac{L_n(\gamma)}{n}\right)$$
has for expectation
$$\esper[D\!L_n(z,\gamma;\eps_1,\ldots,\eps_n)] = \left(\prod_{i=1}^n (1+\eps_i z)\right)\left(\frac{1}{n}\sum_{i=1}^n \frac{1+pz+\eps_i z}{1+\eps_i z}\right).$$
\end{proposition}

\begin{proof}
Set $\delta_i = -n(z^{-1}+p+\eps_i)$. Equation \eqref{eq:MT1} can be rewritten as a sum over \emph{rooted} forests $F \subset G$, where by rooted forest we mean that we choose a root for each tree $T \subset F$:
$$\det(\Delta_n+L_G) = \sum_{\text{rooted forest }F \subset G} \left(\prod_{(T,r) \subset F} \delta_r\right).$$
With $G = G_n(\gamma)$ and $\gamma$ singular graphon with edge density $p$, we therefore have:
\begin{align*}
\det(\Delta_n+L_n(\gamma)) &= \sum_{k=1}^n\sum_{\substack{F =(T_1,i_1) \sqcup \cdots \sqcup (T_k,i_k) \\ F \text{ rooted forest on $n$ vertices}}} 1_{F \subset G_n(\gamma)}\, \delta_{i_1}\cdots \delta_{i_k};\\
\esper[\det(\Delta_n+L_n(\gamma))] &= \sum_{k=1}^n\sum_{\substack{F =(T_1,i_1) \sqcup \cdots \sqcup (T_k,i_k) \\ F \text{ rooted forest on $n$ vertices}}} \proba[F \subset G_n(\gamma)]\, \delta_{i_1}\cdots \delta_{i_k} \\
&= \sum_{k=1}^n p^{n-k}\sum_{\substack{F =(T_1,i_1) \sqcup \cdots \sqcup (T_k,i_k) \\ F \text{ rooted forest on $n$ vertices}}} \,\delta_{i_1}\cdots \delta_{i_k}.
\end{align*}
By Equation \eqref{eq:MT2}, the total number of rooted forests with $k$ connected components on $n$ vertices is $n^{n-k}\,\binom{n-1}{k-1}$. Since the roots $i_1,\ldots,i_k$ play a symmetric role, the number of rooted forests with these $k$ roots is:
$$ n^{n-k}\,\frac{\binom{n-1}{k-1}}{\binom{n}{k}} = k\,n^{n-k-1} ,$$
and summing over the possible sets of roots $\{i_1,\ldots,i_k\}$ yields the elementary symmetric function $e_k(\delta_{1},\ldots,\delta_n)$. Thus,
$$\esper[\det(\Delta_n+L_n(\gamma)) ] =  \frac{1}{n}\sum_{k=1}^n k \,(np)^{n-k}\,e_k(\delta_1,\ldots,\delta_n) $$
for any singular graphon with edge density $p$. The rescaled determinant which we are interested in is 
$$D\!L_n(z,\gamma;\eps_1,\ldots,\eps_n) = \left(-\frac{z}{n}\right)^n\,\det(\Delta_n + L_n(\gamma)),$$
so
$$F_n(z,\gamma;\eps_1,\ldots,\eps_n) = \esper[D\!L_n(z,\gamma;\eps_1,\ldots,\eps_n)] = \frac{1}{n}\sum_{k=1}^n k\, (-pz)^{n-k}\,e_k(\overline{\delta}_1,\ldots,\overline{\delta}_n)  $$
with $\overline{\delta}_i = -\frac{z}{n}\,\delta_i = 1+pz+\eps_i z$. Recall that $\prod_{i=1}^n (1+t\overline{\delta}_i ) = 1+\sum_{k=1}^n t^k\,e_k(\overline{\delta}_1,\ldots,\overline{\delta}_n)$. By taking the derivative and setting $t=-\frac{1}{pz}$, we get:
$$F_n(z,\gamma;\eps_1,\ldots,\eps_n) = \left(\prod_{i=1}^n (\overline{\delta}_i-pz)\right)\left(\frac{1}{n}\sum_{i=1}^n \frac{\overline{\delta}_i}{\overline{\delta}_i-pz}\right),$$
which is the formula stated by the proposition.
\end{proof}

\begin{proposition}\label{prop:concentration_epsi}
Let $\gamma$ be a singular graphon with edge density $p$. For any $n \geq 2$ and any $x \geq 0$,
$$\proba\!\left[\max_{i \in \lle 1,n\rre} |\eps_i(\gamma)| \geq x\sqrt{\frac{\log n}{n}}\,\right] \leq 4\,n^{1-\frac{x^2}{8}}.$$
\end{proposition}

\begin{proof}
We shall compute an upper bound for $\proba[|\eps_1(\gamma)| \geq t]$, and then take $n$ times this upper bound in order to control $\max_{i \in \lle 1,n\rre} |\eps_i(\gamma)|$. Let us fix a representative $g$ of the singular graphon $\gamma$. We consider independent uniform random variables $X_{i \in \lle 1,n\rre}$, such that conditionally to $(X_1,\ldots,X_n)$, the vertices $i$ and $j$ in $\lle 1,n\rre$ are connected according to independent Bernoulli variables with parameter $g(X_i,X_j)$. By Hoeffding's inequality for sums of independent bounded random variables \cite{Hoeff63}*{Theorem 1, Equation (2.3)},
$$\proba\!\left[\left|\deg(1,G_n(\gamma)) - \sum_{j=2}^n g(X_1,X_j) \right|\geq t\,\,\bigg|\,\,(X_1,\ldots,X_n)\right] \leq 2\,\E^{-\frac{2t^2}{n-1}}.$$ 
By using the same inequality, we also obtain:
$$\proba\!\left[\left|\sum_{j=2}^n g(X_1,X_j) - (n-1) \int_{0}^1 g(X_1,x)\DD{x}\right|\geq t\,\,\bigg|\,\,X_1\right] \leq 2\,\E^{-\frac{2t^2}{n-1}}.$$
Therefore,
$$\proba\!\left[\left|\deg(1,G_n(\gamma)) - (n-1) \int_{0}^1 g(X_1,x)\DD{x}\right|\geq 2t\,\, \bigg|\,\,X_1\right] \leq 4\,\E^{-\frac{2t^2}{n-1}}.$$
However, $\int_0^1 g(X_1,x)\DD{x}$ is the random variable $t(F^{\bullet 1},g,X_1)$ which we considered in the proof of Theorem \ref{thm:simple_equations}, with $F= \edgegraph$. So, it is almost surely constant to $p = t(\edgegraph,\gamma)$, and
\begin{align*}
\proba[|\deg(1,G_n(\gamma)) - (n-1) \,p|\geq 2t] &\leq 4\,\E^{-\frac{2t^2}{n-1}} \leq 4\,\E^{-\frac{2t^2}{n}};\\
\proba\!\left[|\eps_1(\gamma)| \geq \frac{u}{\sqrt{n}} + \frac{p}{n}\right] &\leq 4\,\E^{-\frac{u^2}{2}}.
\end{align*}
If $u \leq \sqrt{\log 8}$, then the right-hand side is larger than $1$. Otherwise, $\frac{u}{\sqrt{n}} + \frac{p}{n} \leq \frac{2u}{\sqrt{n}}$ if $n \geq 2$, so we conclude that for any $u>0$,
$$\proba\!\left[|\eps_1(\gamma)| \geq \frac{2u}{\sqrt{n}} \right] \leq 4\,\E^{-\frac{u^2}{2}}.$$
Setting $u = \frac{x}{2}\sqrt{\log n}$ yields the claimed result.
\end{proof}
\medskip

\subsection{Characteristic polynomials of the adjacency matrices} \label{sub:hilbert_schmidt}
The following result explains why we are interested in the spectral properties of the $W$-random graphs attached to singular graphons:

\begin{theorem}\label{thm:spectral_constant_graphon}
Let $\gamma$ be a graphon with edge density $p = t(\edgegraph,\gamma)$ such that we have the convergence in probability
$$\det\!\left(I_n+z\,\frac{A_n(\gamma)}{n}\right) \to_{\proba,\,n \to \infty} (1+pz)\,\E^{-pz-\frac{p(1-p)z^2}{2}},$$
or the convergence in expectation
$$\esper\!\left[\det\!\left(I_n+z\,\frac{A_n(\gamma)}{n}\right)\right] \to_{n \to \infty} (1+pz)\,\E^{-pz-\frac{p(1-p)z^2}{2}}.$$
Then, $\gamma=\gamma_p$.
\end{theorem}

The end of this section is devoted to a proof of this criterion, which relies on general considerations from operator theory. In the following, the $\leb^2$ spaces that we consider are all taken with respect to the Lebesgue measure. Given a graph function $g$, we can associate to it an integral operator on functions $f : [0,1] \to \R$:
$$
(T_g f)(x) = \int_{0}^1 g(x,y)\,f(y)\DD{y}.
$$
If $f \in \leb^2([0,1])$, then $T_gf$ is also in $\leb^2([0,1])$, since
\begin{align*}
\|T_gf\|^2 &= \int_{[0,1]} (T_gf(x))^2\DD{x} = \int_{[0,1]^3} g(x,y)\,g(x,z)\,f(x)\,f(z)\DD{x}\DD{y}\DD{z} \\
&\leq \int_{[0,1]^3} g(x,y)\,g(x,z)\, \frac{(f(x))^2 + (f(z))^2}{2}  \DD{x}\DD{y}\DD{z} \\
&\leq \int_{[0,1]^3} \frac{(f(x))^2 + (f(z))^2}{2}  \DD{x}\DD{y}\DD{z} = \|f\|^2.
\end{align*}
Moreover, since $g \in \leb^2([0,1]^2)$, $T_g$ is actually a Hilbert--Schmidt operator; see for instance \cite{BS05}*{Theorem 2.11}. Hence, there exists a countable family of non-zero eigenvalues $(\lambda_i)_{i \in I}$, and a corresponding orthonormal family $(f_i)_{i \in I}$ of eigenfunctions in $\leb^2([0,1])$, such that
$$T_g = \sum_{i\in I} \lambda_i \,|f_i \rangle \,\langle f_i|.$$
Equivalently, the graph function $g$ is the limit in $\leb^2([0,1]^2)$ of the series $\sum_{i \in I} \lambda_i\,f_i(x)\,f_i(y)$, with $\sum_{i \in I} (\lambda_i)^2 < +\infty$ and $\scal{f_i}{f_j}_{\leb^2([0,1])} = 1_{(i=j)}$. For the general theory of Hilbert--Schmidt integral operators, see for instance \cite{Lax02}*{Chapter 30}.
Though this is not entirely evident, the discrete spectrum $\Spec(T_g) = \{\lambda_i\}_{i \in I}$ only depends on the graphon $\gamma=[g]$:

\begin{lemma}
Let $\gamma \in \Gspa$. The spectrum of an operator $T_g$ with $g \in \gamma$ only depends on $\gamma$.
\end{lemma}

\begin{proof}
Note that if $g' = g^\sigma$ is the conjugate of a graph function by a Lebesgue isomorphism $\sigma$, then the corresponding integral operators $T_{g'}$ and $T_g$ are conjugated by the relation $T_g' = U_\sigma\,T_g \,U_{\sigma^{-1}}$, where $U_\sigma$ is the isometry of $\leb^2([0,1])$ defined by $f \mapsto f\circ \sigma$. Therefore, $\Spec(T_{g'}) = \Spec(T_g)$. However, this is not sufficient, because of the following subtlety: two graph functions $g$ and $g'$ in the same equivalence class $\gamma \in \Gspa$ are not necessarily conjugated by a Lebesgue isomorphism. Instead, there exists a sequence of Lebesgue isomorphisms $(\sigma_n)_{n \in \N}$ such that $g^{\sigma_n} \to_{n \to \infty,\,\oblong} g'$. Let us see why this still implies that $\Spec(T_g)=\Spec(T_{g'})$. For $k \geq 3$, denote $C_k$ the cycle on $k$ vertices. By \cite{Lov12}*{Lemma 10.22}, $|t(C_k,g^{\sigma_n})-t(C_k,g')| \leq k\,\|g^{\sigma_n}-g\|_\oblong$. However, if $g(x,y) = \sum_{i \in I} \lambda_i\,f_i(x)\,f_i(y)$, then
\begin{align*}
t(C_k,g) &=  \sum_{i_1,\ldots,i_k \in I} \lambda_{i_1}\cdots \lambda_{i_k} \int_{[0,1]^k} f_{i_1}(x_1)\,f_{i_1}(x_2)\,f_{i_2}(x_2)\,f_{i_2}(x_3) \cdots f_{i_k}(x_k)\,f_{i_k}(x_1) \DD{x_1}\cdots \DD{x_k} \\
&= \sum_{i \in I} (\lambda_i)^k = \tr((T_g)^k);
\end{align*}
see \cite{Lov12}*{Equation (7.23)}. As a consequence, the two spectra $\Spec(T_g)=\Spec(T_{g^{\sigma_n}}) = \{\lambda_i\}_{i \in I}$ and $\Spec(T_{g'}) = \{\lambda_j'\}_{j \in J}$ satisfy:
$$\sum_{i \in I}(\lambda_i)^k = \sum_{j \in J} (\lambda_j')^k\quad \text{for any } k \geq 3.$$ 
These identities imply that the spectra of $T_g$ and $T_{g'}$ are the same; see for instance \cite{Lov12}*{Proposition A.21}.
\end{proof}
\medskip

A general principle from random matrix theory is that the scaled random matrices $\frac{A_n(\gamma)}{n}$ associated to the $W$-random graphs $G_n(\gamma)$ should be good (random) approximations of the infinite-dimensional operator $T_g$; see for instance \cite{GK00}. In particular, one could expect that the characteristic polynomial of $\frac{A_n(\gamma)}{n}$ converges towards the \emph{Fredholm determinant}
$$\det(\id + zT_g) = \prod_{i \in I} (1+z\lambda_i)$$
of the graphon $\gamma$. However, there are complications related to the fact that the Fredholm determinant makes sense only if $T_g$ is a trace-class operator. Unfortunately, in general we cannot make this assumption: given a graphon $\gamma=[g]$, we only know \emph{a priori} that $T_g$ is Hilbert--Schmidt. A way to overcome this problem is to consider the modified Fredholm determinants
\begin{align*}
\fredholm(\gamma,z) &= \fredholm(\id + zT_g) = \prod_{i\in I} (1+z\lambda_i)\,\E^{-z\lambda_i} ;\\
\ffredholm(\gamma,z) &= \ffredholm(\id + zT_g) = \prod_{i\in I} (1+z\lambda_i)\,\E^{-z\lambda_i+\frac{z^2(\lambda_i)^2}{2}} .
\end{align*}
These infinite products converge for any Hilbert--Schmidt operator, and we have
\begin{align*}
\fredholm(\id+zT_g) &= \det(\id + zT_g)\,\E^{-z\,\tr(T_g)};\\
\ffredholm(\id+zT_g) &= \det(\id+zT_g)\,\E^{-z\,\tr(T_g) + \frac{z^2\,\tr((T_g)^2)}{2}}
\end{align*}
if $T_g$ is trace-class. This idea of compensation goes back to Hilbert, and it enables certain renormalisation procedures in quantum field theory; see \cite{BS05}*{Chapter 9} for an introduction to these techniques.
The modified Fredholm determinants are entire functions on the complex plane; in the following we shall also consider them as holomorphic functions on the open unit disc $\Disk = \{z \in \C\,|\,|z|<1\}$. They are related to the observables of the graphon $\gamma$ by the equation
\begin{align*}
\ffredholm(\gamma,z) &= \exp\left(\sum_{k \geq 3} \frac{(-1)^{k-1}\,z^k}{k}\,\tr((T_g)^k)\right) =\exp\left(\sum_{k \geq 3} \frac{(-1)^{k-1}\,z^k}{k}\,t(\gamma,C_k)\right).
\end{align*}
The ratio $\det_2(\gamma,z)/\det_3(\gamma,z)$ is equal to $\exp(-\frac{z^2}{2}\,\tr((T_g)^2))$ for any representative graph function $g$ of $\gamma$. The trace of $(T_g)^2$ is 
$$\tr((T_g)^2)=\int_{[0,1]^2}(g(x,y))^2\DD{x}\DD{y} = \sum_{i \in I}(\lambda_i)^2.$$ 
In the sequel we denote this quantity $\|\gamma\|^2$, as it only depends on the graphon $\gamma=[g]$. Intuitively, $\|\gamma\|^2$ is the density of the graph $C_2 =\,\begin{tikzpicture}[baseline=-1mm,scale=0.5]
\foreach \x in {(0,0),(1,0)}
\fill \x circle (4pt);
\draw (0,0) .. controls (0.5,0.3) .. (1,0);
\draw (0,0) .. controls (0.5,-0.3) .. (1,0);
\end{tikzpicture}$ in $\gamma$, but since we consider graphs without multiple edges this does not really make sense.
\medskip

We are now ready to prove our criterion for constant graphons. We start with the following general results of convergence of the characteristic polynomials of the adjacency matrices $A_n(\gamma)$ (Proposition \ref{prop:convergence_ffredholm} and Corollary \ref{cor:convergence_fredholm}):

\begin{proposition}\label{prop:convergence_ffredholm}
Let $\gamma=[g] \in \Gspa$ be an arbitrary graphon, and $(G_n(\gamma))_{n \in \N}$ be the associated $W$-random graphs. We denote $\Gamma_n(\gamma)=\gamma_{G_n(\gamma)}$ the graphon associated to $G_n(\gamma)$, and $A_n(\gamma)$ the adjacency matrix of the graph $G_n(\gamma)$. We have the convergence in probability
$$\ffredholm(\Gamma_n(\gamma),z) = \det\!\left(I_n + z\,\frac{A_n(\gamma)}{n}\right)\,\exp\left(\frac{z^2}{2}\,t(\edgegraph,\Gamma_n(\gamma))\right) \to_{\proba,\,n \to \infty} \ffredholm(\gamma,z),$$
the topology being for instance the Montel topology of locally uniform convergence of holomorphic functions on the unit disc $\Disk$.
\end{proposition}

This proposition is connected to the convergence of the scaled eigenvalues of the random graph $G_n(\gamma)$ towards the eigenvalues of the integral operator $T_g$. We refer in particular to \cite{Lov12}*{Theorem 11.54} and \cite{BCLSV12}*{Theorem 6.7}, and we recall that $G_n(\gamma) \to_{\proba,\,n \to \infty} \gamma$ in the space of graphons $\Gspa$. Since we are interested in the (modified) Fredholm determinants, we need to be careful with the infinite sums or products over eigenvalues of the relevant Hilbert--Schmidt operators; the proof below make these manipulations rigorous.

\begin{proof}
For a graph $G$ on $n$ vertices, we have
$$\|[\gamma_G]\|^2 = \int_{[0,1]^2} (g_G(x,y))^2\DD{x}\DD{y} = \int_{[0,1]^2} g_G(x,y)\DD{x}\DD{y} = t(\edgegraph,G).$$
On the other hand, the operator associated to the graph function of $G$ is
$$(T_{g_G}f)(x) = \sum_{1\leq i,j\leq n} 1_{(i\sim_G j)}\,1_{(\frac{i-1}{n}\leq x < \frac{i}{n})}\,\int_{\frac{j-1}{n}}^{\frac{j}{n}} f(y)\DD{y} = \sum_{1\leq i,j\leq n} \frac{A_{ij}}{n} \,e_i(x)\,\scal{f}{e_j},$$
where $A_G=(A_{ij})_{1\leq i,j \leq n}$ is the adjacency matrix of $G$, and $(e_i)_{1\leq i\leq n}$ is the orthonormal family in $\leb^2([0,1])$ defined by $e_i(x) = \sqrt{n}\,1_{\frac{i-1}{n}\leq x <i}$. In other words,
$$T_{g_G} = \sum_{1\leq i,j\leq n} | e_i \rangle\, \frac{A_{ij}}{n} \,\langle e_j |.$$
So, the (random) finite-rank operator $T_{g_{G_n(\gamma)}}$ has for Fredholm determinant $\det(I_n + z\,\frac{A_n(\gamma)}{n})$. Then, to get the $3$-modified Fredholm determinant, we divide by the exponentials of:
\begin{itemize}
    \item $z\,\tr(T_{g_{G_n(\gamma)}}) = 0$, as the adjacency matrix $A_n(\gamma)$ has zeroes on the diagonal;
    \item $-\frac{z^2}{2}\,\|\gamma_{G_n(\gamma)}\|^2 = -\frac{z^2}{2}\,t(\edgegraph,G_n(\gamma))$.
\end{itemize}
So, $\det_3(\Gamma_n(\gamma),z) = \det(I_n + z\,\frac{A_n(\gamma)}{n})\,\exp(\frac{z^2}{2}\,t(\edgegraph,\Gamma_n(\gamma)))$. Let us now prove the convergence in probability. For any $k \geq 3$,
\begin{align*}
\esper[t(C_k,G_n(\gamma))] &= t(C_k,\gamma) + O\!\left(\frac{k^2}{n}\right);\\
\var(t(C_k,G_n(\gamma))) &=O\!\left(\frac{k^2}{n}\right),
\end{align*}
see \cite{LS06}*{Lemma 2.4}. Denote $A_1$ and $A_2$ the implied constants in the $O(\cdot)$'s above. By the Bienaymé--Chebyshev inequality,
$$\proba\!\left[\big|t(C_k,G_n(\gamma)) - \esper[t(C_k,G_n(\gamma))]\big| \geq \frac{k^2}{n^{1/4}}\right] \leq \frac{A_2}{n^{1/2}k^2}.$$
Therefore, with probability larger than $1-\frac{A_2\pi^2}{6n^{1/2}}$, on the disk $\Disk(\eta) = \{z \in \C\,|\,|z|<\eta\}$, we have
\begin{align*}
\left|\sum_{k=3}^{\infty} \frac{(-1)^{k-1}}{k}\,t(C_k,G_n(\gamma)) \,z^k - \sum_{k=3}^{\infty} \frac{(-1)^{k-1}}{k}\,\esper[t(C_k,G_n(\gamma))] \,z^k\right| \leq \frac{1}{n^{1/4}}\sum_{k=3}^\infty k\,\eta^k = O_\eta(n^{-1/4}).
\end{align*}
We also have
$$\left|\sum_{k=3}^{\infty} \frac{(-1)^{k-1}}{k}\,\esper[t(C_k,G_n(\gamma))] \,z^k - \sum_{k=3}^{\infty} \frac{(-1)^{k-1}}{k}\,t(C_k,\gamma) \,z^k\right| \leq \frac{A_1}{n}\sum_{k=3}^\infty k\,\eta^k = O_{\eta}(n^{-1}).$$
Taking the sum of these upper bounds and exponentiating, we obtain the estimate
$$\frac{\det_3(\Gamma_n(\gamma),z)}{\det_3(\gamma,z)} = \exp(O_\eta(n^{-1/4}))\quad \text{uniformly on }\Disk(\eta).$$
This is true with probability $1-O(n^{-1/2})$, whence the convergence in probability.
\end{proof}
\medskip

\begin{corollary}\label{cor:convergence_fredholm}
For any graphon $\gamma \in \Gspa$, we have the convergence in probability
$$\det\!\left(I_n + z\,\frac{A_n(\gamma)}{n}\right) \to_{n \to \infty} \exp\left(-\frac{z^2}{2}\,t(\edgegraph,\gamma)\right)\ffredholm(\gamma,z).$$
We also have convergence of the expectations $\esper[\det(I_n + z\,\frac{A_n(\gamma)}{n})]$ towards the same limit locally uniformly on the open unit disc $\Disk$.
\end{corollary}

\begin{proof}
The first part of the corollary follows immediately from Proposition \ref{prop:convergence_ffredholm} and from the convergence in probability
$$t( \edgegraph,G_n(\gamma)) \to_{\proba,\,n\to \infty} t(\edgegraph,\gamma).$$
To get the convergence of the expectations, it suffices to have a uniform bound on the determinants $\det(I_n + z\,\frac{A_n(\gamma)}{n})$ for $z \in \Disk(\eta)$ and $\eta \in (0,1)$. However, since the cycle observables $t(C_k,\cdot)$ are bounded by $1$,
$$\Re \left(\log \ffredholm(\Gamma_n(\gamma),z)\right) \leq \sum_{k=3}^\infty \frac{\eta^k}{k} = O_\eta(1)$$
uniformly for $z \in \Disk(\eta)$, and similarly $\frac{|z|^2}{2}\,t(\edgegraph,G_n(\gamma)) = O_\eta(1)$, so such a uniform bound indeed exists.
\end{proof}
\medskip

We now identify the constant graphons by means of their $3$-modified Fredholm determinant.

\begin{proposition}\label{prop:ffredholm_constant_p}
Consider a graphon $\gamma \in \Gspa$ and a parameter $p \in [0,1]$. The following assertions are equivalent:
\begin{enumerate}
    \item The graphon $\gamma$ is the class $\gamma_p$ of the constant graph function $g(x,y)=p$.
    \item The graphon $\gamma$ has edge density $p$ and $3$-modified Fredholm determinant
    $$\ffredholm(\gamma,z) = (1+pz)\,\E^{-pz+\frac{p^2z^2}{2}}.$$
\end{enumerate}
\end{proposition}

\begin{proof}
Note that the graphon $\gamma_p$ contains a unique graph function, namely, $g(x,y)=p$ almost everywhere. Suppose that $\gamma=\gamma_p$. 
Obviously, $t( \edgegraph,\gamma_p) = \iint_{[0,1]^2} p \DD{x}\DD{y} = p$. Besides, for any $x \in [0,1]$ and $f \in \leb^2([0,1])$, we have $(T_p f)(x) = p\,\int_0^1 f(y)\DD{y}$, so if we denote 
\begin{align*}
\pi : \leb^2([0,1]) &\to \C \\
f&\mapsto \int_0^1 f(x)\DD{x}
\end{align*}
the projection on the space of constant functions, then $T_p = p\,\pi$. The projection $\pi$ has rank $1$, so it is trace-class and its unique non-zero eigenvalue is $1$ (with multiplicity $1$). It follows immediately that:
\begin{align*}
\det(\id+zT_p) &= \det(\id + pz\,\pi) = 1+pz ;\\
\ffredholm(\gamma_p,z) &= \det(\id + pz\,\pi)\,\E^{-pz\,\tr(\pi)+\frac{p^2z^2}{2}\,\tr(\pi^2)} = (1+pz)\,\E^{-pz+\frac{p^2z^2}{2}}.
\end{align*}
Suppose conversely that $\ffredholm(\gamma,z)=(1+pz)\,\E^{-pz+\frac{p^2z^2}{2}}$ and that $\gamma=[g]$ has edge density $p$. If $\{\lambda_i\}_{i \in I}=\Spec(\gamma)$, then we have for $k \geq 3$
$$\sum_{i\in I} (\lambda_i)^k = \tr((T_g)^k) = (-1)^{k-1}\,k\,[z^k](\log\ffredholm(\id + z T_g)) = p^k.$$
In particular, $\sum_{i \in I} (\lambda_i)^4 = p^4$, so $|\lambda_i|\leq p$ for all $i \in I$. If we had two non-zero eigenvalues $\lambda_1$ and $\lambda_2$, then all the eigenvalues would be of modulus smaller than $q<p$, and we would then have $\sum_{i \in I} (\lambda_i)^6 \leq q^2 \sum_{i \in I} (\lambda_i)^4 = q^2p^4 < p^6$: this contradicts the case $k=6$ of the formula above. So, there is a unique non-zero eigenvalue $\lambda_1 \in \{\pm p\}$, and since $\tr((T_g)^3) = p^3$, $\lambda_1=p$. Then, the symmetric kernel $g(x,y)$ can be written as
$$g(x,y) = p\,f(x)\,f(y)\quad \text{with }f \in \leb^2([0,1])\text{ and }\int_0^1 (f(x))^2\DD{x} = 1.$$
Finally, $ p =t( \edgegraph,\gamma) = \iint_{[0,1]^2} g(x,y) \DD{x}\DD{y} = p\, (\int_0^1 f(x)\DD{x})^2$, so $f$ is a case of equality in the Cauchy--Schwarz inequality. We conclude that $f = 1$ almost everywhere, and that $g(x,y) =p$ almost everywhere on the square. 
\end{proof}

The proof of Theorem \ref{thm:spectral_constant_graphon} is now immediate. Indeed, under the assumptions of this theorem, by Corollary \ref{cor:convergence_fredholm},
 $$\E^{-\frac{pz^2}{2}}\,\ffredholm(\gamma,z) = (1+pz)\,\E^{-pz-\frac{p(1-p)z^2}{2}},$$
and by Proposition \ref{prop:ffredholm_constant_p}, this identity characterises $\gamma_p$ among the graphons with edge density $p$. \medskip

\begin{remark}
Notice that we can actually give an exact formula for the expected characteristic polynomial $\esper[\det(I_n+z\,\frac{A_n(\gamma_p)}{n})]$. Indeed, denote $(B_{ij})_{i\neq j}$ a family of $\binom{n}{2}$ independent Bernoulli variables with parameter $p$, and with $B_{ij}=B_{ji}$. By taking the expansion of the determinant over permutations of size $n$, we get:
\begin{align*}
\esper\!\left[\det\!\left(I_n+z\,\frac{A_n(\gamma_p)}{n}\right)\right]&= \sum_{\sigma \in \sym(n)} \eps(\sigma)\,\esper\!\left[\prod_{i=1}^n \begin{cases}
    1 &\text{if }i=\sigma(i)\\
    \frac{z\,B_{i\sigma(i)}}{n} &\text{if }i \neq \sigma(i) \end{cases}\right]\\
    &= \sum_{\sigma \in \sym(n)} \eps(\sigma)\,\left(\frac{pz^2}{n^2}\right)^{m_2(\sigma)} \prod_{k \geq 3} \left(\frac{pz}{n}\right)^{km_k(\sigma)}\\
    &= \left(\frac{pz}{n}\right)^n  \sum_{\sigma \in \sym(n)} \eps(\sigma) \,\left(\frac{n}{pz}\right)^{\!m_1(\sigma)}\,\left(\frac{1}{p}\right)^{\!m_2(\sigma)},
\end{align*}
where $m_k(\sigma)$ denotes the number of cycles of length $k$ in $\sigma$. The sum is easily computed by using the formalism of symmetric functions, see \cite{Mac95}. Denote $(p_k)_{k \geq 1}$ the Newton power sums in the algebra of symmetric functions $\mathrm{Sym}$, and $p_k(X)$ the following specialisation of $\mathrm{Sym}$:
$$p_1(X) = \frac{n}{pz}\quad;\quad p_2(X) = -\frac{1}{p}\quad;\quad p_{k \geq 3}(X) = (-1)^{k-1}.$$
Then, if $\mathfrak{Y}(n)$ denotes the set of integer partitions with size $n$, we have:
$$
\sum_{\sigma \in \sym(n)} \eps(\sigma) \,\left(\frac{n}{pz}\right)^{\!m_1(\sigma)}\,\left(\frac{1}{p}\right)^{\!m_2(\sigma)} = n!\,\sum_{\mu \in \mathfrak{Y}(n)} \frac{p_\mu(X)}{z_\mu} = n!\,[t^n]\left\{\exp\!\left( \sum_{k=1}^\infty \frac{p_k(X)\,t^k}{k}\right)\right\};
$$
see \cite{Mac95}*{Chapter I, Section 2}. Thus, the expectation is given by:
\begin{align*}
&\left(\frac{pz}{n}\right)^n\,n!\,[t^n]\left\{\exp\!\left(\frac{nt}{pz} - \frac{t^2}{2p} + \sum_{k=3}^\infty \frac{(-1)^{k-1}t^k}{k}\right) \right\}\\
&=\left(\frac{pz}{n}\right)^n\,n!\,[t^n]\left\{\exp\!\left(\left(\frac{n}{pz}-1\right)t + \left(1-\frac{1}{p}\right)\frac{t^2}{2} \right)(1+t) \right\}\\
&= \left(\frac{pz}{n}\right)^n\,n!\,([t^n] + [t^{n-1}])\left\{\exp\!\left(\left(\frac{n}{pz}-1\right)t - \left(\frac{1}{p}-1\right)\frac{t^2}{2} \right)\right\}.
\end{align*}
The coefficients of the power series in $t$ are related to the Hermite polynomials: since $\E^{ut-\frac{t^2}{2}}=\sum_{n=0}^\infty H_n(u)\,\frac{t^n}{n!}$,
$$n!\,[t^n]\left\{\exp\!\left(\left(\frac{n}{pz}-1\right)t + \left(1-\frac{1}{p}\right)\frac{t^2}{2} \right)\right\} =\left(\frac{1}{p}-1\right)^{\!\frac{n}{2}} \,H_n\!\left(\frac{\frac{n}{pz}-1}{\sqrt{\frac{1}{p}-1}}\right).$$
So, we get the explicit formula
$$\esper\!\left[\det\!\left(I_n+z\,\frac{A_n(\gamma_p)}{n}\right)\right] = \left(\frac{pz}{n}\sqrt{\frac{1}{p}-1}\,\right)^{\!n} \left(H_n\!\left(\frac{\frac{n}{pz}-1}{\sqrt{\frac{1}{p}-1}}\right) + \frac{n}{\sqrt{\frac{1}{p}-1}}\,H_{n-1}\!\left(\frac{\frac{n}{pz}-1}{\sqrt{\frac{1}{p}-1}}\right)\right).$$
A saddle point analysis of the integral representation
$$\esper\!\left[\det\!\left(I_n+z\,\frac{A_n(\gamma_p)}{n}\right)\right] = \left(\frac{pz}{n}\right)^n\,\frac{n!}{2\I \pi}\,\oint \exp\!\left(\left(\frac{n}{pz}-1\right)t + \left(1-\frac{1}{p}\right)\frac{t^2}{2} \right)\frac{(1+t)}{t^{n+1}} \DD{t}$$
allows one to recover the asymptotic formula $(1+pz)\,\E^{-pz-\frac{p(1-p)z^2}{2}}$ from Theorem \ref{thm:spectral_constant_graphon}.
\end{remark}\medskip

To close this section, let us consider the functional $F_n(z,\gamma;\eps_1,\ldots,\eps_n)$ from Proposition \ref{prop:expectation_laplacian}. It is natural to consider its evaluation at the \emph{random} parameters $\eps_1(\gamma),\ldots,\eps_n(\gamma)$, and to compute the limit of this quantity. Notice that 
$$F_n(z,\gamma;\eps_1(\gamma),\ldots,\eps_n(\gamma)) = \Pi \times \Sigma$$
with $\Pi = \prod_{i=1}^n (1+\eps_i(\gamma)z)$ and 
$$\Sigma = \frac{1}{n}\sum_{i=1}^n \frac{1+pz+\eps_i(\gamma)z}{1+\eps_i(\gamma)z} = 1+pz - \frac{pz}{n} \sum_{i=1}^n \frac{\eps_i(\gamma)z}{1+\eps_i(\gamma)z}.$$
The sum $\Sigma$ is easy to estimate: by Proposition \ref{prop:concentration_epsi}, all the $\eps_i(\gamma)$ are smaller than $4\sqrt{\log n/n}$ with probability larger than $1-\frac{4}{n}$, so $\Sigma$ converges in probability towards $1+pz$. On the other hand, the product $\Pi$ is related to the tree observables of the graph $G_n(\gamma)$. Indeed,
$$\Pi = \exp\left(\sum_{k=1}^\infty \frac{(-1)^{k-1}z^k}{k}\,\left(\sum_{i=1}^n (\eps_i(\gamma))^k\right)\right),$$
and for any $k \geq 1$,
$$\sum_{i=1}^n (\deg(i,G_n(\gamma)))^k = S_n\left(\begin{tikzpicture}[scale=0.8,baseline=3mm]
\foreach \x in {(0,0),(-1,1),(0,1),(2,1)}
\fill \x circle (2pt);
\draw (-1,1) -- (0,0) -- (0,1);
\draw (2,1) -- (0,0);
\draw (1,1) node {$\cdots$};
\draw (0,-0.3) node {$1$};
\draw (-1,1.3) node {$2$};
\draw (0,1.3) node {$3$};
\draw (2,1.3) node {$k+1$};
\end{tikzpicture}\right)=S_n(T_k),$$
because $\deg(i,G_n(\gamma))^k = \sum_{i_2,\ldots,i_{k+1} \in \lle 1,n\rre} \prod_{l=2}^{k+1} 1_{(i\, \sim_{G_n(\gamma)}\,i_l)}$. By using this observation and our Main Theorem, we can deduce that $\Pi$ converges in distribution when $n$ goes to infinity (but not necessarily in expectation). Indeed, for $k \geq 3$, by Proposition \ref{prop:concentration_epsi},
$$\sum_{i=1}^n |\eps_i(\gamma)|^k = O\!\left(n\,\left(\frac{\log n}{n}\right)^{\!\frac{k}{2}}\right) = O\!\left(\frac{(16\,\log n)^{\frac{k}{2}}}{n^{1-\frac{k}{2}}}\right)$$
with probability larger than $1-\frac{4}{n}$, and with a constant in the $O(\cdot)$ which does not depend on $k$. Therefore, with high probability,
$$\sum_{k=3}^\infty \frac{|z|^k}{k} \left(\sum_{i=1}^n |\eps_i(\gamma)|^k\right) = O\!\left(n\, \left(|z|\sqrt{\frac{\log n}{n}}\right)^3\right) \to_{n \to \infty} 0.$$
So, by using the formulas from Example \ref{ex:expectations_bi_edge}, we see that with high probability, the functional $F_n(z,\gamma;\eps_1(\gamma),\ldots,\eps_n(\gamma))$ is equal up to a multiplicative $(1+o(1))$ to 
\begin{align*}
&(1+pz)\,\exp\!\left(z\,\left(\sum_{i=1}^n \eps_i(\gamma)\right) - \frac{z^2}{2}\left(\sum_{i=1}^n (\eps_i(\gamma))^2\right)\right)\\
&=(1+pz)\,\exp\!\left(z\,n \big(t(\edgegraph,G_n(\gamma))-p\big) - \frac{z^2}{2}\,n\big(t(\begin{tikzpicture}[baseline=-1mm,scale=0.5]
\foreach \x in {(0,0),(1,0),(2,0)}
\fill \x circle (4pt);
\draw (0,0) -- (2,0);
\end{tikzpicture},G_n(\gamma)) - 2p\,t(\edgegraph,G_n(\gamma)) + p^2\big)\right)\\
&\simeq (1+pz)\,\E^{-pz-\frac{pz^2}{2}+\frac{p^2z^2}{2}} \exp\!\left((z+pz^2)\,Y_n(\edgegraph)  - \frac{z^2}{2}\,Y_n(\begin{tikzpicture}[baseline=-1mm,scale=0.5]
\foreach \x in {(0,0),(1,0),(2,0)}
\fill \x circle (4pt);
\draw (0,0) -- (2,0);
\end{tikzpicture}) \right),
\end{align*}
with $Y_n(F) = n(t(F,G_n(\gamma)) - \esper[t(F,G_n(\gamma))])$. Therefore, our Main Theorem ensures that there is a limit in distribution for the product $\Pi\times \Sigma$ for any $z$ fixed, and for any singular graphon $\gamma$. As we do not know in the case of a general singular graphon what is the limiting distribution of the rescaled graph densities, unfortunately we cannot pursue this computation. Nonetheless, we believe that this approach might lead to a spectral reformulation of our Conjecture: this was the objective of the two last Subsections. Notice that the sum $\Sigma$ is always bounded (at least for $z$ small enough), so we also have a strong convergence in moments for this quantity. This is not \emph{a priori} the case of the product $\Pi$. Hence, even if we knew the limiting distribution of the rescaled densities $Y_n(\edgegraph)$ and $Y_n(\begin{tikzpicture}[baseline=-1mm,scale=0.5]
\foreach \x in {(0,0),(1,0),(2,0)}
\fill \x circle (4pt);
\draw (0,0) -- (2,0);
\end{tikzpicture})$, it is not clear that we could use the computation above in order to obtain $\lim_{n \to \infty} \esper[F_n(z,\gamma;\eps_1(\gamma),\ldots,\eps_n(\gamma))]$. Besides, as mentioned in the introduction, this limit is probably different from $\lim_{n \to \infty}\esper[D\!L_n(z,\gamma;\eps_1(\gamma),\ldots,\eps_n(\gamma))]$.
\bigskip

\section{Convergence of the scaled densities}\label{sec:convergence}
We fix again a singular graphon $\gamma$, and we now aim to prove our Main Theorem. By combining Proposition \ref{prop:expansion_cumulant} and Theorem \ref{thm:additional_equations}, we see that for any family of finite graphs $F_1,\ldots,F_s$ with sizes $k_1,\ldots,k_s \geq 2$,
$$\kappa(S_n(F_1),\ldots,S_n(F_s)) = O(n^{k_1+\cdots + k_s-s}).$$
Equivalently, if we consider the rescaled densities $Y_n(F) = n\,(t(F,G_n(\gamma))-\esper[t(F,G_n(\gamma))])$, then all the joint cumulants $\kappa(Y_n(F_1),\ldots,Y_n(F_s))$ are bounded as functions of $n$. In this section, we are going to prove that all these cumulants converge when $n$ goes to infinity, and also that they all satisfy a strong upper bound analogous to Equation \eqref{eq:bound_cumulant}, but with a smaller power of $n$ since we are dealing with singular graphons. This strong upper bound (Theorem \ref{thm:upper_bound}) will allow us to resum the cumulant generating series
$$\log \esper\!\left[\E^{z_1Y_n(F_1)+\cdots + z_sY_n(F_s)}\right] = \sum_{r \geq 2}\sum_{\substack{r_1,\ldots,r_s \geq 0 \\ r_1+\cdots+ r_s=r}} \frac{(z_1)^{r_1}\cdots (z_s)^{r_s}}{(r_1)!\cdots (r_s)!}\,\kappa\big((Y_n(F_1))^{\otimes r_1},\ldots,(Y_n(F_s))^{\otimes r_s}\big),$$
and to prove that this series converges locally uniformly on a fixed polydisk when $n$ goes to infinity. This entails the joint convergence in distribution which we want to establish. We shall also write an explicit formula for the third and fourth cumulants of the edge density $t(\edgegraph,G_n(\gamma))$; this computation relies in particular on the factorisation over transitive motives. The vanishing of the limiting cumulants of order $r \geq 3$ of the edge density will turn out to be equivalent to the Chung--Graham--Wilson criterion stated in the introduction; see Proposition \ref{prop:gaussian_edge_density}.
\medskip

\subsection{A strong upper bound on the cumulants of singular graphons}
Until the end of this paragraph, $s$ is a fixed integer larger than $2$, $F_1,\ldots,F_s$ are fixed finite graphs with sizes $k_1,\ldots,k_s\geq 2$, and $\gamma$ is a fixed singular graphon. We start from the expansion by multilinearity \eqref{eq:multilinearity} of the cumulant of $S_n(F_1),\ldots,S_n(F_s)$, and we gather the families of maps $\phi = (\phi_1,\ldots,\phi_s)$ according to the associated set partitions $\pi(\phi) \in \pym(V)$, where $V=\bigsqcup_{r=1}^s V_{F_r}$:
$$\kappa(S_n(F_1),\ldots,S_n(F_s)) = \sum_{\pi \in \pym(V)} \sum_{\substack{\phi_1 : \lle 1,k_1 \rre \to \lle 1,n\rre \\ \vdots \qquad\vdots \vspace*{1mm}\\ \phi_s : \lle 1,k_s\rre \to \lle 1,n\rre \\ \pi(\phi) = \pi}} \kappa(A_{\phi_1}(F_1),\ldots,A_{\phi_s}(F_s)).$$
By Lemma \ref{lem:connected_hypergraph}, the set partitions $\pi$ such that $H_{\pi}$ is not connected do not contribute to the sum. On the other hand, the set partitions $\pi$ such that $H_{\pi}$ is a hypertree correspond in general to the term with degree $k_1+\cdots+k_s-(s-1)$ in $n$, but here they yield a term with degree $k_1+\cdots +k_s-s$ since $\gamma$ is singular. Therefore, it is convenient to split the joint cumulant in two parts: $\kappa(S_n(F_1),\ldots,S_n(F_s)) = \alpha+\beta$ with
\begin{align*}
\alpha &= \sum_{\pi \in \pym_{\mathrm{hypertree}}(F_1,\ldots,F_s)} \sum_{\substack{\phi_1 : \lle 1,k_1 \rre \to \lle 1,n\rre \\ \vdots \qquad\vdots \vspace*{1mm}\\ \phi_s : \lle 1,k_s\rre \to \lle 1,n\rre \\ \pi(\phi) = \pi}} \kappa(A_{\phi_1}(F_1),\ldots,A_{\phi_s}(F_s)) ;\end{align*}
\begin{align*}
\beta &= \sum_{\substack{\pi \in \pym(V) \\ H_{\pi} \text{ is connected} \\ H_{\pi} \text{ is not a hypertree} }}\sum_{\substack{\phi_1 : \lle 1,k_1 \rre \to \lle 1,n\rre \\ \vdots \qquad\vdots \vspace*{1mm}\\ \phi_s : \lle 1,k_s\rre \to \lle 1,n\rre \\ \pi(\phi) = \pi}} \kappa(A_{\phi_1}(F_1),\ldots,A_{\phi_s}(F_s)).
\end{align*}
The quantity $\alpha$ has already been studied in Subsection \ref{sub:joint_cumulants}. By the discussion leading to Proposition \ref{prop:expansion_cumulant}, it is equal for any graphon $\gamma$ to
$$\sum_{\pi \in \pym_{\mathrm{hypertree}}(F_1,\ldots,F_s)} n^{\downarrow k_1+\cdots + k_s-(s-1)}\, \kappa_\pi(F_1,\ldots,F_s)(\gamma).$$
However, by Theorem \ref{thm:additional_equations}, if $\gamma$ is singular, then
$$\kappa_s(F_1,\ldots,F_s)(\gamma) = \sum_{\pi \in \pym_{\mathrm{hypertree}}(F_1,\ldots,F_s)} \kappa_\pi(F_1,\ldots,F_s)(\gamma) = 0.$$
An adequate upper bound on $\alpha$ follows therefore from:

\begin{proposition}\label{prop:upper_bound_alpha}
For any graphon $\gamma$ and any finite graphs $F_1,\ldots,F_s$ with sizes $k_1,\ldots,k_s \geq 2$,
\begin{align*}
&\left|\sum_{\pi \in \pym_{\mathrm{hypertree}}(F_1,\ldots,F_s)} \left(n^{k_1+\cdots+k_s-(s-1)} - n^{\downarrow k_1+\cdots + k_s-(s-1)}\right)\, \kappa_\pi(F_1,\ldots,F_s)(\gamma) \right| \\
&\leq 2^{s-2} (k_1+\cdots+k_s)^{s}\,k_1k_2\cdots k_s\,n^{k_1+\cdots+k_s-s}.
\end{align*}
\end{proposition}

\begin{lemma}
Consider an arbitrary graphon $\gamma$ and finite graphs $F_1,\ldots,F_s$ with sizes $k_1,\ldots,k_s \geq 2$. We also consider $\phi=(\phi_1,\ldots,\phi_s)$, a family of maps $\phi_r :V_{F_r} \to \lle 1,n\rre$, and we denote as before $\pi =\pi(\phi)$ the associated set partition of $V=\bigsqcup_{i=1}^r V_{F_i}$ and $H_{\pi}$ the corresponding hypergraph on $\lle 1,s\rre$. Then,
$$|\kappa_\pi(A_{\phi_1}(F_1),\ldots,A_{\phi_s}(F_s))| \leq 2^{s-1}\,\left|\ST_{H_\pi}\right|,$$
where $\ST_{H_\pi}$ denotes the set of trees $T$:
\begin{itemize}
\item with vertex set $\lle 1,s\rre$;
\item whose edges are all included in hyperedges of $H_\pi$.
\end{itemize}
\end{lemma}

\begin{proof}
Consider the graph $DG$ with vertex set
$$V_{DG} = \{(\phi,F)\,| \,F\text{ finite graph},\,\phi : V_F \to \lle 1,n\rre\},$$
and with an edge between the vertices $(\phi,F)$ and $(\phi',F')$ if $\phi(V_F) \cap \phi'(V_{F'})\neq \emptyset$. Then, in the sense of \cite{FMN16}*{Definition 9.1.1}, $DG$ is a dependency graph for the Bernoulli random variables $A_{\phi}(F)$. As a consequence, by \cite{FMN16}*{Equation (9.9)}, we have the upper bound
$$|\kappa(A_{\phi_1}(F_1),\ldots,A_{\phi_s}(F_s))| \leq 2^{s-1}\,\left|\ST_{DG[(\phi_1,F_1),\ldots,(\phi_s,F_s)]}\right|,$$
where:
\begin{itemize}
     \item  $DG[(\phi_1,F_1),\ldots,(\phi_s,F_s)]$ is the graph induced by $DG$ on the set of vertices $\lle 1,s\rre$, with $i$ connected to $j$ if $\phi_i(V_{F_i}) \cap \phi_{j}(V_{F_j})\neq \emptyset$.
     \item $\ST_{K}$ denotes the set of spanning trees of a labelled graph $K$.
 \end{itemize} 
 Let us relate the induced graph $K_\phi = DG[(\phi_1,F_1),\ldots,(\phi_s,F_s)]$ to the hypergraphs $G_{\pi(\phi)}$ and $H_{\pi(\phi)}$ introduced in Subsection \ref{sub:joint_cumulants}. We have:
\begin{align*}
&\big(i \leftrightarrow_{K_\phi} j \big) \\
&\iff \big(\text{the images of the maps $\phi_i$ and $\phi_j$ have a non-empty intersection}\big) \\
&\iff \big(\text{the set partition $\pi(\phi)$ has a part which contains a point in $V_{F_i}$ and a point in $V_{F_j}$}\big)\\
&\iff \big(i \leftrightarrow_{H_{\pi(\phi)}} j\big).
\end{align*}
Thus, $K_\phi$ is the graph obtained from $H_{\pi(\phi)}$ by replacing each hyperedge $e=\{r_1,r_2,\ldots,r_d\}$ by the complete graph on the $d$ vertices $r_1,\ldots,r_d$, and then by deleting the possible loops and multiple edges which have been created. For instance, if $H_{\pi(\phi)}$ is the hypertree in Figure \ref{fig:hypertree}, then $K_\phi$ is:
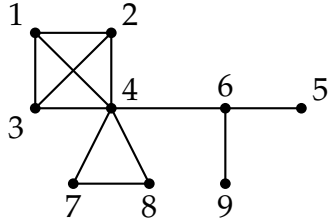
\begin{center}
\begin{tikzpicture}[scale=1]
\foreach \x in {(0,0),(0,1),(-1,0),(-1,1),(1.5,0),(2.5,0),(1.5,-1),(-0.5,-1),(0.5,-1)}
\fill \x circle (2pt);
\draw [thick] (0,0) -- (-0.5,-1) -- (0.5,-1) -- (0,0) -- (-1,0) -- (-1,1) -- (0,1) -- (0,0) -- (1.5,0) -- (2.5,0);
\draw [thick] (1.5,0) -- (1.5,-1);
\draw [thick] (-1,1) -- (0,0);
\draw [thick] (0,1) -- (-1,0);
\draw (-1.25,1.25) node {$1$};
\draw (0.25,1.25) node {$2$};
\draw (0.25,0.25) node {$4$};
\draw (-1.25,-0.25) node {$3$};
\draw (-0.5,-1.3) node {$7$};
\draw (0.5,-1.3) node {$8$};
\draw (1.5,-1.3) node {$9$};
\draw (1.5,0.3) node {$6$};
\draw (2.75,0.25) node {$5$};
\end{tikzpicture}
\end{center}
It is then easy to see that the spanning trees of $K_\phi$ are the trees included in the hypergraph $H_{\pi(\phi)}$.
\end{proof}

\begin{proof}[Proof of Proposition \ref{prop:upper_bound_alpha}]
For any integer $K \geq 1$,  $n^{K} - n^{\downarrow K}$ is the number of non-injective maps from $\lle 1,K\rre$ to $\lle 1,n\rre$, so it is always smaller than $\binom{K}{2}\,n^{K-1}$: the binomial coefficient $\binom{K}{2}$ corresponds to the choice of two elements in $\lle 1,K\rre$ sent to the same image, and then we have to choose less than $K-1$ images in $\lle 1,n\rre$. Consequently, the quantity that we want to control is smaller than:
$$\binom{k_1+\cdots+k_s-(s-1)}{2}\,2^{s-1}\,n^{k_1+\cdots+k_s-s}\,\card\{(T,\pi)\,|\,T\in \ST_{H_\pi},\,H_\pi \in \Hset(s)\}.$$
We claim that the cardinality of the set of pairs $(T,\pi)$ is equal to:
$$(k_1+\cdots+k_s)^{s-2}\,k_1k_2\cdots k_s.$$
Note that this combinatorial formula implies immediately the proposition. Suppose given a Cayley tree $T$ on $s$ vertices (spanning tree of the complete graph $K_s$), and for each integer $r \in \lle 1,s\rre$, a $(\deg(r,T))$-tuple of elements in $\lle 1,k_r\rre$. A convenient way to represent this data is as follows: we draw the Cayley tree $T$ and, for each edge $\{i,j\} \in E_T$, we label the half-edge stemming from $i$ by an integer in $\lle 1,k_i\rre$. We thus obtain a \emph{bi-edge-labelled Cayley tree}. For instance, if $k_1=k_2=k_3=4$ and $T = \begin{tikzpicture}[baseline=-1mm,scale=0.5]
\foreach \x in {(0,0),(1,0),(2,0)}
\fill \x circle (4pt);
\draw (0,0) -- (2,0);
\end{tikzpicture}$, then a possible choice is
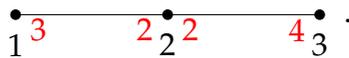
\begin{center}
 \begin{tikzpicture}[baseline=-1mm,scale=2]
\foreach \x in {(0,0),(1,0),(2,0)}
\fill \x circle (1pt);
\draw (0,0) -- (2,0);
\foreach \x in {1,2,3}
\draw (\x-1,-0.2) node {$\x$};
\draw [color=red] (0.15,-0.1) node {$3$};
\draw [color=red] (0.85,-0.1) node {$2$};
\draw [color=red] (1.15,-0.1) node {$2$};
\draw [color=red] (1.85,-0.1) node {$4$};
\end{tikzpicture}\,.
 \end{center} 
We can associate to a bi-edge-labelled Cayley tree $T$ of size $s$ a set partition $\pi$ of $V = \bigsqcup_{r=1}^s V_{F_r}$. Starting from the trivial set partition with $|V|$ singletons, we enumerate the edges of the bi-edge-labelled tree $T$ according for instance to a breadth-first search which starts from the vertex $1$. For each bi-labelled edge
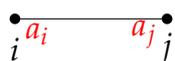
\begin{center}
 \begin{tikzpicture}[baseline=-1mm,scale=2]
\foreach \x in {(0,0),(1,0)}
\fill \x circle (1pt);
\draw (0,0) -- (1,0);
\draw [color=red] (0.15,-0.1) node {$a_i$};
\draw (0,-0.2) node {$i$};
\draw (1,-0.2) node {$j$};
\draw [color=red] (0.85,-0.1) node {$a_j$};
\end{tikzpicture}
 \end{center} 
that we encounter, we join the parts which contain the two elements $(i,a_i \in \lle 1,k_i\rre)$ and $(j,a_j \in \lle 1,k_j\rre)$. The resulting set partition $\pi \in \pym(V)$ does not depend on the order in which one makes the junctions.
For example, the bi-edge-labelled Cayley tree on $3$ vertices drawn above corresponds to the set partition 
$$\pi=\{(1,3),(2,2),(3,4)\} \sqcup \bigsqcup \{\text{singletons}\}$$
of the set $V$ which has $3\times 4 =12$ elements. It is easily seen that the corresponding hypergraph $H_\pi$ contains $T$, and we claim that it is always a hypertree. Indeed, when going from the trivial set partition $\pi_{\mathrm{trivial}}$ to the set partition $\pi$ associated to the bi-edge-labelled Cayley tree $T$, the corresponding hypertrees are modified as follows. Each bi-labelled edge of $T$ between $i$ and $j$ makes one:
\begin{itemize}
\item either create an edge $\{i,j\}$;
\item or, replace an hyperedge $e$ containing $i$ by $e\sqcup \{j\}$. 
\end{itemize}
Because of the rules above, the sum of the degrees in $H_\pi$ is equal to the number of edges of $T$, that is $s-1$; hence, $H_\pi$ is a connected hypergraph (because it contains $T$) with sum of degrees $s-1$, so it is indeed a hypertree. \medskip

Thus, to each bi-edge-labelled tree $T$ with vertex set $\lle 1,s\rre$ and whose labels stemming from $r \in \lle 1,s\rre$ belong to $\lle 1,k_r\rre$, we have associated a pair $(T,\pi)$ with $T \subset H_{\pi}$ and $H_{\pi} \in \Hset(s)$. Conversely, if $(T,\pi)$ is such a pair, then for each part $\{(r_1,a_1),\ldots,(r_t,a_t)\}$ of $\pi$, the trace of $T$ on the set of vertices $\{r_1,\ldots,r_t\}$ is a tree, and we can label each edge $\{r_i,r_j\}$ of this trace by the two labels $a_i$ and $a_j$. We conclude that there is a bijection between the bi-edge-labelled trees $T$ whose labels stemming from $r \in \lle 1,s\rre$ belong to $\lle 1,k_r\rre$, and the pairs which we wanted to count. The formula $(k_1+\cdots+k_s)^{s-2}\,k_1k_2\cdots k_s$ is then a generalisation of Kirchhoff's theorem: the generating function of the spanning trees of the complete graph $K_s$ counted according to their degree distribution is:
\begin{equation}
\sum_{T \in \ST(K_s)} (z_1)^{\deg(1,T)}(z_2)^{\deg(2,T)}\cdots (z_s)^{\deg(s,T)} = (z_1z_2\cdots z_s)\,(z_1+\cdots+z_s)^{s-2}.\label{eq:MT3}\tag{MT3}
\end{equation}
This generalisation is proved in the Appendix at the end of the article. Setting $z_1=z_2=\cdots=z_s=1$ yields the classical formula $s^{s-2} = |\ST(K_s)|$. Setting $z_1=k_1,\ldots,z_s=k_s$, we obtain the number of correctly bi-edge-labelled Cayley trees, whence our estimate.
\end{proof}

As an immediate corollary, we see that for a singular graphon $\gamma$,
$$|\alpha| \leq 2^{s-2} (k_1+\cdots+k_s)^{s}\,k_1k_2\cdots k_s\,n^{k_1+\cdots+k_s-s}$$
if $s \geq 2$. Let us now evaluate the quantity $\beta$. By using again the upper bound on the elementary joint cumulants $\kappa(A_{\phi_1}(F_1),\ldots,A_{\phi_s}(F_s))$, we obtain:
$$|\beta| \leq 2^{s-1}\,\card\{(T,\phi)\,|\,T \in \ST_{H_{\pi(\phi)}},\,H_{\pi(\phi)} \text{ is connected but is not a hypertree}\}.$$
Given a hypergraph $H$ on the set of vertices $\lle 1,s\rre$, we call \emph{graph reduction} of $H$ a graph $G$ on the vertex set $\lle 1,s\rre$, possibly with loops and with multiple edges, which is obtained by replacing each hyperedge $e=\{i_1,\ldots,i_d\}$ of $H$ by one of the $d^{d-2}$ Cayley tree connecting $i_1,\ldots,i_d$. To be more precise, for each hyperedge $e \in E_H$, we first choose a Cayley tree $T_e$ on the set of vertices $\lle 1,d\rre$, and then we map it to the elements $i_1,\ldots,i_d \in \lle 1,s\rre$. The resulting graph on $\{i_1,\ldots,i_d\}$ is not necessarily a tree, because we allow hyperedges $e \in E_H$ with multiple occurrences of a vertex $i$. 
\begin{example}
We have drawn below a hypergraph on $3$ vertices (on the left) and a possible graph reduction of it (on the right).
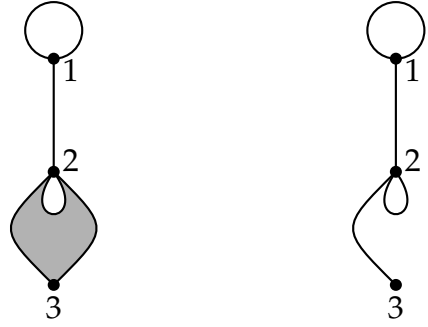
\begin{figure}[ht]
\begin{center}
\begin{tikzpicture}[scale=1.5]
\draw [thick] (0,0.25) circle (0.25);
\draw [thick] (0,0) -- (0,-1);
\fill [white!70!black] (0,-2) .. controls (-0.5,-1.5) .. (0,-1) .. controls (0.5,-1.5) .. (0,-2);
\fill [white] (0,-1) .. controls (0.35,-1.5) and (-0.35,-1.5) .. (0,-1);
\draw [thick] (0,-2) .. controls (-0.5,-1.5) .. (0,-1) .. controls (0.5,-1.5) .. (0,-2);
\draw [thick] (0,-1) .. controls (0.35,-1.5) and (-0.35,-1.5) .. (0,-1);
\foreach \x in {(0,0),(0,-1),(0,-2)}
\fill \x circle (1.5pt);
\draw (0.15,-0.1) node {$1$};
\draw (0.15,-0.9) node {$2$};
\draw (0,-2.2) node {$3$};
\begin{scope}[shift={(3,0)}]
\draw [thick] (0,0.25) circle (0.25);
\draw [thick] (0,0) -- (0,-1);
\draw [thick] (0,-2) .. controls (-0.5,-1.5) .. (0,-1);
\draw [thick] (0,-1) .. controls (0.35,-1.5) and (-0.35,-1.5) .. (0,-1);
\foreach \x in {(0,0),(0,-1),(0,-2)}
\fill \x circle (1.5pt);
\draw (0.15,-0.1) node {$1$};
\draw (0.15,-0.9) node {$2$};
\draw (0,-2.2) node {$3$};
\end{scope}
\end{tikzpicture}
\end{center}
\caption{A graph reduction of a hypergraph is obtained by replacing each hyperedge by a Cayley tree on the underlying multiset.}
\end{figure}
The hyperedges with degree $1$ are conserved, whereas the hyperedge $\{2,2,3\}$ must be replaced by the image of a Cayley tree on $\{a,b,c\}$ by the map $a \mapsto 2,\,b\mapsto 2,\,c\mapsto 3$. Here, we have chosen the Cayley tree $\begin{tikzpicture}[baseline=-1mm,scale=0.5]
\foreach \x in {(0,0),(1,0),(2,0)}
\fill \x circle (4pt);
\draw (0,0) -- (2,0);
\draw (0,-0.4) node {\tiny $a$};
\draw (1,-0.4) node {\tiny $b$};
\draw (2,-0.4) node {\tiny $c$};
\end{tikzpicture}$.
\end{example}

In general, since we replace each hyperedge $e$ of $H$ with degree $\deg e$ by a graph with $\deg e$ edges, the total number of edges of a graph reduction of a hypergraph $H$ is equal to $\sum_{e \in E_H}\deg e$. On the other hand, a hypergraph $H$ is connected if and only if its graph reductions are connected. By combining these two remarks, we see that a hypergraph $H$ is a hypertree if and only if its graph reductions are trees. This observation leads us to the following:

\begin{lemma}\label{lem:create_a_cycle}
Let $(T,H)$ be a pair which consists in a Cayley tree on the vertex set $\lle 1,s\rre$, and a connected hypergraph $H$ on $\lle 1,s\rre$ which is not a hypertree and such that $T \in \ST_H$. Then, there exists a pair $(i,j) \in \lle 1,s\rre^2 \setminus E_T$ and a graph reduction $G$ of $H$ such that $E_T \sqcup \{i,j\} \subset E_G$. 
\end{lemma}

\begin{proof}
By assumption, each edge $e \in E_T$ is included in at least one hyperedge $h(e) \in E_H$; note that there might be several such hyperedges of $H$ for a given edge of $T$. We can nonetheless choose a map $h : E_T \to E_H$ such that $e \subset h(e)$ for any $e \in E_T$. Fix a hyperedge $f \in E_H$: it can contain multiple occurrences of vertices in $\lle 1,s\rre$, but we can choose a labelling $\{i_1,i_2,\ldots,i_c,\ldots,i_d\}$ of its vertices such that $c=\card(f)$ and $i_1\neq i_2\neq \cdots \neq i_c$ are the distinct vertices of $f$. Then, $h^{-1}(f) \subset E_T$ is a forest on $\{i_1,i_2,\ldots,i_c\}$, because $T$ does not contain a cycle. So, we have for any $f \in E_H$ a (possibly empty) forest $h^{-1}(f)$ on the vertices of $f$. 
The union of these forests is the graph $T$, and $T$ is not a graph reduction of $H$, because this would imply that $H$ is a hypertree. This means that some forest $h^{-1}(f)$ is not the image by the map $\lle 1,d\rre \to \{i_1,i_2,\ldots,i_d\}$ of a Cayley tree on $\lle1,d\rre$. All the forests $h^{-1}(f)$ can be completed by adding edges in order to get Cayley trees, and thus a graph reduction $G$ of $H$; and we have proved that some forest has to be completed in a non trivial way, by adding at least one edge (which might be a loop $\{i,i\}$ if $f$ contains multiple occurrences of a vertex $i$). 
\end{proof}
\medskip

We can now estimate without difficulties the number of pairs $(T,\phi)$ involved in the upper bound on $|\beta|$. Let us fix a Cayley tree $T$ with vertex set $\lle 1,s\rre$ and a pair $(i,j) \in \lle 1,s\rre^2$. We want to count the number of maps $\phi : V \to \lle 1,n\rre$ such that $H_\phi$ is connected, is not a hypertree and admits a graph reduction $G$ with $E_T \sqcup \{i,j\} \subset E_G$. Note that for each edge $\{t,r\} \in E_T$, we can find a pair $(c_t,c_r) \in \lle 1,k_t\rre \times \lle 1,k_r\rre$ such that $\phi_t(c_t)=\phi_r(c_r)$: indeed, this is because each edge $\{t,r\}$ is included in a hyperedge of $H_\pi$. Let us also fix these elements $c_r$; there are
$$(k_1)^{\deg(1,T)}\,(k_2)^{\deg(2,T)}\cdots (k_s)^{\deg(s,T)}$$
possibilities for these choices.
\begin{lemma}
With $T$, $(i,j)$ and $\{(c_t,c_r)\,|\,\{t,r\} \in E_T\}$ fixed as above, the number of compatible maps $\phi$ is smaller than 
$$
\begin{cases}
\frac{(k_i)^2}{2}\,n^{k_1+k_2+\cdots+k_s-s}&\text{if }i=j,\\
k_ik_j\,n^{k_1+\cdots+k_2+k_s-s}&\text{if }i\neq j.
\end{cases}
$$
\end{lemma}

\begin{proof}
Let us first draw the graph with vertex set $V = \{(r,a_r)\,|\,r\in \lle 1,s\rre,\,\,a_r \in \lle 1,k_r\rre\}$ and with one edge between $(t,c_t)$ and $(r,c_r)$ for each edge $\{t,r\} \in E_T$. This is a forest $C=\bigsqcup_{\gamma \in \Gamma} C_\gamma$ on $V=\bigsqcup_{r=1}^s V_{F_r}$, and for each tree $C_\gamma$ of this forest, all the vertices of $C_\gamma$ are by hypothesis sent to the same image by $\phi$. Moreover, the image of $C$ by the map $(r,a_r)\mapsto r$ is the tree $T$; see Figure \ref{fig:forest_c_associated_to_t} for an example with $s=6$ and $k_1=k_2=\cdots=k_s=3$.
\begin{figure}[ht]
\begin{center}
\begin{tikzpicture}[scale=1]
\draw (0,0) -- (0,-2);
\draw (0,-1) -- (1,-1);
\draw (1,0) -- (1,-2);
\foreach \x in {(0,0),(0,-1),(0,-2),(1,0),(1,-1),(1,-2)}
\fill \x circle (2pt);
\draw (-0.3,0) node {$1$};
\draw (-0.3,-1) node {$2$};
\draw (-0.3,-2) node {$3$};
\draw (1.3,0) node {$4$};
\draw (1.3,-1) node {$5$};
\draw (1.3,-2) node {$6$};
\draw (0.5,-2.5) node {$T$};
\begin{scope}[shift={(4,0)}]
\filldraw [fill=white!90!black] (-0.3,-0.3) rectangle (2.3,0.3);
\filldraw [fill=white!90!black, shift={(0,-1)}] (-0.3,-0.3) rectangle (2.3,0.3);
\filldraw [fill=white!90!black, shift={(0,-2)}] (-0.3,-0.3) rectangle (2.3,0.3);
\filldraw [fill=white!90!black, shift={(4,0)}] (-0.3,-0.3) rectangle (2.3,0.3);
\filldraw [fill=white!90!black, shift={(4,-1)}] (-0.3,-0.3) rectangle (2.3,0.3);
\filldraw [fill=white!90!black, shift={(4,-2)}] (-0.3,-0.3) rectangle (2.3,0.3);
\foreach \x in {0,1,2,4,5,6}
{\foreach \y in {0,-1,-2}
\fill (\x,\y) circle (2pt);};
\draw (0,0) -- (1,-1) -- (0,-2);
\draw (2,-1) -- (4,-1);
\draw (5,0) -- (6,-1) -- (5,-2);
\draw (3,-2.5) node {$C$};
\end{scope}
\end{tikzpicture}
\end{center}
\caption{Forest $C$ associated to $T$ and to a fixed set $\{(c_t,c_r)\,|\,\{t,r\} \in E_T\}$.\label{fig:forest_c_associated_to_t}}
\end{figure}
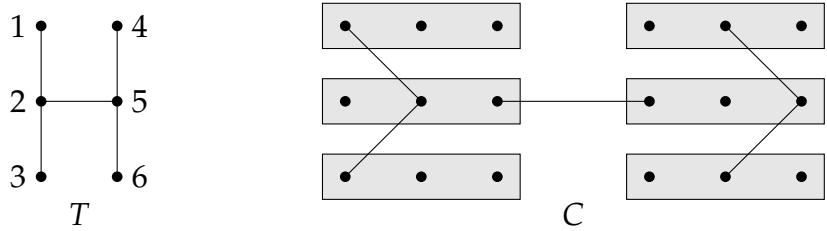
Let us replace each tree $C_\gamma$ by a part containing all the vertices of $C_\gamma$. We then obtain a set partition $\pi'$ which is finer than the set partition $\pi = \pi(\phi)$, see Figure \ref{fig:pi_prime} (by finer we mean that each part of $\pi'$ is included in a part of $\pi$; we denote this by $\pi' \preceq \pi$).
\begin{figure}[ht]
\begin{center}
\begin{tikzpicture}[scale=1]
\filldraw [fill=white!90!black] (-0.3,-0.3) rectangle (2.3,0.3);
\filldraw [fill=white!90!black, shift={(0,-1)}] (-0.3,-0.3) rectangle (2.3,0.3);
\filldraw [fill=white!90!black, shift={(0,-2)}] (-0.3,-0.3) rectangle (2.3,0.3);
\filldraw [fill=white!90!black, shift={(4,0)}] (-0.3,-0.3) rectangle (2.3,0.3);
\filldraw [fill=white!90!black, shift={(4,-1)}] (-0.3,-0.3) rectangle (2.3,0.3);
\filldraw [fill=white!90!black, shift={(4,-2)}] (-0.3,-0.3) rectangle (2.3,0.3);
\draw [red, thick] (2,-1) -- (4,-1);
\filldraw [red,fill=red!50!white,thick] (0,0) -- (1,-1) -- (0,-2) -- (0.5,-1) -- (0,0);
\filldraw [red,fill=red!50!white,thick] (5,0) -- (6,-1) -- (5,-2) -- (5.5,-1) -- (5,0);
\foreach \x in {0,1,2,4,5,6}
{\foreach \y in {0,-1,-2}
\fill (\x,\y) circle (2pt);};
\end{tikzpicture}
\end{center}
\caption{The set partition $\pi'$.\label{fig:pi_prime}}
\end{figure}
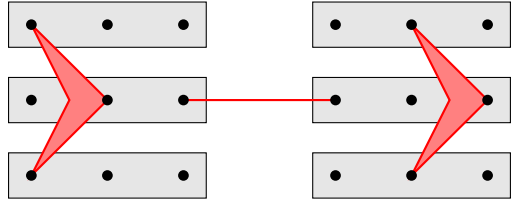
Now, because of the inclusion $E_T \sqcup \{i,j\} \subset E_G$ for some graph reduction $G$ of the hypergraph $H_\pi$, the set partition $\pi'$ is strictly finer than $\pi=\pi(\phi)$, and more precisely:
\begin{itemize}
    \item If $i =j$, then there must be two elements $d_i\neq d_i' \in \lle 1,k_i\rre$ which are not in the same part of $\pi'$, but which are in the same part of $\pi$.
    \item If $i\neq j$, then there must be two elements $d_i \in \lle 1,k_i\rre$ and $d_j \in \lle 1,k_j\rre$ which are not in the same part of $\pi'$, but which are in the same part of $\pi$. 
\end{itemize}
Hence, we can construct another set partition $\pi''$ of $V$ such that $\pi'\preceq \pi'' \preceq \pi$ and such that $\pi''$ is obtained from $\pi'$ by putting in the same part $d_i \neq d_i'$ in the case $i=j$, and by putting in the same part $d_i \neq d_j$ in the case $i \neq j$. Note then that the number of parts of $\pi'$ (including singletons) is $k_1+\cdots+k_s-(s-1)$, and therefore that the number of parts of $\pi''$ is $k_1+\cdots+k_s-s$. There are always less than $\frac{(k_i)^2}{2}$ or $k_ik_j$ possibilities for $\pi'$, and any map $\phi$ such that $\pi'' \preceq \pi(\phi)$ is determined by $n^{k_1+\cdots+k_s-s}$ images, whence the result.
\end{proof}

\begin{proposition}\label{prop:upper_bound_beta}
For any graphon $\gamma$ and any finite graphs $F_1,\ldots,F_s$ with sizes $k_1,\ldots,k_s \geq 2$,
\begin{align*}
&\left|  \sum_{\substack{\pi \in \pym(V) \\ H_{\pi} \text{ is connected} \\ H_{\pi} \text{ is not a hypertree} }}\sum_{\substack{\phi_1 : \lle 1,k_1 \rre \to \lle 1,n\rre \\ \vdots \qquad\vdots \vspace*{1mm}\\ \phi_s : \lle 1,k_s\rre \to \lle 1,n\rre \\ \pi(\phi) = \pi}} \kappa(A_{\phi_1}(F_1),\ldots,A_{\phi_s}(F_s)) \right| \\
&\leq 2^{s-2} (k_1+\cdots+k_s)^{s}\,k_1k_2\cdots k_s\,n^{k_1+\cdots+k_s-s}.
\end{align*}
\end{proposition}
\begin{proof}
The previous discussion shows that
\begin{align*}
|\beta| &\leq 2^{s-1}\,n^{k_1+\cdots+k_s-s} \left(\sum_{i=1}^s \frac{(k_i)^2}{2} + \sum_{i \neq j} k_ik_j\right)\left(\sum_{T \in \ST(K_s))} (k_1)^{\deg(1,T)} \cdots (k_s)^{\deg(s,T)}\right)\\
&\leq 2^{s-2}\,n^{k_1+\cdots+k_s-s}\,(k_1+\cdots+k_s)^2 \sum_{T \in \ST(K_s))}(k_1)^{\deg(1,T)}\cdots  (k_s)^{\deg(s,T)},
\end{align*}
and we conclude again by using the formula for the generating function of the degrees of the Cayley trees.
\end{proof}

Combining Propositions \ref{prop:upper_bound_alpha} and \ref{prop:upper_bound_beta}, we get:
\begin{theorem}\label{thm:upper_bound}
For any singular graphon $\gamma$ and any finite graphs $F_1,\ldots,F_s$ with sizes $k_1,\ldots,k_s \geq 2$,
$$|\kappa(S_n(F_1),\ldots,S_n(F_s))| \leq 2^{s-1}\,(k_1+\cdots+k_s)^{s}\,k_1k_2\cdots k_s\,n^{k_1+\cdots+k_s-s}.$$
\end{theorem}
\medskip

\subsection{Asymptotics of the graph densities in the singular case}
Given a singular graphon $\gamma$ and a finite graph $F$ with size $k$, recall that
$$Y_n(F) = n(t(F,G_n(\gamma)) - \esper[t(F,G_n(\gamma))]) = \frac{S_n(F)-\esper[S_n(F)]}{n^{k-1}}.$$
The discussion from Subsection \ref{sub:moments_graph_densities} shows that any joint moment of the variables $S_n(F)$ is a polynomial in $n$; the same result holds for the joint cumulants. When $\gamma$ is singular, Theorem \ref{thm:upper_bound} shows that for $s \geq 2$, the polynomial $\kappa(S_n(F_1),\ldots,S_n(F_s))$ is of degree smaller than $|V_{F_1}|+\cdots+|V_{F_s}|-s$, and with a leading term smaller in absolute value than $2^{s-1}\,(k_1+\cdots+k_s)^{s}\,k_1k_2\cdots k_s$. After rescaling the random variables $S_n(F)$, we therefore get:
\begin{corollary}
For  any singular graphon $\gamma$ and any finite graphs $F_1,\ldots,F_s$ with sizes $k_1,\ldots,k_s \geq 2$, $\kappa(Y_n(F_1),\ldots,Y_n(F_s))$ is a polynomial in the variable $n^{-1}$, and we have the upper bound:
$$|\kappa(Y_n(F_1),\ldots,Y_n(F_s))| \leq 2^{s-1}\,(k_1+\cdots+k_s)^{s}\,k_1k_2\cdots k_s$$
for any $n \geq 1$. In particular, there is a limit $\kappa_\infty(F_1,\ldots,F_s) = \lim_{n\to \infty} \kappa(Y_n(F_1),\ldots,Y_n(F_s))$ which is smaller in absolute value than $2^{s-1}\,(k_1+\cdots+k_s)^{s}\,k_1k_2\cdots k_s$.
\end{corollary}
\medskip

We can finally prove the Main Theorem. The generating series $\log (\esper[\E^{z_1Y_n(F_1)+\cdots+z_sY_n(F_s)}])$ of the cumulants is for any $n\geq 1$ absolutely convergent on the polydisk
$$\left\{|z_1|<  \frac{1}{2s\E K^2},\ldots,|z_s|< \frac{1}{2s\E K^2} \right\}$$
with $K = \max\{k_1,\ldots,k_s\}$. Indeed, if $|Z| = \max\{|z_1|,\ldots,|z_s|\}$, then:
\begin{align*}
&\sum_{r=2}^\infty \sum_{\substack{r_1,r_2,\ldots,r_s \geq 0 \\ r_1+r_2+\cdots+r_s=r}} \frac{|z_1|^{r_1}|z_2|^{r_2}\cdots |z_s|^{r_s}}{r_1!\,r_2!\cdots r_s!}\,\,\kappa\!\left((Y_n(F_1))^{\otimes r_1},\ldots,(Y_n(F_s))^{\otimes r_s}\right)\\
&\leq \sum_{r=2}^\infty \sum_{\substack{r_1,r_2,\ldots,r_s \geq 0 \\ r_1+r_2+\cdots+r_s=r}} \frac{(2|Z|)^{r}}{r_1!\cdots r_s!} \,(r_1k_1+\cdots+r_sk_s)^{r} (k_1)^{r_1}\cdots (k_s)^{r_s} \\
&\leq \sum_{r=2}^\infty \sum_{\substack{r_1,r_2,\ldots,r_s \geq 0 \\ r_1+r_2+\cdots+r_s=r}} \frac{(2K^2|Z|r)^{r}}{r_1!\cdots r_s!} \leq \sum_{r=2}^\infty (2\E K^2|Z|)^{r} \left(\sum_{\substack{r_1,r_2,\ldots,r_s \geq 0 \\ r_1+r_2+\cdots+r_s=r}}  \frac{r!}{r_1!\cdots r_s!}\right) = \sum_{r=2}^\infty (2s\E K^2|Z|)^{r} ,
\end{align*}
which is convergent if $|Z|<\frac{1}{2s\E K^2}$. All the coefficients of the series converge when $n$ goes to infinity. Therefore, we have uniform convergence on a polydisk of the joint Laplace transform of the variables $Y_n(F_1),\ldots,Y_n(F_s)$ towards a power series which is also absolutely convergent on this polydisk. This implies the convergence in distribution, the limit being a random vector whose law is determined by its joint moments or its joint cumulants (see for instance \cite{Bil95}*{p.~390}).
\medskip

We cannot identify the distribution of the limiting random variables $Y_\infty(F)$, but the strong upper bound on cumulants implies a concentration inequality:
\begin{proposition}
For any singular graphon $\gamma$, any finite graph $F$ on $k$ vertices, any $x >0$ and any $n \geq 1$,
$$\proba[|Y_n(F)| \geq \E k^2 x] \leq 2 \exp\left(\sqrt{1+x}-1-\frac{x}{2}\right),$$
and the same inequality holds for the limit in distribution $Y_\infty(F)$.
\end{proposition}
\begin{proof}
We prove an inequality for $\proba[Y_n(F) \geq y]$ with $\E k^2 x = y>0$; the case of $\proba[Y_n(F) \leq -y]$ is symmetric. By Chernoff's inequality, for any $t \geq 0$,
\begin{align*}
\log (\proba[Y_n(F) \geq y])& \leq -ty + \log (\esper[\E^{tY_n(F)}]) \leq -ty + \sum_{r=2}^\infty \frac{\kappa^{(r)}(Y_n(F))}{r!}\,t^r \\
&\leq -ty + \frac{1}{2}\sum_{r=2}^\infty (2\E k^2 t)^{r} = -ty + \frac{\lambda^2 t^2 }{2(1-\lambda t)}
\end{align*}
with $\lambda = 2\E k^2$ and $t < \frac{1}{\lambda}$. The optimal choice is 
$$t = \frac{1}{\lambda} \left(1-\frac{1}{\sqrt{1+\frac{2y}{\lambda}}}\right),$$ and replacing in the inequality above yields the result.
\end{proof}
\medskip

\subsection{Gaussian characterisation of the constant graphons}
By using the equations established in Section \ref{sec:spectral}, let us compute explicitly in terms of the observables of a singular graphon $\gamma$ with edge density $p$ the third and fourth cumulants of the random edge count $S_n(\edgegraph)$. By using Proposition \ref{prop:moment}, we can compute $\esper[S_n(\edgegraph^r)]$ for $r \in \{2,3,4\}$. More precisely, let us enumerate the set partitions $\pi$ which yield loopless graphs $(\edgegraph^r)\downarrow \pi$ according to their type $t(\pi)$, which is an integer partition of size $2r$ without part of size larger than $r$ (if we put more than $r+1$ elements in a common part, then the contracted graph has necessarily a loop). Each integer partition corresponds to several possible contractions, which are not always isomorphic. The table hereafter lists all the possible contracted graphs and their multiplicities when $r=2$. \\
\begin{table}[ht]
\begin{center}
\begin{tabular}{|c|c|c|}
\hline $t(\pi)$ & contracted graph & multiplicity \\
\hline\hline $(1^4)$ & $\begin{tikzpicture}[baseline=1mm,scale=0.5]
\foreach \x in {(0,0),(1,0),(0,0.75),(1,0.75)}
\fill \x circle (4pt);
\draw (0,0) -- (1,0);
\draw (0,0.75) -- (1,0.75);
\end{tikzpicture}$ & $1$\\[2mm]
\hline $(2\,1^2)$ & $\begin{tikzpicture}[baseline=1mm,scale=0.5]
\filldraw [black,fill=white!90!black] (-0.2,-0.2) rectangle (0.2,0.95);
\foreach \x in {(0,0),(1,0),(0,0.75),(1,0.75)}
\fill \x circle (4pt);
\draw (0,0) -- (1,0);
\draw (0,0.75) -- (1,0.75);
\begin{scope}[shift={(3,0)}]
\draw (-1,0) node {$\rightarrow$};
\foreach \x in {(0,0),(1,0),(2,0)}
\fill \x circle (4pt);
\draw (0,0) -- (2,0);
\end{scope}
\end{tikzpicture}$ & $4$\\[2mm]
\hline $(2^2)$ & $\begin{tikzpicture}[baseline=1mm,scale=0.5]
\filldraw [black,fill=white!90!black] (-0.2,-0.2) rectangle (0.2,0.95);
\filldraw [black,fill=white!90!black] (0.8,-0.2) rectangle (1.2,0.95);
\foreach \x in {(0,0),(1,0),(0,0.75),(1,0.75)}
\fill \x circle (4pt);
\draw (0,0) -- (1,0);
\draw (0,0.75) -- (1,0.75);
\begin{scope}[shift={(3,0)}]
\draw (-1,0) node {$\rightarrow$};
\foreach \x in {(0,0),(1,0)}
\fill \x circle (4pt);
\draw (0,0) -- (1,0);
\end{scope}
\end{tikzpicture}$ & $2$ \\[2mm]
\hline 
\end{tabular}
\vspace*{3mm}

\end{center}
\caption{List of contracted graphs when $r=2$.}\label{tab:r2}
\end{table}

\noindent This implies:
\begin{align*}
\esper[S_n(\edgegraph^2)] &= n^{\downarrow 4}\,t(\edgegraph,\gamma)^2 + 4\,n^{\downarrow 3}\,t(\begin{tikzpicture}[baseline=-1mm,scale=0.5]
\foreach \x in {(0,0),(1,0),(2,0)}
\fill \x circle (4pt);
\draw (0,0) -- (2,0);
\end{tikzpicture},\gamma) + 2\,n^{\downarrow 2}\,t(\edgegraph,\gamma) \\
&= n^{\downarrow 3}(n+1)\,p^2 + 2n^{\downarrow 2}\,p;\\
\kappa^{(2)}(Y_n(\edgegraph)) &= 2\,p(1-p)\left(1-\frac{1}{n}\right).
\end{align*}
Similarly, for $r=3$, the possible contractions are listed in Table \ref{tab:r3}. By using Corollary \ref{cor:tree_densities}, we can compute all the densities in $\gamma$ of these contracted graphs , except the density of the triangle $C_3$. The relation between moments and cumulants leads then to:

\begin{align*}
\esper[S_n(\edgegraph^3)] 
&= n^{\downarrow 4} (n^2 +3n +4 ) p^3 + 6 n^{\downarrow 3}(n+1)p^2 + 4n^{\downarrow 2}\,p + 8n^{\downarrow 3}\,t(C_3,\gamma);\\
\kappa^{(3)}(S_n(\edgegraph)) &= 8\,n^{\downarrow 3}\,(t(C_3,\gamma) -p^3)  + 4\,n^{\downarrow 2} p(p-1)(2p-1);\\
\kappa^{(3)}(Y_n(\edgegraph)) &= 8\,(t(C_3,\gamma) -p^3) + O\!\left(\frac{1}{n}\right).
\end{align*}

\begin{table}[ht]
\begin{center}
\begin{tabular}{|c|c|c|}
\hline $t(\pi)$ & contracted graph & multiplicity \\
\hline\hline $(1^6)$ & $\begin{tikzpicture}[baseline=1mm,scale=0.5]
\foreach \x in {(0,0),(1,0),(0,0.75),(1,0.75),(0,-0.75),(1,-0.75)}
\fill \x circle (4pt);
\draw (0,0) -- (1,0);
\draw (0,0.75) -- (1,0.75);
\draw (0,-0.75) -- (1,-0.75);
\end{tikzpicture}$ & $1$ \\[5mm]
\hline $(2\,1^4)$ & $\begin{tikzpicture}[baseline=1mm,scale=0.5]
\filldraw [black,fill=white!90!black] (-0.2,-0.2) rectangle (0.2,0.95);
\foreach \x in {(0,0),(1,0),(0,0.75),(1,0.75),(0,-0.75),(1,-0.75)}
\fill \x circle (4pt);
\draw (0,0) -- (1,0);
\draw (0,0.75) -- (1,0.75);
\draw (0,-0.75) -- (1,-0.75);
\begin{scope}[shift={(3,0)}]
\draw (-1,0) node {$\rightarrow$};
\foreach \x in {(0,0.75),(1,0.75),(2,0.75),(1,0),(0,0)}
\fill \x circle (4pt);
\draw (0,0) -- (1,0);
\draw (0,0.75) -- (2,0.75);
\end{scope}
\end{tikzpicture}$ & $12$\\[5mm]
\hline $(3\,1^3)$ & $\begin{tikzpicture}[baseline=1mm,scale=0.5]
\filldraw [black,fill=white!90!black] (-0.2,-0.95) rectangle (0.2,0.95);
\foreach \x in {(0,0),(1,0),(0,0.75),(1,0.75),(0,-0.75),(1,-0.75)}
\fill \x circle (4pt);
\draw (0,0) -- (1,0);
\draw (0,0.75) -- (1,0.75);
\draw (0,-0.75) -- (1,-0.75);
\begin{scope}[shift={(3,0)}]
\draw (-1,0) node {$\rightarrow$};
\foreach \x in {(0,0),(1,0.75),(1,0),(2,0)}
\fill \x circle (4pt);
\draw (0,0) -- (2,0);
\draw (1,0.75) -- (1,0);
\end{scope}
\end{tikzpicture}$ & $8$\\[5mm]
\hline $(2^2\,1^2)$ & $\begin{tikzpicture}[baseline=1mm,scale=0.5]
\filldraw [black,fill=white!90!black] (-0.2,-0.2) rectangle (0.2,0.95);
\filldraw [black,fill=white!90!black] (0.8,-0.2) rectangle (1.2,0.95);
\foreach \x in {(0,0),(1,0),(0,0.75),(1,0.75),(0,-0.75),(1,-0.75)}
\fill \x circle (4pt);
\draw (0,0) -- (1,0);
\draw (0,0.75) -- (1,0.75);
\draw (0,-0.75) -- (1,-0.75);
\begin{scope}[shift={(3,0)}]
\draw (-1,0) node {$\rightarrow$};
\foreach \x in {(0,0),(1,0),(0,0.75),(1,0.75)}
\fill \x circle (4pt);
\draw (0,0) -- (1,0);
\draw (0,0.75) -- (1,0.75);
\end{scope}
\end{tikzpicture}$ & $6$\\[5mm]
\hline $(2^2\,1^2)$ & $\begin{tikzpicture}[baseline=1mm,scale=0.5]
\filldraw [black,fill=white!90!black] (-0.2,-0.2) rectangle (0.2,0.95);
\filldraw [black,fill=white!90!black] (0.8,-0.95) rectangle (1.2,0.2);
\foreach \x in {(0,0),(1,0),(0,0.75),(1,0.75),(0,-0.75),(1,-0.75)}
\fill \x circle (4pt);
\draw (0,0) -- (1,0);
\draw (0,0.75) -- (1,0.75);
\draw (0,-0.75) -- (1,-0.75);
\begin{scope}[shift={(3,0)}]
\draw (-1,0) node {$\rightarrow$};
\foreach \x in {(0,0),(1,0),(2,0),(3,0)}
\fill \x circle (4pt);
\draw (0,0) -- (3,0);
\end{scope}
\end{tikzpicture}$ & $24$ \\[5mm]
\hline $(3\,2\,1)$& $\begin{tikzpicture}[baseline=1mm,scale=0.5]
\filldraw [black,fill=white!90!black] (-0.2,-0.95) rectangle (0.2,0.95);
\filldraw [black,fill=white!90!black] (0.8,-0.2) rectangle (1.2,0.95);
\foreach \x in {(0,0),(1,0),(0,0.75),(1,0.75),(0,-0.75),(1,-0.75)}
\fill \x circle (4pt);
\draw (0,0) -- (1,0);
\draw (0,0.75) -- (1,0.75);
\draw (0,-0.75) -- (1,-0.75);
\begin{scope}[shift={(3,0)}]
\draw (-1,0) node {$\rightarrow$};
\foreach \x in {(0,0),(1,0),(2,0)}
\fill \x circle (4pt);
\draw (0,0) -- (2,0);
\end{scope}
\end{tikzpicture}$ & $24$\\[5mm]
\hline $(2^3)$ & $\begin{tikzpicture}[baseline=1mm,scale=0.5]
\filldraw [black,fill=white!90!black] (-0.2,-0.2) rectangle (0.2,0.95);
\filldraw [black,fill=white!90!black] (0.8,-0.95) rectangle (1.2,0.2);
\filldraw [black,fill=white!90!black] (-0.2,-0.55) -- (0.4,-0.55) -- (0.4,0.95) -- (1.2,0.95) -- (1.2,0.55) -- (0.6,0.55) -- (0.6,-0.95) -- (-0.2,-0.95) -- (-0.2,-0.55);
\foreach \x in {(0,0),(1,0),(0,0.75),(1,0.75),(0,-0.75),(1,-0.75)}
\fill \x circle (4pt);
\draw (0,0) -- (1,0);
\draw (0,0.75) -- (1,0.75);
\draw (0,-0.75) -- (1,-0.75);
\begin{scope}[shift={(3,0)}]
\draw (-1,0) node {$\rightarrow$};
\foreach \x in {(0,0),(1,0),(0.5,0.75)}
\fill \x circle (4pt);
\draw (0,0) -- (1,0) -- (0.5,0.75) -- (0,0);
\end{scope}
\end{tikzpicture}$ & $8$\\[5mm]
\hline $(3^2)$ & $\begin{tikzpicture}[baseline=1mm,scale=0.5]
\filldraw [black,fill=white!90!black] (-0.2,-0.95) rectangle (0.2,0.95);
\filldraw [black,fill=white!90!black] (0.8,-0.95) rectangle (1.2,0.95);
\foreach \x in {(0,0),(1,0),(0,0.75),(1,0.75),(0,-0.75),(1,-0.75)}
\fill \x circle (4pt);
\draw (0,0) -- (1,0);
\draw (0,0.75) -- (1,0.75);
\draw (0,-0.75) -- (1,-0.75);
\begin{scope}[shift={(3,0)}]
\draw (-1,0) node {$\rightarrow$};
\foreach \x in {(0,0),(1,0)}
\fill \x circle (4pt);
\draw (0,0) -- (1,0);
\end{scope}
\end{tikzpicture}$ & $4$\\[5mm]
\hline 
\end{tabular}

\vspace*{3mm}
\end{center}
\caption{List of contracted graphs when $r=3$.}\label{tab:r3}
\end{table}

Finally, for $r=4$, the possible contractions and their multiplicities are listed in Table \ref{tab:r4}. Note that since the motive $\begin{tikzpicture}[baseline=-1mm,scale=0.5]
\foreach \x in {(0,0),(1,0),(2,0.5),(2,-0.5)}
\fill \x circle (4pt);
\draw (0,0) -- (1,0) -- (2,0.5) -- (2,-0.5) -- (1,0);
\end{tikzpicture}\,$ is join-transitive, we have
$$t\big(\begin{tikzpicture}[baseline=-1mm,scale=0.5]
\foreach \x in {(0,0),(1,0),(2,0.5),(2,-0.5)}
\fill \x circle (4pt);
\draw (0,0) -- (1,0) -- (2,0.5) -- (2,-0.5) -- (1,0);
\end{tikzpicture},\,\gamma\big) = p\,\times \,t(C_3,\gamma).$$\\

\begin{table}
\begin{center}
\begin{tabular}{|c|c|c||c|c|c|}
\hline $t(\pi)$ & contracted graph & multiplicity & $t(\pi)$ & contracted graph & multiplicity \\
\hline\hline $(1^8)$ & $\begin{tikzpicture}[baseline=1mm,scale=0.5]
\foreach \x in {(0,0),(1,0),(0,0.75),(1,0.75),(0,-0.75),(1,-0.75),(0,-1.5),(1,-1.5)}
\fill \x circle (4pt);
\draw (0,0) -- (1,0);
\draw (0,0.75) -- (1,0.75);
\draw (0,-0.75) -- (1,-0.75);
\draw (0,-1.5) -- (1,-1.5);
\end{tikzpicture}$ & $1$ & $(4\,2\,1^2)$ & $\begin{tikzpicture}[baseline=1mm,scale=0.5]
\filldraw [black,fill=white!90!black] (-0.2,-1.7) rectangle (0.2,0.95);
\filldraw [black,fill=white!90!black] (0.8,-0.2) rectangle (1.2,0.95);
\foreach \x in {(0,0),(1,0),(0,0.75),(1,0.75),(0,-0.75),(1,-0.75),(0,-1.5),(1,-1.5)}
\fill \x circle (4pt);
\draw (0,0) -- (1,0);
\draw (0,0.75) -- (1,0.75);
\draw (0,-0.75) -- (1,-0.75);
\draw (0,-1.5) -- (1,-1.5);
\begin{scope}[shift={(3,0)}]
\draw (-1,0) node {$\rightarrow$};
\foreach \x in {(0,0),(1,0),(2,0),(1,-0.75)}
\fill \x circle (4pt);
\draw (0,0) -- (2,0);
\draw (1,0) -- (1,-0.75);
\end{scope}
\end{tikzpicture}$ & $96$ \\[9mm]
\hline $(2\,1^6)$ & $\begin{tikzpicture}[baseline=1mm,scale=0.5]
\filldraw [black,fill=white!90!black] (-0.2,-0.2) rectangle (0.2,0.95);
\foreach \x in {(0,0),(1,0),(0,0.75),(1,0.75),(0,-0.75),(1,-0.75),(0,-1.5),(1,-1.5)}
\fill \x circle (4pt);
\draw (0,0) -- (1,0);
\draw (0,0.75) -- (1,0.75);
\draw (0,-0.75) -- (1,-0.75);
\draw (0,-1.5) -- (1,-1.5);
\begin{scope}[shift={(3,0)}]
\draw (-1,0) node {$\rightarrow$};
\foreach \x in {(0,0),(1,0),(0,0.75),(1,0.75),(2,0.75),(0,-0.75),(1,-0.75)}
\fill \x circle (4pt);
\draw (0,0) -- (1,0);
\draw (0,0.75) -- (2,0.75);
\draw (0,-0.75) -- (1,-0.75);
\end{scope}
\end{tikzpicture}$ & $24$ & $(3^2\,1^2)$ & $\begin{tikzpicture}[baseline=1mm,scale=0.5]
\filldraw [black,fill=white!90!black] (-0.2,-0.95) rectangle (0.2,0.95);
\filldraw [black,fill=white!90!black] (0.8,-0.95) rectangle (1.2,0.95);
\foreach \x in {(0,0),(1,0),(0,0.75),(1,0.75),(0,-0.75),(1,-0.75),(0,-1.5),(1,-1.5)}
\fill \x circle (4pt);
\draw (0,0) -- (1,0);
\draw (0,0.75) -- (1,0.75);
\draw (0,-0.75) -- (1,-0.75);
\draw (0,-1.5) -- (1,-1.5);
\begin{scope}[shift={(3,0)}]
\draw (-1,0) node {$\rightarrow$};
\foreach \x in {(0,0),(1,0),(0,-0.75),(1,-0.75)}
\fill \x circle (4pt);
\draw (0,0) -- (1,0);
\draw (0,-0.75) -- (1,-0.75);
\end{scope}
\end{tikzpicture}$ & $16$ \\[9mm]
\hline $(3\,1^5)$ & $\begin{tikzpicture}[baseline=1mm,scale=0.5]
\filldraw [black,fill=white!90!black] (-0.2,-0.95) rectangle (0.2,0.95);
\foreach \x in {(0,0),(1,0),(0,0.75),(1,0.75),(0,-0.75),(1,-0.75),(0,-1.5),(1,-1.5)}
\fill \x circle (4pt);
\draw (0,0) -- (1,0);
\draw (0,0.75) -- (1,0.75);
\draw (0,-0.75) -- (1,-0.75);
\draw (0,-1.5) -- (1,-1.5);
\begin{scope}[shift={(3,0)}]
\draw (-1,0) node {$\rightarrow$};
\foreach \x in {(0,0),(1,0),(2,0),(1,0.75),(0,-0.75),(1,-0.75)}
\fill \x circle (4pt);
\draw (0,0) -- (2,0);
\draw (1,0) -- (1,0.75);
\draw (0,-0.75) -- (1,-0.75);
\end{scope}
\end{tikzpicture}$ & $32$ & $(3^2\,1^2)$ & $\begin{tikzpicture}[baseline=1mm,scale=0.5]
\filldraw [black,fill=white!90!black] (-0.2,-0.95) rectangle (0.2,0.95);
\filldraw [black,fill=white!90!black] (0.8,-1.7) rectangle (1.2,0.2);
\foreach \x in {(0,0),(1,0),(0,0.75),(1,0.75),(0,-0.75),(1,-0.75),(0,-1.5),(1,-1.5)}
\fill \x circle (4pt);
\draw (0,0) -- (1,0);
\draw (0,0.75) -- (1,0.75);
\draw (0,-0.75) -- (1,-0.75);
\draw (0,-1.5) -- (1,-1.5);
\begin{scope}[shift={(3,0)}]
\draw (-1,0) node {$\rightarrow$};
\foreach \x in {(0,0),(1,0),(2,0),(3,0)}
\fill \x circle (4pt);
\draw (0,0) -- (3,0);
\end{scope}
\end{tikzpicture}$ & $96$ \\[9mm]
\hline $(2^2\,1^4)$ & $\begin{tikzpicture}[baseline=1mm,scale=0.5]
\filldraw [black,fill=white!90!black] (-0.2,-0.2) rectangle (0.2,0.95);
\filldraw [black,fill=white!90!black] (0.8,-0.2) rectangle (1.2,0.95);
\foreach \x in {(0,0),(1,0),(0,0.75),(1,0.75),(0,-0.75),(1,-0.75),(0,-1.5),(1,-1.5)}
\fill \x circle (4pt);
\draw (0,0) -- (1,0);
\draw (0,0.75) -- (1,0.75);
\draw (0,-0.75) -- (1,-0.75);
\draw (0,-1.5) -- (1,-1.5);
\begin{scope}[shift={(3,0)}]
\draw (-1,0) node {$\rightarrow$};
\foreach \x in {(0,0),(1,0),(0,0.75),(1,0.75),(0,-0.75),(1,-0.75)}
\fill \x circle (4pt);
\draw (0,0) -- (1,0);
\draw (0,0.75) -- (1,0.75);
\draw (0,-0.75) -- (1,-0.75);
\end{scope}
\end{tikzpicture}$ & $12$ & $(3\,2^2\,1)$ & $\begin{tikzpicture}[baseline=1mm,scale=0.5]
\filldraw [black,fill=white!90!black] (-0.2,-0.95) rectangle (0.2,0.95);
\filldraw [black,fill=white!90!black] (0.8,-0.2) rectangle (1.2,0.95);
\filldraw [black,fill=white!90!black] (0.8,-1.7) rectangle (1.2,-0.55);
\foreach \x in {(0,0),(1,0),(0,0.75),(1,0.75),(0,-0.75),(1,-0.75),(0,-1.5),(1,-1.5)}
\fill \x circle (4pt);
\draw (0,0) -- (1,0);
\draw (0,0.75) -- (1,0.75);
\draw (0,-0.75) -- (1,-0.75);
\draw (0,-1.5) -- (1,-1.5);
\begin{scope}[shift={(3,0)}]
\draw (-1,0) node {$\rightarrow$};
\foreach \x in {(0,0),(1,0),(2,0),(3,0)}
\fill \x circle (4pt);
\draw (0,0) -- (3,0);
\end{scope}
\end{tikzpicture}$ & $192$ \\[9mm]
\hline $(2^2\,1^4)$ & $\begin{tikzpicture}[baseline=1mm,scale=0.5]
\filldraw [black,fill=white!90!black] (-0.2,-0.2) rectangle (0.2,0.95);
\filldraw [black,fill=white!90!black] (0.8,-0.95) rectangle (1.2,0.2);
\foreach \x in {(0,0),(1,0),(0,0.75),(1,0.75),(0,-0.75),(1,-0.75),(0,-1.5),(1,-1.5)}
\fill \x circle (4pt);
\draw (0,0) -- (1,0);
\draw (0,0.75) -- (1,0.75);
\draw (0,-0.75) -- (1,-0.75);
\draw (0,-1.5) -- (1,-1.5);
\begin{scope}[shift={(3,0)}]
\draw (-1,0) node {$\rightarrow$};
\foreach \x in {(0,0),(1,0),(0,-0.75),(1,-0.75),(2,0),(3,0)}
\fill \x circle (4pt);
\draw (0,0) -- (3,0);
\draw (0,-0.75) -- (1,-0.75);
\end{scope}
\end{tikzpicture}$ & $96$ & $(3\,2^2\,1)$ & $\begin{tikzpicture}[baseline=1mm,scale=0.5]
\filldraw [black,fill=white!90!black] (-0.2,-0.95) rectangle (0.2,0.95);
\filldraw [black,fill=white!90!black] (-0.2,-1.3) -- (0.4,-1.3) -- (0.4,0.2) -- (1.2,0.2) -- (1.2,-0.2) -- (0.6,-0.2) -- (0.6,-1.7) -- (-0.2,-1.7) -- (-0.2,-1.3);
\filldraw [black,fill=white!90!black] (0.8,-1.7) rectangle (1.2,-0.55);
\foreach \x in {(0,0),(1,0),(0,0.75),(1,0.75),(0,-0.75),(1,-0.75),(0,-1.5),(1,-1.5)}
\fill \x circle (4pt);
\draw (0,0) -- (1,0);
\draw (0,0.75) -- (1,0.75);
\draw (0,-0.75) -- (1,-0.75);
\draw (0,-1.5) -- (1,-1.5);
\begin{scope}[shift={(3,0)}]
\draw (-1,0) node {$\rightarrow$};
\foreach \x in {(0,0),(1,0),(2,0.5),(2,-0.5)}
\fill \x circle (4pt);
\draw (0,0) -- (1,0) -- (2,0.5) -- (2,-0.5) -- (1,0);
\end{scope}
\end{tikzpicture}$ & $192$ \\[9mm]
\hline $(2^2\,1^4)$ & $\begin{tikzpicture}[baseline=1mm,scale=0.5]
\filldraw [black,fill=white!90!black] (-0.2,-0.2) rectangle (0.2,0.95);
\filldraw [black,fill=white!90!black] (-0.2,-1.7) rectangle (0.2,-0.55);
\foreach \x in {(0,0),(1,0),(0,0.75),(1,0.75),(0,-0.75),(1,-0.75),(0,-1.5),(1,-1.5)}
\fill \x circle (4pt);
\draw (0,0) -- (1,0);
\draw (0,0.75) -- (1,0.75);
\draw (0,-0.75) -- (1,-0.75);
\draw (0,-1.5) -- (1,-1.5);
\begin{scope}[shift={(3,0)}]
\draw (-1,0) node {$\rightarrow$};
\foreach \x in {(0,0),(1,0),(2,0),(2,-0.75),(0,-0.75),(1,-0.75)}
\fill \x circle (4pt);
\draw (0,0) -- (2,0);
\draw (0,-0.75) -- (2,-0.75);
\end{scope}
\end{tikzpicture}$ & $48$ & $(2^4)$ & $\begin{tikzpicture}[baseline=1mm,scale=0.5]
\filldraw [black,fill=white!90!black] (-0.2,-0.2) rectangle (0.2,0.95);
\filldraw [black,fill=white!90!black] (0.8,-0.2) rectangle (1.2,0.95);
\filldraw [black,fill=white!90!black] (-0.2,-1.7) rectangle (0.2,-0.55);
\filldraw [black,fill=white!90!black] (0.8,-1.7) rectangle (1.2,-0.55);
\foreach \x in {(0,0),(1,0),(0,0.75),(1,0.75),(0,-0.75),(1,-0.75),(0,-1.5),(1,-1.5)}
\fill \x circle (4pt);
\draw (0,0) -- (1,0);
\draw (0,0.75) -- (1,0.75);
\draw (0,-0.75) -- (1,-0.75);
\draw (0,-1.5) -- (1,-1.5);
\begin{scope}[shift={(3,0)}]
\draw (-1,0) node {$\rightarrow$};
\foreach \x in {(0,0),(1,0),(0,-0.75),(1,-0.75)}
\fill \x circle (4pt);
\draw (0,0) -- (1,0);
\draw (0,-0.75) -- (1,-0.75);
\end{scope}
\end{tikzpicture}$ & $12$ \\[9mm]
\hline $(4\,1^4)$ & $\begin{tikzpicture}[baseline=1mm,scale=0.5]
\filldraw [black,fill=white!90!black] (-0.2,-1.7) rectangle (0.2,0.95);
\foreach \x in {(0,0),(1,0),(0,0.75),(1,0.75),(0,-0.75),(1,-0.75),(0,-1.5),(1,-1.5)}
\fill \x circle (4pt);
\draw (0,0) -- (1,0);
\draw (0,0.75) -- (1,0.75);
\draw (0,-0.75) -- (1,-0.75);
\draw (0,-1.5) -- (1,-1.5);
\begin{scope}[shift={(3,0)}]
\draw (-1,0) node {$\rightarrow$};
\foreach \x in {(0,0),(1,0),(2,0),(1,-0.75),(1,0.75)}
\fill \x circle (4pt);
\draw (0,0) -- (2,0);
\draw (1,0.75) -- (1,-0.75);
\end{scope}
\end{tikzpicture}$ & $16$ & $(2^4)$ & $\begin{tikzpicture}[baseline=1mm,scale=0.5]
\filldraw [black,fill=white!90!black] (-0.2,-0.2) rectangle (0.2,0.95);
\filldraw [black,fill=white!90!black] (0.8,-0.95) rectangle (1.2,0.2);
\filldraw [black,fill=white!90!black] (-0.2,-1.7) rectangle (0.2,-0.55);
\filldraw [black,fill=white!90!black] (0.8,-1.7) -- (0.8,-1.3) -- (1.5,-1.3) -- (1.5,0.55) -- (0.8,0.55) -- (0.8,0.95) -- (1.9,0.95) -- (1.9,-1.7) -- (0.8,-1.7);
\foreach \x in {(0,0),(1,0),(0,0.75),(1,0.75),(0,-0.75),(1,-0.75),(0,-1.5),(1,-1.5)}
\fill \x circle (4pt);
\draw (0,0) -- (1,0);
\draw (0,0.75) -- (1,0.75);
\draw (0,-0.75) -- (1,-0.75);
\draw (0,-1.5) -- (1,-1.5);
\begin{scope}[shift={(3.8,0)}]
\draw (-1,0) node {$\rightarrow$};
\foreach \x in {(0,0),(1,0),(0,-0.75),(1,-0.75)}
\fill \x circle (4pt);
\draw (0,0) -- (1,0) -- (1,-0.75) -- (0,-0.75) -- (0,0);
\end{scope}
\end{tikzpicture}$ & $48$ \\[9mm]
\hline $(3\,2\,1^3)$ & $\begin{tikzpicture}[baseline=1mm,scale=0.5]
\filldraw [black,fill=white!90!black] (-0.2,-0.95) rectangle (0.2,0.95);
\filldraw [black,fill=white!90!black] (0.8,-0.2) rectangle (1.2,0.95);
\foreach \x in {(0,0),(1,0),(0,0.75),(1,0.75),(0,-0.75),(1,-0.75),(0,-1.5),(1,-1.5)}
\fill \x circle (4pt);
\draw (0,0) -- (1,0);
\draw (0,0.75) -- (1,0.75);
\draw (0,-0.75) -- (1,-0.75);
\draw (0,-1.5) -- (1,-1.5);
\begin{scope}[shift={(3,0)}]
\draw (-1,0) node {$\rightarrow$};
\foreach \x in {(0,0),(1,0),(2,0),(1,-0.75),(0,-0.75)}
\fill \x circle (4pt);
\draw (0,0) -- (2,0);
\draw (0,-0.75) -- (1,-0.75);
\end{scope}
\end{tikzpicture}$ & $96$ & $(4\,3\,1)$ & $\begin{tikzpicture}[baseline=1mm,scale=0.5]
\filldraw [black,fill=white!90!black] (-0.2,-1.7) rectangle (0.2,0.95);
\filldraw [black,fill=white!90!black] (0.8,-0.95) rectangle (1.2,0.95);
\foreach \x in {(0,0),(1,0),(0,0.75),(1,0.75),(0,-0.75),(1,-0.75),(0,-1.5),(1,-1.5)}
\fill \x circle (4pt);
\draw (0,0) -- (1,0);
\draw (0,0.75) -- (1,0.75);
\draw (0,-0.75) -- (1,-0.75);
\draw (0,-1.5) -- (1,-1.5);
\begin{scope}[shift={(3,0)}]
\draw (-1,0) node {$\rightarrow$};
\foreach \x in {(0,0),(1,0),(2,0)}
\fill \x circle (4pt);
\draw (0,0) -- (2,0);
\end{scope}
\end{tikzpicture}$ & $64$\\[9mm]
\hline $(3\,2\,1^3)$ & $\begin{tikzpicture}[baseline=1mm,scale=0.5]
\filldraw [black,fill=white!90!black] (-0.2,-0.95) rectangle (0.2,0.95);
\filldraw [black,fill=white!90!black] (0.8,-1.7) rectangle (1.2,-0.55);
\foreach \x in {(0,0),(1,0),(0,0.75),(1,0.75),(0,-0.75),(1,-0.75),(0,-1.5),(1,-1.5)}
\fill \x circle (4pt);
\draw (0,0) -- (1,0);
\draw (0,0.75) -- (1,0.75);
\draw (0,-0.75) -- (1,-0.75);
\draw (0,-1.5) -- (1,-1.5);
\begin{scope}[shift={(3,0)}]
\draw (-1,0) node {$\rightarrow$};
\foreach \x in {(0,0),(1,0),(2,0),(3,0),(1,-0.75)}
\fill \x circle (4pt);
\draw (0,0) -- (3,0);
\draw (1,0) -- (1,-0.75);
\end{scope}
\end{tikzpicture}$ & $192$ & $(4\,2^2)$ & $\begin{tikzpicture}[baseline=1mm,scale=0.5]
\filldraw [black,fill=white!90!black] (-0.2,-1.7) rectangle (0.2,0.95);
\filldraw [black,fill=white!90!black] (0.8,-0.2) rectangle (1.2,0.95);
\filldraw [black,fill=white!90!black] (0.8,-1.7) rectangle (1.2,-0.55);
\foreach \x in {(0,0),(1,0),(0,0.75),(1,0.75),(0,-0.75),(1,-0.75),(0,-1.5),(1,-1.5)}
\fill \x circle (4pt);
\draw (0,0) -- (1,0);
\draw (0,0.75) -- (1,0.75);
\draw (0,-0.75) -- (1,-0.75);
\draw (0,-1.5) -- (1,-1.5);
\begin{scope}[shift={(3,0)}]
\draw (-1,0) node {$\rightarrow$};
\foreach \x in {(0,0),(1,0),(2,0)}
\fill \x circle (4pt);
\draw (0,0) -- (2,0);
\end{scope}
\end{tikzpicture}$ & $48$\\[9mm]
\hline $(2^3\,1^2)$ & $\begin{tikzpicture}[baseline=1mm,scale=0.5]
\filldraw [black,fill=white!90!black] (-0.2,-0.2) rectangle (0.2,0.95);
\filldraw [black,fill=white!90!black] (0.8,-0.2) rectangle (1.2,0.95);
\filldraw [black,fill=white!90!black] (-0.2,-1.7) rectangle (0.2,-0.55);
\foreach \x in {(0,0),(1,0),(0,0.75),(1,0.75),(0,-0.75),(1,-0.75),(0,-1.5),(1,-1.5)}
\fill \x circle (4pt);
\draw (0,0) -- (1,0);
\draw (0,0.75) -- (1,0.75);
\draw (0,-0.75) -- (1,-0.75);
\draw (0,-1.5) -- (1,-1.5);
\begin{scope}[shift={(3,0)}]
\draw (-1,0) node {$\rightarrow$};
\foreach \x in {(0,0),(1,0),(2,0),(0,-0.75),(1,-0.75)}
\fill \x circle (4pt);
\draw (0,0) -- (2,0);
\draw (0,-0.75) -- (1,-0.75);
\end{scope}
\end{tikzpicture}$ & $48$ & $(3^2\,2)$ & $\begin{tikzpicture}[baseline=1mm,scale=0.5]
\filldraw [black,fill=white!90!black] (-0.2,-0.95) rectangle (0.2,0.95);
\filldraw [black,fill=white!90!black] (0.8,-1.7) rectangle (1.2,0.2);
\filldraw [black,fill=white!90!black] (-0.2,-1.3) -- (0.4,-1.3) -- (0.4,0.95) -- (1.2,0.95) -- (1.2,0.55) -- (0.6,0.55) -- (0.6,-1.7) -- (-0.2,-1.7) -- (-0.2,-1.3);
\foreach \x in {(0,0),(1,0),(0,0.75),(1,0.75),(0,-0.75),(1,-0.75),(0,-1.5),(1,-1.5)}
\fill \x circle (4pt);
\draw (0,0) -- (1,0);
\draw (0,0.75) -- (1,0.75);
\draw (0,-0.75) -- (1,-0.75);
\draw (0,-1.5) -- (1,-1.5);
\begin{scope}[shift={(3,0)}]
\draw (-1,0) node {$\rightarrow$};
\foreach \x in {(0,0),(1,0),(0.5,0.75)}
\fill \x circle (4pt);
\draw (0,0) -- (1,0) -- (0.5,0.75) -- (0,0);
\end{scope}
\end{tikzpicture}$ & $48$ \\[9mm]
\hline $(2^3\,1^2)$ & $\begin{tikzpicture}[baseline=1mm,scale=0.5]
\filldraw [black,fill=white!90!black] (-0.2,-0.2) rectangle (0.2,0.95);
\filldraw [black,fill=white!90!black] (0.8,-0.95) rectangle (1.2,0.2);
\filldraw [black,fill=white!90!black] (-0.2,-1.7) rectangle (0.2,-0.55);
\foreach \x in {(0,0),(1,0),(0,0.75),(1,0.75),(0,-0.75),(1,-0.75),(0,-1.5),(1,-1.5)}
\fill \x circle (4pt);
\draw (0,0) -- (1,0);
\draw (0,0.75) -- (1,0.75);
\draw (0,-0.75) -- (1,-0.75);
\draw (0,-1.5) -- (1,-1.5);
\begin{scope}[shift={(3,0)}]
\draw (-1,0) node {$\rightarrow$};
\foreach \x in {(0,0),(1,0),(2,0),(3,0),(4,0)}
\fill \x circle (4pt);
\draw (0,0) -- (4,0);
\end{scope}
\end{tikzpicture}$ & $192$ & $(4^2)$ & $\begin{tikzpicture}[baseline=1mm,scale=0.5]
\filldraw [black,fill=white!90!black] (-0.2,-1.7) rectangle (0.2,0.95);
\filldraw [black,fill=white!90!black] (0.8,-1.7) rectangle (1.2,0.95);
\foreach \x in {(0,0),(1,0),(0,0.75),(1,0.75),(0,-0.75),(1,-0.75),(0,-1.5),(1,-1.5)}
\fill \x circle (4pt);
\draw (0,0) -- (1,0);
\draw (0,0.75) -- (1,0.75);
\draw (0,-0.75) -- (1,-0.75);
\draw (0,-1.5) -- (1,-1.5);
\begin{scope}[shift={(3,0)}]
\draw (-1,0) node {$\rightarrow$};
\foreach \x in {(0,0),(1,0)}
\fill \x circle (4pt);
\draw (0,0) -- (1,0);
\end{scope}
\end{tikzpicture}$ & $8$ \\[9mm]
\hline $(2^3\,1^2)$ & $\begin{tikzpicture}[baseline=1mm,scale=0.5]
\filldraw [black,fill=white!90!black] (-0.2,-0.2) rectangle (0.2,0.95);
\filldraw [black,fill=white!90!black] (0.8,-0.95) rectangle (1.2,0.2);
\filldraw [black,fill=white!90!black] (-0.2,-0.55) -- (0.4,-0.55) -- (0.4,0.95) -- (1.2,0.95) -- (1.2,0.55) -- (0.6,0.55) -- (0.6,-0.95) -- (-0.2,-0.95) -- (-0.2,-0.55);
\foreach \x in {(0,0),(1,0),(0,0.75),(1,0.75),(0,-0.75),(1,-0.75),(0,-1.5),(1,-1.5)}
\fill \x circle (4pt);
\draw (0,0) -- (1,0);
\draw (0,0.75) -- (1,0.75);
\draw (0,-0.75) -- (1,-0.75);
\draw (0,-1.5) -- (1,-1.5);
\begin{scope}[shift={(3,0)}]
\draw (-1,0) node {$\rightarrow$};
\foreach \x in {(0,0),(1,0),(0.5,0.75),(0,-0.75),(1,-0.75)}
\fill \x circle (4pt);
\draw (0,0) -- (1,0) -- (0.5,0.75) -- (0,0);
\draw (0,-0.75) -- (1,-0.75);
\end{scope}
\end{tikzpicture}$ & $32$ & & &\\[9mm]
\hline
\end{tabular}

\vspace*{3mm}
\end{center}
\caption{List of contracted graphs when $r=4$.}\label{tab:r4}
\end{table}

\noindent By using a computer algebra system in order to keep track of all the terms, we get:
\begin{align*}
\kappa^{(4)}(S_n(\edgegraph)) &= 48\,n^{\downarrow 4}\,(t(C_4,\gamma) - p^4) + 48\,(1-4p)\,n^{\downarrow 3}\,t(C_3,\gamma) \\
&\quad + 8p\,n^{\downarrow 2}\,(p^3 (24n-54) - 12p^2(n-3) - 7p + 1);\\
\kappa^{(4)}(Y_n(\edgegraph)) &= 48\,(t(C_4,\gamma) - p^4)+O\!\left(\frac{1}{n}\right).
\end{align*}
These equations lead to the following surprising result:
\begin{proposition}\label{prop:gaussian_edge_density}
Let $\gamma$ be a singular graphon with edge density $p$. The following assertions are equivalent:
\begin{enumerate}
    \item The graphon $\gamma$ is the constant graphon $\gamma_p$.
    \item The limit in distribution $Y_\infty(\edgegraph) = \lim_{n \to \infty} n(t(\edgegraph,G_n(\gamma)) - \esper[t(\edgegraph,G_n(\gamma))])$ has a Gaussian distribution.
\item The fourth cumulant of the limit $Y_\infty(\edgegraph)$ vanishes.
\end{enumerate}
\end{proposition}

\begin{proof}
By the previous computations, the limiting fourth cumulant is proportional to the difference $t(C_4,\gamma)-p^4$, and by the Chung--Graham--Wilson criterion, this difference vanishes if and only if $\gamma=\gamma_p$. In this case, the central limit theorem due to Janson and Nowicki holds.
\end{proof}
\bigskip

\section*{Appendix: general form of the matrix-tree theorem}
In this appendix, we prove the generalisations of the matrix-tree theorem which were used during the proofs of Propositions \ref{prop:expectation_laplacian}, \ref{prop:upper_bound_alpha} and \ref{prop:upper_bound_beta}. The original form of the matrix tree theorem is due to Kirchhoff \cite{Kirch47}: it states that for any finite unoriented graph $G$, the product of the non-zero eigenvalues of the Laplacian matrix $L_G$ is the number of spanning trees $|\ST_G|$. A modern exposition of this result can be found in \cite{Biggs93}*{Chapter 7}. Recently, a combinatorial approach of this result and its generalisations has been made popular by Zeilberger \cite{Zeil85}; see also \cite{KL20} for a version where the entries of the Laplacian matrix are taken in arbitrary rings. In the sequel, we follow the Zeilberger approach in order to establish the Equations  \eqref{eq:MT1}, \eqref{eq:MT2} and \eqref{eq:MT3}.\medskip

As is well known since the works of Tutte \cite{Tut48}, the matrix tree theorems are natural in the setting of \emph{directed} graphs, so in the sequel we exceptionally consider a directed graph $G\dir = (V,E)$ with $V = \lle 1,n\rre$ and $E \subset V^2$ set of pairs $(i \to j)$. We use the index $\mathrm{d}$ in order to indicate that a graph is directed. A \emph{directed spanning tree} of $G\dir$ is a collection $T\dir$ of $n-1$ edges in $E$ which are not loops $(i \to i)$, and such that every vertex in $\lle 1,n\rre$ but one vertex $r$ is the starting point of an edge $(i \to j) \in T\dir$. For instance, if 
$$ G\dir = 
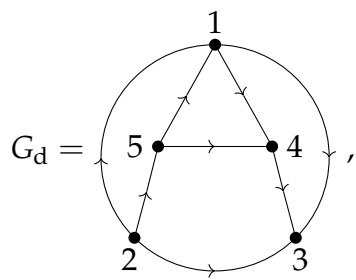
\begin{tikzpicture}[baseline=0mm,scale=1.5]
\foreach \x in {(0,1),(225:1),(315:1),(0.5,0.1),(-0.5,0.1)}
\fill \x circle (1.5pt);
\draw [->] (180:1) arc (180:0:1);
\draw (0:1) arc (0:-90:1);
\draw [->] (225:1) arc (225:270:1);
\draw [->] (225:1) arc (225:180:1);
\draw [->] (225:1) -- (-0.6,-0.3);
\draw (-0.6,-0.3) -- (-0.5,0.1);
\draw (315:1) -- (0.6,-0.3);
\draw [<-] (0.6,-0.3) -- (0.5,0.1);
\draw [->] (-0.5,0.1) -- (0,0.1);
\draw (0,0.1) -- (0.5,0.1);
\draw [->] (-0.5,0.1) -- (-0.25,0.55);
\draw [->] (-0.25,0.55) -- (0,1) -- (0.25,0.55);
\draw (0.25,0.55) -- (0.5,0.1);
\draw (0,1.2) node {$1$};
\draw (-0.75,-0.9) node {$2$};
\draw (0.75,-0.9) node {$3$};
\draw (-0.7,0.1) node {$5$};
\draw (0.7,0.1) node {$4$};
\end{tikzpicture}\,\,,$$
then a possible directed spanning tree (with root $r=3$) is
$$ T\dir = 
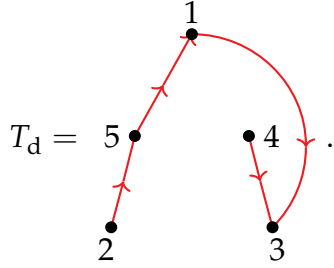
\begin{tikzpicture}[baseline=0mm,scale=1.5]
\draw [->,Red,thick] (90:1) arc (90:0:1);
\draw [Red,thick] (0:1) arc (0:-45:1);
\draw [->,Red,thick] (225:1) -- (-0.6,-0.3);
\draw [Red,thick] (-0.6,-0.3) -- (-0.5,0.1);
\draw [->,Red,thick] (0.5,0.1) -- (0.6,-0.3);
\draw [Red,thick] (0.6,-0.3) -- (315:1);
\draw [->,Red,thick] (-0.5,0.1) -- (-0.25,0.55);
\draw [->,Red,thick] (-0.25,0.55) -- (0,1);
\draw (0,1.2) node {$1$};
\draw (-0.75,-0.9) node {$2$};
\draw (0.75,-0.9) node {$3$};
\draw (-0.7,0.1) node {$5$};
\draw (0.7,0.1) node {$4$};
\foreach \x in {(0,1),(225:1),(315:1),(0.5,0.1),(-0.5,0.1)}
\fill \x circle (1.5pt);
\end{tikzpicture}\,\,.$$
More generally, we call \emph{directed forest} on $G\dir$ a collection $F\dir$ of edges in $E$ which is a disjoint union of directed spanning trees on subsets of $V$. If $k$ is the number of connected components of $F\dir$ (including the trivial components without edges), then its number of edges is $n-k$. \medskip

Given a set of commutative variables $(a_{ij})_{i, j}$ and a directed graph $G\dir=(V,E)$ with vertex set $V=\lle 1,n\rre$, we define the \emph{weight} of $G\dir$ by
$$\wt(G\dir) = \prod_{(i \to j) \in E} a_{ij}.$$
For instance, the two previous graphs yield $\wt(G\dir) = a_{13}a_{14}a_{21}a_{23}a_{25}a_{43}a_{51}a_{54}$ and $\wt(T\dir)=a_{13}a_{25}a_{43}a_{51}$. Given a \emph{cycle} $C\dir$, that is a directed graph on a set $\{i_1,i_2,\ldots,i_l\}$ with edges $(i_1 \to i_2)$, $(i_2\to i_3)$, \emph{etc.}, $(i_l \to i_1)$, the \emph{modified weight} of $C\dir$ is:
$$\mwt(C\dir) = -\wt(C\dir) = - a_{i_1i_2}a_{i_2i_3}\cdots a_{i_li_1}.$$
By convention, if $l=1$, then $\wt(C\dir)=a_{i_1i_1}$ and $\mwt(C\dir)=-a_{i_1i_1}$. More generally, given a disjoint union of cycles $C\dir = C_{1,\mathrm{d}} \sqcup C_{2,\mathrm{d}} \sqcup \cdots \sqcup C_{r,\mathrm{d}}$, we set
$$\mwt(C\dir) = \prod_{i=1}^r \mwt(C_{i,\mathrm{d}}) = (-1)^r \,\wt(C\dir).$$
Note that we have:
\begin{align*}
\det((-a_{ij})_{1\leq i,j\leq n}) &= \sum_{\sigma \in \sym(n)} (-1)^{c(\sigma)} \left(\prod_{C\dir \text{ cycle of }\sigma} \wt(C\dir)\right) \\
&=\sum_{\sigma \in \sym(n)}  \left(\prod_{C\dir \text{ cycle of }\sigma} \mwt(C\dir)\right)= \sum_{\sigma \in \sym(n)} \mwt(\sigma).
\end{align*}
More generally, if $(d_i)_{1\leq i \leq n}$ is a set of diagonal coefficients, then
\begin{align*}
\det((\delta_{ij}d_i-a_{ij})_{1\leq i,j\leq n}) &= \sum_{\sigma \in \sym(n)} \left(\prod_{\substack{C\dir \text{ cycle of }\sigma \\ \ell(C\dir)\geq 2}} \mwt(C\dir)\right)\left(\prod_{i \text{ fixed point of }\sigma} d_i-a_{ii}\right)\\
&= \sum_{\substack{V_H \subset \lle 1,n\rre \\ H\dir\text{ union of cycles on }V_H}} \mwt(H\dir)\left(\prod_{i \notin V_H} d_i\right).
\end{align*}
Given commutative variables $z_1,\ldots,z_n,\delta_1,\ldots,\delta_n$, let us consider the special case of the identity above where:
\begin{itemize}
     \item $A_{G\dir}(z)=(a_{ij})_{1\leq i,j\leq n}$ is the marked adjacency matrix of a loopless directed graph $G\dir=(V,E)$: $a_{ij}=z_iz_j$ if $(i \to j) \in E$, $a_{ij}=0$ if $(i \to j)\notin E$, and $E$ does not contain a loop $(i \to i)$, so $a_{ii}=0$;
     \item $d_i=\delta_i + \sum_{j \neq i}a_{ij}$.
 \end{itemize}  
 The matrix 
 $$L_{G\dir}(z) = \begin{pmatrix}
\sum_{j \neq 1}a_{1j} & -a_{12} & \cdots & - a_{1n} \\
 -a_{21} & \sum_{j \neq 2}a_{2j} & \ddots & \vdots \\
 \vdots & \ddots & \ddots & -a_{(n-1)n}\\
 -a_{n1} & \cdots & -a_{n(n-1)} & \sum_{j \neq n}a_{nj}
 \end{pmatrix}$$
 is the \emph{marked directed Laplacian} of $G\dir$, and we are thus interested in $\det (\Delta+L_{G\dir}(z))$ for a general diagonal matrix $\Delta$. Note that given a union of cycles $H\dir$ on a subset $V_H$ of $\lle 1,n\rre$, we have $\mwt(H\dir)=0$ if $H\dir$ contains a cycle of length $1$. Otherwise, $$\mwt(H\dir)=(-1)^{\text{number of cycles of $H\dir$}}\,\prod_{i \in V_H}(z_i)^2.$$ So,
 $$\det(\Delta+L_{G\dir}(z)) = \sum_{\substack{V_H \subset \lle 1,n\rre \\ H\dir\text{ union of cycles of length $\geq 2$ on }V_H}} (-1)^{\text{number of cycles of $H\dir$}}\left(\prod_{i \in V_H}z_i\right)^{\!2}\left(\prod_{i \notin V_H} d_i\right).$$
 The product of diagonal elements expands as follows:
 $$\prod_{i \notin V_H} d_i = \sum_{\substack{E_{K\dir} \text{ set of edges $(k_1\to k_2)$ with distinct} \\ \text{starting points $k_1$ in a set }S_K \subset \lle 1,n\rre \setminus V_H}} \left(\prod_{(k_1 \to k_2) \in E_{K\dir}} z_{k_1}z_{k_2}\right)\left(\prod_{i \notin V_H \cup S_K} \delta_i\right).$$
 So, we have an expansion
 $$\det(\Delta+L_{G\dir}(z)) = \sum_{(H\dir,K\dir)} (-1)^{\text{number of cycles of }H\dir} \left(\prod_{i \in V_H}z_i\right)^{\!2}\left(\prod_{(k_1 \to k_2) \in E_{K\dir}} z_{k_1}z_{k_2}\right)\left(\prod_{i \notin V_H \cup S_K} \delta_i\right),$$
where the pairs $(H\dir,K\dir)$ run over some set $\mathcal{G}$ of pairs of directed subgraphs of $G\dir$ that consist in disjoint edges. This set $\mathcal{G}$ of pairs of subgraphs is endowed with the following involution $\phi$: given a pair $(H\dir,K\dir)$, if at least one cycle appears in $H\dir$ or in $K\dir$, we consider the cycle which contains the smallest element, and we move it from $H\dir$ to $K\dir$ or from $K\dir$ to $H\dir$. We convene that $\phi(H\dir,K\dir) = (H\dir,K\dir)$ if $H\dir = \emptyset$ and $K\dir$ does not contain any cycle. Note that the involution $\phi$:
\begin{itemize}
     \item leaves invariant the factor $\prod_{i \notin V_H \cup S_K} \delta_i$;
     \item also leaves invariant the factor $\prod_{i \in V_H}(z_i)^2\,\prod_{(k_1 \to k_2) \in E_{K\dir}} z_{k_1}z_{k_2}$, because we are moving a cycle;
     \item changes the sign $(-1)^{\text{number of cycles of }H\dir}$ if $\phi(H\dir,K\dir) \neq (H\dir,K\dir)$.
 \end{itemize} 
Therefore, in the expansion of $\det(\Delta+L_{G\dir}(z))$, all the terms for which $H\dir \neq \emptyset$ or $K\dir$ contains a cycle cancel one another. We thus obtain:
$$\det (\Delta+L_{G\dir}(z)) = \sum_{\substack{E_{K\dir} \text{ set of edges $(k_1 \to k_2)$ with} \\ \text{distinct starting points $k_1$ in a set}\\ S_K \subset \lle 1,n\rre,\text{ and without cycle}}} \left(\prod_{(k_1 \to k_2) \in E_{K\dir}} z_{k_1}z_{k_2}\right)\left(\prod_{i \notin S_K} \delta_i\right).$$
Now, a set of directed edges of $G\dir$ without cycle and with distinct starting points is the same as a directed forest, and the $i$'s which are not starting points are exactly the roots of the trees that form the forest. Thus, we obtain:
$$\det (\Delta+L_{G\dir}(z)) = \sum_{F\dir \text{ directed forest on }G\dir} \left(\prod_{(T\dir,r) \text{ rooted tree in } F\dir}\left(\delta_r \prod_{(i \to j) \in T\dir} z_iz_j\right)\right).$$
Consider the particular case where $G=G\dir$ is not oriented: $(i,j) \in E$ if and only if $(j,i) \in E$. Then, the adjacency matrix of $G$ is symmetric, and on the other hand, given a non-oriented forest on $G$, a choice of orientation of the forest is equivalent to a choice of a root for every tree in the forest. So, for any simple unoriented graph $G = (\lle 1,n\rre,E)$ and any diagonal matrix $\Delta_n=(\delta_1,\ldots,\delta_n)$, we obtain:
\begin{equation}
\det(\Delta_n+L_G(z)) = \sum_{F \text{ forest on }G} \left(\prod_{i=1}^n (z_i)^{\deg(i,F)}\right)\left(\prod_{T \subset F} \left(\sum_{x \in T} \delta_x\right)\right). \label{eq:MT}\tag{MT}
\end{equation}

\begin{example}
Consider the non-oriented simple graph 
$$G = 
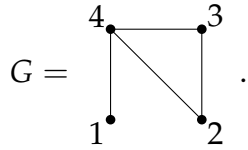
\begin{tikzpicture}[scale=1.2,baseline=5mm]
\foreach \x in {(0,0),(0,1),(1,0),(1,1)}
\fill \x circle (1.5pt);
\draw (0,0) -- (0,1) -- (1,1) -- (1,0) -- (0,1);
\draw (-0.15,-0.15) node {$1$};
\draw (1.15,-0.15) node {$2$};
\draw (1.15,1.15) node {$3$};
\draw (-0.15,1.15) node {$4$};
\end{tikzpicture}\,.$$
A computer algebra system allows one to compute:
\begin{align*}
 &\det (\Delta_4 + L_G(z))\\
  &=\delta_1\delta_2\delta_3\delta_4 + z_1z_4(\delta_1+\delta_4)\delta_2\delta_3 + z_2z_3(\delta_2+\delta_3)\delta_1\delta_4 + z_2z_4(\delta_2+\delta_4)\delta_1\delta_3 \\
 &\quad + z_3z_4(\delta_3+\delta_4)\delta_1\delta_2 + z_1z_2z_3z_4(\delta_1+\delta_4)(\delta_2+\delta_3) + z_1z_2(z_4)^2\delta_3(\delta_1+\delta_2+\delta_4) \\
 &\quad + z_1z_3(z_4)^2\delta_2(\delta_1+\delta_3+\delta_4) + z_2(z_3)^2z_4\delta_1(\delta_2+\delta_3+\delta_4) + (z_2)^2z_3z_4\delta_1(\delta_2+\delta_3+\delta_4) \\
 &\quad + z_2z_3(z_4)^2\delta_1(\delta_2+\delta_3+\delta_4) + z_1z_2(z_3)^2(z_4)^2(\delta_1+\delta_2+\delta_3+\delta_4)\\
 &\quad + z_1(z_2)^2z_3(z_4)^2(\delta_1+\delta_2+\delta_3+\delta_4) + z_1z_2z_3(z_4)^3 (\delta_1+\delta_2+\delta_3+\delta_4) ,
\end{align*}

\noindent and this polynomial is indeed the sum of the contributions of the $14$ directed forests:
\vspace*{2mm}
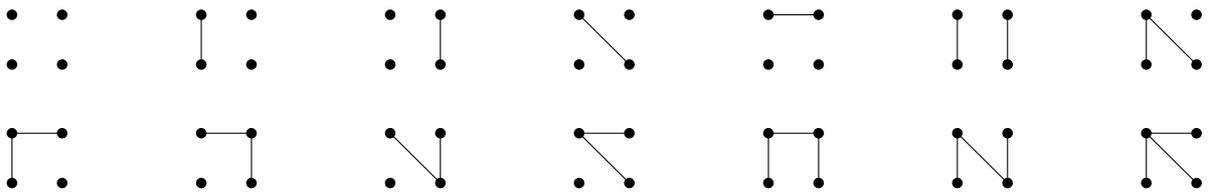
\begin{align*}
&\begin{tikzpicture}[scale=0.66,baseline=2mm]
\foreach \x in {(0,0),(0,1),(1,0),(1,1)}
\fill \x circle (3pt);
\end{tikzpicture}\qquad\qquad
\begin{tikzpicture}[scale=0.66,baseline=2mm]
\foreach \x in {(0,0),(0,1),(1,0),(1,1)}
\fill \x circle (3pt);
\draw (0,0) -- (0,1);
\end{tikzpicture}\qquad\qquad
\begin{tikzpicture}[scale=0.66,baseline=2mm]
\foreach \x in {(0,0),(0,1),(1,0),(1,1)}
\fill \x circle (3pt);
\draw (1,1) -- (1,0);
\end{tikzpicture}\qquad\qquad
\begin{tikzpicture}[scale=0.66,baseline=2mm]
\foreach \x in {(0,0),(0,1),(1,0),(1,1)}
\fill \x circle (3pt);
\draw (0,1) -- (1,0);
\end{tikzpicture}\qquad\qquad
\begin{tikzpicture}[scale=0.66,baseline=2mm]
\foreach \x in {(0,0),(0,1),(1,0),(1,1)}
\fill \x circle (3pt);
\draw (0,1) -- (1,1);
\end{tikzpicture}\qquad\qquad
\begin{tikzpicture}[scale=0.66,baseline=2mm]
\foreach \x in {(0,0),(0,1),(1,0),(1,1)}
\fill \x circle (3pt);
\draw (0,0) -- (0,1);
\draw (1,0) -- (1,1);
\end{tikzpicture}\qquad\qquad
\begin{tikzpicture}[scale=0.66,baseline=2mm]
\foreach \x in {(0,0),(0,1),(1,0),(1,1)}
\fill \x circle (3pt);
\draw (0,0) -- (0,1) -- (1,0);
\end{tikzpicture}\\
&\\
&\begin{tikzpicture}[scale=0.66,baseline=2mm]
\foreach \x in {(0,0),(0,1),(1,0),(1,1)}
\fill \x circle (3pt);
\draw (0,0) -- (0,1) -- (1,1);
\end{tikzpicture}\qquad\qquad
\begin{tikzpicture}[scale=0.66,baseline=2mm]
\foreach \x in {(0,0),(0,1),(1,0),(1,1)}
\fill \x circle (3pt);
\draw (0,1) -- (1,1) -- (1,0);
\end{tikzpicture}\qquad\qquad
\begin{tikzpicture}[scale=0.66,baseline=2mm]
\foreach \x in {(0,0),(0,1),(1,0),(1,1)}
\fill \x circle (3pt);
\draw (1,1) -- (1,0) -- (0,1);
\end{tikzpicture}\qquad\qquad
\begin{tikzpicture}[scale=0.66,baseline=2mm]
\foreach \x in {(0,0),(0,1),(1,0),(1,1)}
\fill \x circle (3pt);
\draw (1,0) -- (0,1) -- (1,1);
\end{tikzpicture}\qquad\qquad
\begin{tikzpicture}[scale=0.66,baseline=2mm]
\foreach \x in {(0,0),(0,1),(1,0),(1,1)}
\fill \x circle (3pt);
\draw (0,0) -- (0,1) -- (1,1) -- (1,0);
\end{tikzpicture}\qquad\qquad
\begin{tikzpicture}[scale=0.66,baseline=2mm]
\foreach \x in {(0,0),(0,1),(1,0),(1,1)}
\fill \x circle (3pt);
\draw (0,0) -- (0,1) -- (1,0) -- (1,1);
\end{tikzpicture}\qquad\qquad
\begin{tikzpicture}[scale=0.66,baseline=2mm]
\foreach \x in {(0,0),(0,1),(1,0),(1,1)}
\fill \x circle (3pt);
\draw (0,0) -- (0,1) -- (1,1);
\draw (0,1) -- (1,0);
\end{tikzpicture}
\end{align*}
\end{example}
\bigskip

We now examine special cases of the Formula \eqref{eq:MT}. First, if we set $z_i=1$ for any $i \in \lle 1,n\rre$, then we recover Equation \eqref{eq:MT1}. On the other hand, let us consider the case of the complete graph $K_n$ on $n$ vertices. We introduce an additional variable $z_0$, and we consider the diagonal matrix 
$$\Delta_n=\diag(z_0z_1,z_0z_2,\ldots,z_0z_n).$$
The right-hand side of Equation \eqref{eq:MT} can then be interpreted as follows. A forest $F=T_1\sqcup T_2\sqcup \cdots \sqcup T_k$ on the graph $G=K_n$ together with a choice of roots $i_1 \in T_1,i_2 \in T_2,\ldots,i_k \in T_k$ can be considered as a Cayley tree $T$ of size $n+1$ rooted at $0$: we connect the root $0$ to all the vertices $i_1,i_2,\ldots,i_k$. Note then that with $\delta_i=z_0z_i$,
 $$\left(\prod_{i=1}^n (z_i)^{\deg(i,F)}\right) \left(\prod_{j=1}^k z_0z_{i_j}\right) = \prod_{i=0}^n (z_i)^{\deg(i,T)}.$$
 So, the matrix-tree theorem reads as follows:
 $$\det(\diag(z_0z_1,\ldots,z_0z_n) + L_{K_n}(z)) = \sum_{T \text{ Cayley tree on }\lle 0,n\rre} \left(\prod_{i=0}^n (z_i)^{\deg(i,T)}\right).$$
The right-hand side is the generating series of the degrees of the Cayley trees in size $n+1$. During the proof of Proposition \ref{prop:upper_bound_alpha}, we claimed that this was equal to $(z_0+\cdots+z_n)^{n-1} z_0z_1\cdots z_n$. In order to show this, we compute explicitly the determinant:
\begin{align*}
\det &= \left|\begin{matrix}
z_1 (z_0+z_2+\cdots+z_n) &-z_1z_2 & \cdots &-z_1z_n \\
-z_2z_1 & z_2(z_0+z_1+\cdots+z_n) & \cdots &-z_2z_n \\
\vdots & & &\vdots \\
-z_nz_1 &-z_nz_2 & \cdots & z_n(z_0+z_1+\cdots+z_{n-1})
\end{matrix}\right| \\
&=z_1\cdots z_n \left|\begin{matrix}
z_0+z_2+\cdots+z_n &-z_2 & \cdots &-z_n \\
-z_1 & z_0+z_1+\cdots+z_n & \cdots &-z_n \\
\vdots & & &\vdots \\
-z_1 &-z_2 & \cdots & z_0+z_1+\cdots+z_{n-1}
\end{matrix}\right| \\
&=z_1\cdots z_n\det ((z_0+z_1+\cdots+z_n)I_n - J_n(z_1,z_2,\ldots,z_n)),
\end{align*}
where $J_n(z_1,z_2,\ldots,z_n)$ is the matrix with rank $1$ and with all its entries equal to $z_j$ on the $j$-th column. The spectrum of this matrix is $\{0^{n-1},(z_1+z_2+\cdots+z_n)^1\}$, so we get:
\begin{align*}
\det &= z_1\cdots z_n (z_0+z_1+\cdots+z_n)^{n-1} ((z_0+z_1+\cdots+z_n)-(z_1+\cdots+z_n)) \\&
= z_0z_1\cdots z_n(z_0+z_1+\cdots+z_n)^{n-1}.
\end{align*}
Hence, Equation \eqref{eq:MT3} is proved. Finally, consider again the complete graph on $n$ vertices, and set $z_1=z_2=\cdots=z_n=1$ in the general form \eqref{eq:MT} of the matrix tree theorem. We obtain:
$$\det(\Delta_n + L_{K_n}) = \sum_{k=1}^n \sum_{1\leq i_1<i_2<\cdots<i_k\leq n} \delta_{i_1}\cdots \delta_{i_k}\,C_k(i_1,\ldots,i_k),$$
where $C_k(i_1,\ldots,i_k)$ is the number of families of disjoint trees $T_1,\ldots,T_k$ such that $i_j \in T_j$ for any index $j$, and such that $T_1 \sqcup \cdots \sqcup T_k = F$ is a forest on $K_n$. Since all the vertices play a symmetric role in $K_n$, the number $C_k=C_k(i_1,\ldots,i_k)$ does not depend on the fixed roots $i_1,\ldots,i_k$. Therefore,
$$C_k \times \binom{n}{k} = \text{number of families of $k$ disjoint rooted trees covering a set of $n$ vertices}.$$
Equation \eqref{eq:MT2} states that the right-hand side is equal to $n^{n-k}\,\binom{n-1}{k-1}$. Indeed, if $\Delta_n=x\,I_n$ and $J_n = J_n(1,1,\ldots,1)$, then
$$\det(\Delta_n + L_{K_n}) = \det((x+n)I_n - J_n) = x (x+n)^{n-1} = \sum_{k=1}^n \binom{n-1}{k-1}\,n^{n-k}\,x^k,$$
We have thus recovered Formula \eqref{eq:MT2}, and on the other hand,
$$C_k = \frac{\binom{n-1}{k-1}}{\binom{n}{k}}\,n^{n-k} = k\,n^{n-k-1},$$
so we get the expansion in elementary symmetric functions:
$$\det(\diag(\delta_1,\ldots,\delta_n)+L_{K_n}) =  \sum_{k=1}^n k\,n^{n-k-1}\,e_k(\delta_1,\ldots,\delta_n).$$ 
\bigskip
\bigskip


\bibliography{singular.bib}

\end{document}